\documentclass[10pt,a4paper,francais]{smfart}

\usepackage[latin1]{inputenc}
\usepackage[francais]{babel}
\usepackage{amsopn}                    % pour définir de nouveaux opérateurs
\usepackage{amsmath} 
\usepackage{amssymb}
\usepackage[all,2cell,emtex]{xy}
\UseAllTwocells
\xyoption{2cell}
\input xypic       

\usepackage[T1]{fontenc}
\usepackage{enumitem}
\usepackage{mathtools}
\usepackage{version}

\setlist{nosep}
\setlist[enumerate]{label=\emph{(\alph*)}}

\newlength{\tinyskipamount}
\newcommand\tinyskip{\vspace\tinyskipamount}
\newcommand\tinybreak{%
\par\ifdim\lastskip <\tinyskipamount\removelastskip\penalty -500\tinyskip\fi
}
\let\tb\tinybreak
\let\ts\tinyskip

\SilentMatrices
\SelectTips{cm}{10}

\theoremstyle{plain} 
\newtheorem{thm}{Th\'eor\`eme}[section]
\newtheorem{prop}[thm]{Proposition}
\newtheorem{lemme}[thm]{Lemme}
\newtheorem{cor}[thm]{Corollaire}
\theoremstyle{remark}                                             % env. "remark" de l'AMS
\newtheorem{rem}[thm]{Remarque}

\newtheorem{exemple}[thm]{Exemple}

\newtheorem{sch}[thm]{Scholie}
\theoremstyle{definition}                                         % env. "definition" de l'AMS
\newtheorem{definition}[thm]{D\'efinition}
\newtheorem{paragr}[thm]{}
\newtheorem*{notations}{Notations et terminologie}
\newtheorem*{remerciements}{Remerciements}
\newtheorem*{organisation}{Organisation de l'article}
\theoremstyle{plain} 
\swapnumbers
\newtheorem{argptitob}[thm]{L'argument du petit objet}

\newtheorem{lemtransf}[thm]{Lemme du transfert}
\swapnumbers

\numberwithin{equation}{thm}

\theoremstyle{plain} 
\newtheorem{dprop}{Proposition}[subsection]
\theoremstyle{definition}
\newtheorem{dparagr}[dprop]{}

\theoremstyle{plain}
\newtheorem{sprop}{Proposition}[thm]
\theoremstyle{definition}
\newtheorem{sparagr}[sprop]{}

\makeatletter

\def\@begintheorem#1#2[#3]{%
  \def\@theoremhead{\normalfont\the\thm@headfont
    \@ifempty{#1}{\let\thmname\@gobble}{\let\thmname\@iden}%
    \@ifempty{#2}{\let\thmnumber\@gobble}{\let\thmnumber\@iden}%
    \@ifempty{#3}{\let\thmnote\@gobble}{\let\thmnote\@iden}%
    \thm@swap\swappedhead\thmhead{#1}{#2}{#3}}%
  \sbox\@tempboxa{\@theoremhead}%
  \ifdim\wd\@tempboxa>0.7\linewidth \smf@skippttrue\fi
  \ifsmf@skippt
    \global\smf@skipptfalse
    \item [\thm@indent]%
          {\sloppy\@theoremhead\parskip\z@\@@par}%
    \nobreak\everypar{}%
    \let\thmheadnl\relax
  \else
    \item[\hskip\labelsep\thm@indent\unhbox\@tempboxa\the\thm@headpunct]%
  \fi
  \@restorelabelsep
  \thmheadnl % possibly a newline.
  \hypertarget{\csname @currentHref\endcsname}{}
  \ignorespaces}

\makeatother

\newcommand{\gmrelax}[1]{\relax}

\newcommand{\pref}[1]{{\widehat{ #1 }}}
\newcommand{\cats}{\mathbf{\Delta}}
\newcommand{\simpl}{\pref{\cats}}
\newcommand{\cat}{{\mathcal{C} \mspace{-2.mu} \it{at}}}
\newcommand{\ens}{{\mathcal{E} \mspace{-2.mu} \it{ns}}}
\newcommand{\ord}{{\mathcal{O} \mspace{-2.mu} \it{rd}}}

\newcommand{\Hom}{\operatorname{\mathsf{Hom}}}
\newcommand{\sHom}{\operatorname{\kern.5truept\underline{\kern-.5truept\mathsf{Hom}\kern-.5truept}\kern1truept}}

\newcommand{\ob}{\operatorname{\mathsf{Ob}}}
\newcommand{\W}{\mathcal{W}}
\newcommand{\fl}{\operatorname{\mathsf{Fl}}}

\newcommand{\smp}[1]{ \varDelta_{#1}}
\newcommand{\e}{\varepsilon}

\newcommand{\im}{\operatorname{\mathrm{Im}}}

\newcommand{\card}{\mathsf{card}}

\renewcommand{\leq}{\leqslant}
\renewcommand{\geq}{\geqslant}

\newcommand{\toto}{{\hskip -2.5pt\xymatrixcolsep{1.3pc}\xymatrix{\ar[r]&}\hskip -2.5pt}}
\renewcommand{\to}{{\hskip -2.5pt\xymatrixcolsep{1pc}\xymatrix{\ar[r]&}\hskip -2.5pt}}

\newcommand{\todouble}{\xymatrixcolsep{1pc}\xymatrix{\ar@<.5ex>[r]\ar@<-.5ex>[r]&}}
\newcommand{\todoubleop}{\xymatrixcolsep{1pc}\xymatrix{\ar@<.5ex>[r]&\ar@<.5ex>[l]}}
\renewcommand{\hookrightarrow}{{\hskip -1.5pt\raise 1.5pt\vbox{\xymatrixcolsep{.9pc}\xymatrix{\ar@{^{(}->}[r]&}}\hskip -3.5pt}}
\renewcommand{\longmapsto}{{\hskip -2.5pt\xymatrixcolsep{1.3pc}\xymatrix{\ar@{|->}[r]&}\hskip -2.5pt}}
\renewcommand{\mapsto}{{\hskip -2.5pt\xymatrixcolsep{.9pc}\xymatrix{\ar@{|->}[r]&}\hskip -2.5pt}}

\def\limind{\mathop{\oalign{\rm lim\cr
\hidewidth$\longrightarrow$\hidewidth\cr}}}%
\newcommand{\cm}[2]{\mathchoice {#1\raise -1.8pt\vbox{\hbox{$\kern -.8pt/#2$}}} {#1\raise -1.8pt\vbox{\hbox{$\kern -.8pt/#2$}}\kern .8pt} {#1\raise -1.8pt\vbox{\hbox{$\scriptstyle\kern -.8pt /#2$}}} {#1\raise -1.8pt\vbox{\hbox{$\scriptscriptstyle\kern -.8pt /#2$}}} }

\newcommand{\mc}[2]{\mathchoice {\raise -1.8pt\vbox{\hbox{$#1\backslash$}}#2} {\raise -1.8pt\vbox{\hbox{$#1\backslash$}}#2} {\raise -1.8pt\vbox{\hbox{$\scriptstyle#1\backslash$}}#2} {\raise -1.8pt\vbox{\hbox{$\scriptscriptstyle#1\backslash$}}#2} }

\newcommand{\mathCat}[1]{\mathcal{#1}}
\newcommand{\mathMod}[1]{\mathcal{#1}}
\newcommand{\mathOp}[1]{\mathsf{#1}}

\newcommand{\build}[3]{\mathrel{\mathop{\kern 0pt#1}\limits_{#2}^{#3}}}

\newcommand{\comp}[2]{\build{*}{#2}{#1}}

\newcommand{\C}{\mathCat{C}}
\newcommand{\D}{\mathCat{D}}
\newcommand{\F}{\mathCat{F}}
\newcommand{\M}{\mathMod{M}}

\newcommand{\del}[1]{\partial^{}_{#1}}

\newcommand{\delet}[1]{\partial^{*}_{#1}}

\newcommand{\face}[2]{\delta^{#1}_{#2}}
\newcommand{\bord}[1]{\partial\smp{#1}}
\newcommand{\cornet}[2]{\Lambda^{#2}_{#1}}
\newcommand{\bordCS}[1]{\partial\Phi_{#1}}
\newcommand{\cornetCS}[2]{\Phi^{#2}_{#1}}

\newcommand{\Sd}{Sd\kern 1pt}
\newcommand{\Ex}{Ex}

\newcommand{\nCat}[1]{{#1}\hbox{\protect\nbd-}\kern1pt\cat}

\newcommand{\Cof}[1]{\mathOp{Cof}^{}_{#1}}
\newcommand{\cell}{\mathOp{cell}}
\newcommand{\Ornt}[1]{\mathCat{O}_{#1}}
\newcommand{\trOrnt}[2]{\mathCat{O}^{\leq #2}_{#1}}
\newcommand{\Orntg}[1]{\mathCat{O}(#1)}
\newcommand{\Orntgfonct}{\mathCat{O}}

\newcommand{\Fl}{\underline{\fl\kern-1pt}\kern1pt}
\newcommand{\FlCart}{\Fl_{\kern.8pt\mathrm{cart}}}

\newcommand{\Nbisimpl}{\mathcal{N}_\mathsf{bs}}

\newcommand{\trb}{\tau^{}_{\mathrm b}}
\newcommand{\tri}{\tau^{}_{\mathrm i}}

\newcommand{\undl}[1]{\underline{#1\kern-2pt}\kern2pt}

\newcommand{\sauf}{\mathchoice{\raise 1.8pt\hbox{${\scriptstyle\kern
2.5pt\smallsetminus\kern 2.5pt}$}}{\raise 1.8pt\hbox{${\scriptstyle\kern
2.5pt\smallsetminus\kern 2.5pt}$}}{\raise
1.8pt\hbox{${\scriptscriptstyle\kern 1.5pt\smallsetminus\kern
1.5pt}$}}{\raise 1.8pt\hbox{${\scriptscriptstyle\kern
1.5pt\smallsetminus\kern 1.5pt}$}}}

\newcommand{\Ho}{{\operatorname{\mathrm{Ho}}}}

\newcommand\zbox[1]{\makebox[0pt][l]{#1}}
\newcommand\pbox[1]{\zbox{\quad#1}}
\newcommand\id[1]{{1^{}_{#1}}}
\let\ncat\nCat
\newcommand\dcat{\ncat2}
\newcommand{\ncatnorm}[1]{\text{$#1$-$\widetilde{\cat}$}}
\newcommand\dcatnorm{\ncatnorm2}

\newcommand{\Uo}{\nu_1}
\newcommand{\Ut}{\nu_2}
\newcommand{\toz}{\smash{\hskip-2.5pt\xymatrixcolsep{1pc}\xymatrix{\ar[r]\ar@<.25ex>@{}[r]_{\scriptscriptstyle0}&}\hskip -2.5pt}}
\let\rightrightarrows\todouble

\let\oldxrightarrow\xrightarrow
\newcommand{\xrightarrowz}[1]{\oldxrightarrow[\protect{\raisebox{3pt}[0pt][0pt]{\makebox[8pt][c]{$\scriptscriptstyle0$}}}]{#1}}
\renewcommand{\xrightarrow}[1]{\oldxrightarrow[\protect{\makebox[8pt]{}}]{#1}}
\newcommand{\DeltaeX}[2]{\partial^{#2}_{#1}}

\newcommand{\Hot}{\mathsf{Hot}}

\let\eps\varepsilon
\let\wt\widetilde
\let\Deltan\smp 
\let\Deltae\del
\let\cDelta\cats

\let\Homi\sHom
\let\Ob\ob
\let\Cat\cat
\let\comp\ast
\let\nbd\nobreakdash
\let\forlang\emph
\let\ndef\emph

\newcommand\hpz{\hphantom{0}}
\newcommand\hpzz{\hphantom{00}}

\newcommand\sstrut{\vphantom{\beta_1}}

\author{Dimitri Ara}
\address{Radboud Universiteit Nijmegen\\
Institute for Mathematics, Astrophysics and Particle Physics\\
Heyendaalseweg 135\\
6525 AJ Nijmegen\\
The Netherlands}
\email{d.ara@math.ru.nl}
\urladdr{http://www.math.ru.nl/\raise -3.3pt\vbox{\hbox{$\widetilde{ \ }\,$}}dara/}

\author{Georges Maltsiniotis}
\address{Institut de Math\'ematiques de Jussieu\\
Universit\'e Paris 7 Denis Diderot\\
Case Postale 7012\\
B\^atiment Sophie Germain\\
75205 Paris Cedex 13\\
France}
\email{maltsin\at math.jussieu.fr}
\urladdr{http://www.math.jussieu.fr/\raise -3.3pt\vbox{\hbox{$\widetilde{ \ }\,$}}maltsin/}

\title[Vers une structure de catégorie de modèles à la Thomason sur
$n$-$\cat$]{Vers une structure de catégorie de modèles à~la Thomason sur la
catégorie des $n$-catégories strictes}

\alttitle{Towards a Thomason model structure on $n$-$\cat$}

\begin{document}

\frontmatter

\let\sbigskipamount\bigskipamount
\divide\bigskipamount by 3

\begin{abstract}
Le but de cet article est d'exposer quelques réflexions en vue d'obtenir une
structure de catégorie de modèles sur la catégorie des petites
$n$-catégories strictes généralisant celle obtenue par Thomason dans le cas
des catégories ordinaires. En s'inspirant d'idées de \hbox{Grothendieck} et
de Cisinski, nous obtenons un \og théorème de \hbox{Thomason} abstrait\fg{}
dont le théorème de Thomason classique est une conséquence facile. Nous en
déduisons un théorème de Thomason $2$-catégorique dont une preuve incorrecte
a été publiée par K.~Worytkiewicz, K.~Hess, P.~Parent et A.~Tonks. Pour $n >
2$, nous isolons des conditions suffisantes pour obtenir un théorème de
Thomason $n$\nbd-catégorique. L'étude de ces conditions fera l'objet de
travaux ultérieurs. 
\end{abstract}

\begin{altabstract}
The purpose of this article is to present ideas towards obtaining a model
category structure on the category of small strict $n$-categories,
generalizing the one obtained by Thomason on ordinary categories. Following
ideas of Grothendieck and Cisinski, we obtain an ``abstract Thomason
theorem'', which easily implies the classical Thomason theorem. We deduce a
$2$-categorical Thomason theorem, an incorrect proof of which has been
published by K.~Worytkiewicz, K.~Hess, P.~Parent and A.~Tonks. For $n > 2$,
we isolate sufficient conditions to obtain an $n$-categorical Thomason
theorem. These conditions will be investigated in further work.
\end{altabstract}


\maketitle

\tableofcontents

\mainmatter

\let\bigskipamount\sbigskipamount

\section*{Introduction}

La théorie de l'homotopie dans $\Cat$, la catégorie des petites catégories,
débute avec l'introduction par Grothendieck en 1961 du foncteur
nerf~\cite{Nerf} associant à une petite catégorie un ensemble simplicial.
Ce foncteur nerf permet de définir une notion d'équivalence faible dans
$\Cat$ en décrétant qu'un foncteur entre petites catégories est une
équivalence faible si son nerf est une équivalence d'homotopie faible
d'ensembles simpliciaux. Dans sa thèse~\cite{IL}, Illusie démontre
(attribuant ce résultat à Quillen) que le foncteur nerf induit une
équivalence de catégories de la catégorie homotopique des petites catégories
$\Ho(\Cat)$ vers la catégorie homotopique $\Hot = \Ho(\simpl)$ des ensembles
simpliciaux (catégorie homotopique signifiant ici catégorie localisée par
les équivalences faibles). En particulier, grâce à un théorème de Milnor
\cite{Mil}, les petites catégories modélisent les types d'homotopie des
espaces topologiques.

\tb

En 1980, Thomason démontre l'existence d'une structure de catégorie de
modèles de Quillen sur $\Cat$~\cite{Th} dont les équivalences faibles sont
les équivalences faibles définies \forlang{via} le nerf. Ceci entraîne par
exemple que la théorie homotopique de $\cat$ est homotopiquement complète
et cocomplète~\cite{Ci1}. De plus, Thomason définit une équivalence de 
Quillen entre $\Cat$ munie de cette structure et les
ensembles simpliciaux munies de la structure définie par Quillen
dans~\cite{Qu1}. Ceci montre que $\Ho(\Cat)$ et $\Hot$ sont non seulement
équivalentes comme catégories, mais aussi comme $(\infty, 1)$-catégories
(par exemple au sens des catégories simpliciales~\cite{DK}).

\tb

Une des premières applications spectaculaires de la théorie de l'homotopie
dans $\Cat$ est sans doute la définition par Quillen de la $\mathrm
K$\nbd-théorie algébrique~\cite{QuK}. En effet, les groupes de $\mathrm
K$\nbd-théorie algébrique sont définis dans \cite{QuK} comme les groupes
d'homotopie du nerf d'une certaine catégorie. De plus, afin d'établir les
théorèmes de base de cette théorie, Quillen est amené à démontrer ses fameux
théorèmes A et B qui jouent un rôle crucial dans la théorie de l'homotopie
dans $\Cat$.

\tb

Cette théorie a également été largement explorée par
Grothendieck qui lui consacre une grande partie de \emph{Pursuing
Stacks}~\cite{PS1} et des \emph{Dérivateurs}~\cite{Der} (voir
aussi~\cite{MalPRST}). L'idée de Grothendieck est de fonder la théorie
de l'homotopie des espaces sur la théorie de l'homotopie dans $\Cat$. C'est
ce point de vue qui lui permet de développer la théorie des catégories test
(voir~\cite{PS1} et~\cite{MalPRST}). Cette étude est poursuivie
par Cisinski dans sa thèse~\cite{Ci}. Cette thèse contient notamment une
variante plus conceptuelle de la preuve du théorème de Thomason qui est un
des points de départ du présent article.

\tb

Pour fonder une théorie de l'homotopie analogue dans $\nCat{n}$, $1 \le n
\le \infty$, la catégorie des petites $n$-catégories strictes et des
$n$-foncteurs stricts, on a besoin d'un foncteur nerf associant, par
exemple, à une $n$-catégorie (stricte) un ensemble simplicial. En 1987,
Street a construit un tel foncteur pour $n = \infty$~\cite{S1} (voir
aussi~\cite{S2,S2cor,Steiner,SteinerOr}). Pour $n \ge 1$, la restriction de
ce foncteur à $\nCat{n}$ définit un foncteur nerf de $\nCat{n}$ vers la
catégorie des ensembles simpliciaux que nous appellerons $n$-nerf de Street.
Pour $n=1$, on retrouve le nerf introduit par Grothendieck. On dispose ainsi
d'une notion d'équivalence faible dans $\ncat{n}$ : nous appellerons
équivalence faible de \hbox{Thomason} un $n$-foncteur qui s'envoie sur une
équivalence d'homotopie faible d'ensembles simpliciaux \emph{via} le
$n$-nerf de Street.

\tb

On dispose également d'autres foncteurs nerf : pour $n$ fini, un nerf
$n$-simplicial (voir par exemple~\cite{Tamsamani}) à valeurs dans les
ensembles $n$-simpliciaux et, pour $n$ arbitraire, un nerf cellulaire
(introduit par M.~A.~Batanin et R.~Street \cite{BataninStreet} et étudié
indépendamment par M.~Makkai et M.~Zawadowski~\cite{MakZaw} et
C.~Berger~\cite{CB1}) à valeurs dans les préfaisceaux sur la catégorie
$\Theta_n$ de Joyal \cite{Joyal}. Une question fondamentale de la théorie de
l'homotopie dans $\nCat{n}$ est la comparaison de ces différents foncteurs
nerf pour montrer, en particulier, qu'ils induisent la même notion
d'équivalence faible de $n$-catégories. Nous montrerons dans \cite{DG2} que
les équivalences faibles définies \forlang{via} le nerf $n$-simplicial
coïncident avec celles définies \forlang{via} le nerf $n$-cellulaire et nous
étudierons le problème de la comparaison de ces deux foncteurs nerf avec le
$n$-nerf de Street.

\tb

Le cas particulier de la théorie de l'homotopie dans $\dcat$ est activement
étudié par A.~M.~Cegarra et ses collaborateurs. Dans~\cite{BCeg},
M.~Bullejos et A.~M.~\hbox{Cegarra} démontrent que les équivalences faibles
définies \forlang{via} le $2$-nerf de Street coïncident avec celles définies
\forlang{via} le nerf bisimplicial. Ils prouvent par ailleurs un analogue
$2$\nbd-catégorique du théorème~A de Quillen. Dans~\cite{Ceg}, A.~M.~Cegarra
généralise le théorème~B de Quillen et étudie les colimites homotopiques
dans ce cadre. Dans sa thèse~\cite{ChicheThese}, J.~Chiche développe une
théorie de l'homotopie dans $\dcat$ analogue à celle développée par
Grothendieck dans $\cat$. La thèse~\cite{Hoyo} de M.~L.~del Hoyo (qui précède
celle de J.~Chiche) contient une partie importante consacrée à ces
questions. En particulier, en utilisant des résultats de cette thèse, on
peut montrer que l'inclusion de~$\cat$ dans $\dcat$ induit une équivalence
entre les catégories homotopiques. Une preuve complète de ce résultat a été
obtenue par J.~Chiche dans~\cite{Chiche}. On en déduit immédiatement que le
$2$-nerf de Street induit une équivalence de catégories entre $\Ho(\dcat)$
et~$\Hot$.

\tb

Il est naturel de se demander s'il existe une structure de catégorie de
modèles sur $\nCat{n}$ ayant comme équivalences faibles les équivalences
faibles de Thomason, et si celle-ci est Quillen équivalente à celle des
ensembles simpliciaux. Pour $n=2$, K.~\hbox{Worytkiewicz} et~\emph{al}.
affirment dans~\cite{Wor} apporter une réponse positive à ces questions.
Malheureusement, bien que la définition qu'ils proposent pour cette
structure soit correcte, leurs démonstrations des deux points clés
permettant d'adapter à la dimension $2$ la preuve de Cisinski sont erronées
(voir les scholies \ref{trivfaux} et \ref{sch:rdf_wor} du présent article).
Par ailleurs, contrairement à ce qui est affirmé dans l'introduction, leur
texte ne contient pas de preuve du fait que la structure prétendument
construite est Quillen équivalente à celle des ensembles simpliciaux (on
peut déduire facilement ce fait du résultat postérieur de J.~Chiche).

\medbreak

Le présent article est le premier d'une série d'articles qui a pour but
d'obtenir une structure de catégorie de modèles à la Thomason sur
$\ncat{n}$. Dans ce premier article, nous démontrons, en suivant les grandes
lignes de la preuve de Cisinski du théorème de Thomason classique, un \og
théorème de Thomason abstrait\fg{} (théorème~\ref{Thomabs}). Étant donnés
une catégorie $\C$ et un \og foncteur nerf\fg{} $U$ de $\C$ vers la
catégorie des ensembles simpliciaux admettant un adjoint à gauche $F$, on
présente des conditions suffisantes permettant de définir une structure de
catégorie de modèles \og à la Thomason\fg{} sur~$\C$, ayant comme
équivalences faibles (resp. comme fibrations) les flèches de $\C$ dont
l'image par $U$ (resp. par $\Ex^2U$) est une équivalence faible (resp. une
fibration) d'ensembles simpliciaux (où $\Ex$ désigne le foncteur d'extension
de Kan, adjoint à droite du foncteur de subdivision $\Sd$~\cite{Kan}), de
sorte que $(F\Sd^2,\Ex^2U)$ soit une adjonction de Quillen.

\tb

La vérification de ces conditions pour $\C=\cat$ et $U$ le foncteur nerf
usuel est très simple et on retrouve donc le théorème de Thomason
classique.

\tb

Pour $\C = \ncat{n}$ et $U = N_n$ le $n$-nerf de Street, nous démontrons que la
plupart des conditions du théorème de Thomason abstrait sont satisfaites. En
particulier, nous montrons que pour obtenir une structure de catégorie de
modèles à la Thomason sur $\ncat{n}$, il suffit, en notant $c_n$ l'adjoint à
gauche de $N_n$, de prouver les deux propriétés suivantes :
\tb
\begin{enumerate}
  \item[($\text{\emph{d}}'$)] si $E'\to E$ est un crible d'ensembles
    ordonnés admettant une rétraction qui est aussi un adjoint à droite,
    son image par le foncteur $c_nN$ est une équivalence faible de Thomason
    et le reste après tout cochangement de base ;
  \item[(\emph{e})] pour tout ensemble ordonné $E$, le morphisme d'adjonction
  $N(E)\to N_nc_nN(E)$ est une équivalence faible d'ensembles
  simpliciaux.
\end{enumerate}
\tb
La partie la plus technique de l'article est consacrée à la démonstration de
ces conditions dans le cas $n = 2$, corrigeant en particulier les erreurs de
\cite{Wor}. Nous obtenons ainsi une structure de catégorie de modèles à la
Thomason sur $\dcat$. De plus, on déduit d'un théorème de J.~Chiche déjà cité
que cette structure est Quillen équivalente à la structure de catégorie de
modèles sur les ensembles simpliciaux définie par Quillen. En particulier,
la $(\infty, 1)$\nbd-catégorie $\Ho(\dcat)$ est équivalente à la $(\infty,
1)$\nbd-catégorie $\Hot$. Nous montrons par ailleurs que le foncteur
d'inclusion de $\Cat$ munie de la structure de Thomason dans $\nCat{2}$
munie de la structure à la Thomason est une équivalence de Quillen à droite.

\tb

Comme l'ont remarqué les auteurs de \cite{Wor}, si $i : E' \to E$ est un
crible d'ensembles ordonnés vérifiant les hypothèses de la
condition~($\text{\emph{d}}'$), alors le $2$-foncteur 
\hbox{$c_2N(i) : c_2N(E') \to c_2N(E)$} 
est un crible (en un sens adéquat) admettant une rétraction $r$
pour laquelle il existe une homotopie oplax (avec nos conventions, 
lax avec celles de~\cite{Wor}) normalisée $H$ de $ri$ vers
l'identité de $c_2N(E)$, c'est-à-dire un $2$\nbd-foncteur oplax normalisé $H
: \Delta_1 \times c_2N(E) \to c_2N(E)$ compatible en un sens évident à ces
deux $2$-foncteurs. Il est facile de voir qu'un tel crible est une
équivalence faible de Thomason. Pour démontrer la
condition~($\text{\emph{d}}'$), il suffirait donc de s'assurer que cette
notion est stable par images directes.  Les auteurs de \cite{Wor} affirment
que c'est bien le cas. La preuve qu'ils produisent contient cependant une
erreur. De fait, nous doutons de la véracité de cette assertion (voir le
scholie~\ref{sch:rdf_wor} pour plus de détails). Dans la
section~\ref{sec:2-cat}, nous introduisons une condition supplémentaire sur
un tel crible et nous montrons que $c_2N(i)$, où $i$ est comme ci-dessus,
satisfait à cette condition. La délicate section~\ref{sec:2rdf} est
consacrée à la preuve du fait que cette notion est stable par images
directes.

\tb

Les auteurs de \cite{Wor} affirment par ailleurs que le morphisme de la
condition~(\emph{e}) est un isomorphisme. L'ensemble totalement ordonné à
trois éléments fournit pourtant un contre-exemple à cette assertion (voir le
scholie~\ref{trivfaux}). Dans la section~\ref{sec:2-cat}, nous expliquons
comment déduire la condition~(\emph{e}) du résultat de comparaison du
$2$-nerf de Street et du nerf bisimplicial obtenu par M.~Bullejos et
A.~M.~Cegarra dans \cite{BCeg}.

\tb

Pour $n > 2$, les propriétés ($\text{\emph{d}}'$) et (\emph{e}) ne sont pas
encore établies.  Afin de généraliser notre démonstration de la
condition~($\text{\emph{d}}'$) à $\ncat{n}$, $n > 2$, nous aurons besoin
d'une notion de $n$-foncteur lax (ou oplax) normalisé pour $n > 2$. Ce
concept ne semble pas avoir été encore exploré. Dans \cite{DG1}, nous
introduirons et étudierons une telle notion. Par ailleurs, pour établir la
condition $(\emph{e})$ pour $n$ quelconque, il faudra comparer le $n$-nerf de
Street et le nerf $n$-simplicial. Ce sera l'objet de \cite{DG2}.  Une fois
ces deux conditions démontrées, on aura donc une structure de catégorie de
modèles à la Thomason sur $\nCat{n}$, et il sera facile de vérifier (en
généralisant la preuve de la proposition~\ref{eq-Quillen-incl}) que, pour
$0<m<n$, le couple de foncteurs formé de l'inclusion $\nCat{m}\to\nCat{n}$ et
de son adjoint à gauche forment une adjonction de Quillen. Une
généralisation du théorème de J.~Chiche permettra d'affirmer que cette
adjonction est une équivalence de Quillen, et en particulier que les
$n$\nbd-catégories strictes modélisent les types d'homotopie. Enfin, des
arguments simples (analogues à ceux conduisant au
corollaire~\ref{bieqeqThom}) montreront que les cofibrations de la structure
à la Thomason sur $\nCat{n}$ sont des cofibrations de la structure \og
folklorique\fg{} \cite{LMW} et que les équivalences faibles de cette
dernière structure sont des équivalences faibles de Thomason.

\tb

Il est à noter qu'une généralisation du théorème de Thomason dans une
direction différente a été obtenue par T.~M.~Fiore et S.~Paoli~\cite{FP} qui
construisent une structure de catégorie de modèles à la Thomason sur la
catégorie des petites catégories $n$\nbd-uples, Quillen équivalente à
celle des ensembles simpliciaux. 

\medbreak

\begin{organisation}
La plupart des résultats de la section~\ref{cofGroth} sont dus à Grothendieck.
Dans une catégorie de modèles de Quillen, les équivalences faibles ne
déterminent pas les cofibrations. L'idée de Grothendieck est d'associer à
toute classe de flèches $\W$ d'une catégorie, considérées comme des
équivalences faibles, une notion de $\W$-cofibrations, flèches ayant les
propriétés formelles des cofibrations d'une catégorie de modèles de Quillen
propre à gauche. La section~\ref{cofGroth} est consacrée à l'étude de cette
notion, ainsi que d'une variante, dans le cas où $\W$ est l'image inverse
par un foncteur de la classe des équivalences faibles d'une catégorie de
modèles.

\tb

Dans la section~\ref{cribles}, on introduit le formalisme des catégories à
cribles et cocribles, version abstraite des cribles et cocribles dans
$\cat$. On présente l'exemple, crucial pour cet article, des cribles et
cocribles dans $\nCat{n}$. La section~\ref{sec:transfert} est consacrée à
des rappels sur les catégories de modèles combinatoires et le lemme de
transfert classique pour ces structures. On en déduit un nouveau théorème de
transfert. Dans la section~\ref{absThom}, on démontre le \og théorème de
Thomason abstrait\fg{}, ce qui termine la partie abstraite de l'article.

\tb

Dans la section~\ref{nCat}, on présente des conditions suffisantes sur un
foncteur $\cats\to\nCat{n}$ pour que le couple de foncteurs adjoints qui
lui est associé par le procédé de Kan satisfasse à certaines hypothèses du
\og théorème de Thomason abstrait\fg{}. On introduit les orientaux de
Street, ainsi que leur variante $n$\nbd-tronquée. On prouve que le foncteur
correspondant $i^{}_n:\cats\to\nCat{n}$ satisfait à ces conditions
suffisantes. Dans la section~\ref{sec:2-cat}, on s'intéresse au cas de
$\dcat$.  On démontre, modulo un résultat $2$-catégorique établi dans la
section suivante, que le couple de foncteurs adjoints associé au foncteur
$i^{}_2$ satisfait aux hypothèses restantes du \og théorème de Thomason
abstrait\fg{}, achevant ainsi la preuve du théorème de Thomason
$2$\nbd-catégorique. La section~\ref{sec:2rdf} est consacrée à la propriété
$2$-catégorique admise dans la section précédente, à savoir la stabilité par
images directes de la notion de $2$-crible rétracte par déformation oplax
fort. Dans la dernière section, on démontre quelques propriétés des objets
cofibrants pour la structure à la Thomason sur $\nCat{2}$, généralisant à
$\nCat{2}$ un théorème de Thomason sur $\Cat$.

\tb

Enfin, on présente dans un appendice quelques applications d'un théorème de
Smith~\cite{Be}, faisant intervenir la notion de $\W$\nbd-cofibrations de
Grothendieck.
\end{organisation}

\medbreak

\begin{remerciements}
Le second auteur souhaiterait remercier sincèrement Andy Tonks et
\hbox{Krzysztof} Worytkiewicz pour les exposés qu'ils ont donnés sur leur
article~\cite{Wor} dans deux groupes de travail de l'université Paris~7 et
pour les longues discussions qui s'en sont suivies. Ce sont ces exposés et
discussions qui l'ont poussé à s'intéresser à la question de la
généralisation au cadre $n$-catégorique. (Ce n'est que plus tard, en
étudiant en détail \cite{Wor} en vue de cette généralisation
$n$\nbd-catégorique, que les auteurs du présent article ont découvert les
erreurs de ce texte.) Les auteurs remercient par ailleurs Jonathan Chiche
pour les éléments de théorie de l'homotopie des $2$-catégories qu'il leur a
appris, ainsi que pour ses remarques et corrections.
\end{remerciements}

\medbreak

\begin{notations}
Si $A$ est une catégorie, on note $\ob(A)$ l'ensemble de ses objets et
$\fl(A)$ celui de ses flèches. Pour tout couple d'objets $x,y$ de $A$, on
note $\Hom_A(x,y)$ l'ensemble des flèches de $A$ de source $x$ et but $y$.
Si $A$ est une petite catégorie, $\pref{A}$~désigne la catégorie des
préfaisceaux d'ensembles sur $A$, foncteurs contravariants de $A$ vers la
catégorie $\ens$ des ensembles. Si $u:A\to B$ est un foncteur et $b$ un
objet de $B$, on note, quand aucune confusion n'en résulte, $\cm{A}{b}$
(resp.  $\mc{b}{A}$) la comma catégorie $(u\downarrow b)$
(resp.~$(b\downarrow u)$) dont les objets sont les couples $(a,g)$, où $a$
est un objet de $A$ et $g:u(a)\to b$ (resp.~$g:b\to u(a)$) une flèche de
$B$. Toutes les $n$\nbd-catégories considérées dans cet article seront
strictes, et on écrira en général \og $n$\nbd-catégorie \fg{} pour
\og $n$\nbd-catégorie stricte \fg{}. De même, sauf mention expresse du
contraire, les $n$\nbd-foncteurs seront stricts, et on
écrira souvent \og $n$\nbd-foncteur \fg{} pour \og $n$\nbd-foncteur strict
\fg{}. Si $C$ est une $n$\nbd-catégorie et $x,y$ sont deux objets de $C$, on
note $\sHom_C(x,y)$ la $(n-1)$\nbd-catégorie des $1$\nbd-flèches de $C$ de
$x$ vers $y$ (l'indice $C$ étant parfois omis, quand aucune ambiguïté n'en résulte).
On dit que deux $i$\nbd-flèches, $1\leq i\leq n$, de $C$ sont parallèles si elles
ont même source et même but.
\end{notations}

\section{$\W$-cofibrations de Grothendieck}\label{cofGroth}

\begin{paragr}\label{defcofGroth}
Soient $\C$ une catégorie admettant des sommes amalgamées, et $\W$ une classe de flèches de $\C$. Une \emph{$\W$\nbd-cofibration de Grothendieck}, ou plus simplement \emph{$\W$\nbd-cofibration}, est une flèche $i:A\to B$ de $\C$ telle que pour tout diagramme de carrés cocartésiens
$$
\xymatrix{
&A\ar[r]\ar[d]_{i}
&A'\ar[r]^{s}\ar[d]
&A''\ar[d]
\\
&B\ar[r]
&B'\ar[r]_{t}
&B''
&\hskip -20pt,\hskip 20pt
}
$$
si $s$ est dans $\W$, il en est de même de $t$. On note $\Cof{\W}$ la classe des $\W$\nbd-cofibrations de $\C$. La classe $\Cof{\W}$ est stable par composition et images directes. Si la classe $\W$ est stable par isomorphisme (dans la catégorie des flèches de $\C$), alors $\Cof{\W}$ contient les isomorphismes.
\end{paragr}

\begin{prop}\label{prprcofGroth}
Si la catégorie $\C$ admet des petites limites inductives filtrantes et si la classe $\W$ est stable par ces limites, alors la classe $\Cof{\W}$ partage la même propriété de stabilité. En particulier, $\Cof{\W}$ est alors stable par rétractes et par composition transfinie. 
\end{prop}

\begin{proof}
La proposition résulte de la commutation des limites inductives entre elles, et du fait que la stabilité par limites inductives filtrantes implique celle par rétractes ainsi que celle par composition transfinie (grâce à la stabilité par composition binaire).
\end{proof}

\begin{paragr}
On rappelle que la classe $\W$ \emph{satisfait à la propriété du deux sur trois} si pour tout triangle commutatif de $\C$, si deux des flèches le composant sont dans $\W$, il en est de même de la troisième.
\end{paragr}

\begin{prop}[Grothendieck]\label{lemmeGroth}
Si la classe $\W$ satisfait à la propriété du deux sur trois, et si toute flèche $f$ de $\C$ se décompose en $f=si$ avec $i\in\Cof{\W}$ et $s\in\W$, alors la classe $\Cof{\W}\cap\W$ est stable par images directes.
\end{prop}

\begin{proof}
Soit 
$$
\xymatrix{
A\ar[r]^f\ar[d]_j
&A''\ar[d]^{j''}
\\
B\ar[r]
&B''
}
$$
un carré cocartésien, avec $j$ dans $\Cof{\W}\cap\W$. Il s'agit de montrer que $j''$ est aussi dans $\Cof{\W}\cap\W$. Comme les $\W$\nbd-cofibrations sont stables par images directes, il suffit de montrer que $j''$ est dans $\W$. Par hypothèse, il existe une décomposition $f=si$ de $f$, avec $i\in\Cof{\W}$ et $s\in\W$. On en déduit un diagramme de carrés cocartésiens
$$
\xymatrix{
&A\ar[r]^i\ar[d]_j
&A'\ar[r]^s\ar[d]^{j'}
&A''\ar[d]^{j''}
\\
&B\ar[r]
&B'\ar[r]_t
&B''
&\kern -15pt.\kern15pt
}
$$
Comme $j$ est une $\W$\nbd-cofibration et $s$ est dans $\W$, $t$ est aussi dans $\W$. Comme $i$ est une $\W$\nbd-cofibration et $j$ est dans $\W$, $j'$ est aussi dans $\W$. La propriété du deux sur trois implique alors que $j''$ est dans $\W$.
\end{proof}

\begin{exemple}\label{excofGroth}
Soient $\M$ une catégorie de modèles fermée et $\W$ sa classe d'équivalences faibles. Si $\M$ est propre à gauche, alors les cofibrations de $\M$ sont des $\W$\nbd-cofibrations. Plus précisément, la catégorie de modèles $\M$ est propre à gauche si et seulement si toute cofibration de $\M$ est une $\W$\nbd-cofibration.
\end{exemple}

\begin{paragr}\label{UeqUcof}
Soient $\C$ une catégorie, $\M$ une catégorie de modèles fermée, et $U:\C\to\M$ un foncteur. Une \emph{$U$\nbd-équivalence} est une flèche $s$ de $\C$ telle que $U(s)$ soit une équivalence faible de $\M$. Une \emph{$U$\nbd-cofibration} est une flèche $i:A\to B$ de $\C$ telle que pour tout carré cocartésien de $\C$
$$
\xymatrix{
&A\ar[r]^f\ar[d]_i
&A'\ar[d]^{i'}
\\
&B\ar[r]_g
&B'
&\kern -15pt,\kern15pt
}
$$
le carré
$$\xymatrix{
&U(A)\ar[r]^{U(f)}\ar[d]_{U(i)}
&U(A')\ar[d]^{U(i')}
\\
&U(B)\ar[r]_{U(g)}
&U(B')
&\kern -15pt\kern15pt
}
$$
soit un carré homotopiquement cocartésien de $\M$. On vérifie immédiatement que les $U$\nbd-cofibrations sont stables par composition et images directes, et que tout isomorphisme est une $U$\nbd-cofibration. Plus généralement, toute flèche de $\C$ qui est une $U$\nbd-équivalence et le reste après tout cochangement de base est une $U$\nbd-cofibration.
\end{paragr}

\begin{prop}\label{UcofWcof}
Soient $\C$ une catégorie admettant des sommes amalgamées, $\M$ une catégorie de modèles fermée, $U:\C\to\M$ un foncteur, et $\W$ la classe des $U$\nbd-équivalences. Alors toute $U$\nbd-cofibration est une $\W$\nbd-cofibration. Réciproquement, si toute flèche $f$ de $\C$ se décompose en $f=si$, où $i$ est une $U$\nbd-cofibration et $s$ une $U$\nbd-équivalence, alors toute $\W$\nbd-cofibration est une $U$\nbd-cofibration.
\end{prop}

\begin{proof}
Soient $i:A\to B$ une $U$\nbd-cofibration, et 
$$
\xymatrix{
&A\ar[r]\ar[d]_{i}
&A'\ar[r]^{s}\ar[d]
&A''\ar[d]
\\
&B\ar[r]
&B'\ar[r]_{t}
&B''
&\hskip -20pt\hskip 20pt
}
$$
un diagramme formé de carrés cocartésiens, avec $s$ dans $\W$. Alors le diagramme
$$
\xymatrix{
&U(A)\ar[r]\ar[d]_{U(i)}
&U(A')\ar[r]^{U(s)}\ar[d]
&U(A'')\ar[d]
\\
&U(B)\ar[r]
&U(B')\ar[r]_{U(t)}
&U(B'')
&\hskip -20pt\hskip 20pt
}
$$
est formé de carrés homotopiquement cocartésiens, et comme $U(s)$ est une équivalence faible, il en est de même de $U(t)$, donc $t$ est une $U$\nbd-équivalence. Réciproquement, supposons que toute flèche de $\C$ se décompose en une $U$\nbd-cofibration suivie d'une $U$\nbd-équivalence, et soit
$$
\xymatrix{
A\ar[r]^f\ar[d]_j
&A''\ar[d]^{j''}
\\
B\ar[r]
&B''
}
$$
un carré cocartésien, avec $j$ une $\W$-cofibration. Par hypothèse, il existe une décomposition $f=si$ de $f$, avec $i$ une $U$\nbd-cofibration et $s$ une $U$\nbd-équivalence. On en déduit un diagramme de carrés cocartésiens
$$
\xymatrix{
&A\ar[r]^i\ar[d]_j
&A'\ar[r]^s\ar[d]^{j'}
&A''\ar[d]^{j''}
\\
&B\ar[r]
&B'\ar[r]_t
&B''
&\kern -15pt,\kern15pt
}
$$
et comme $j$ est une $\W$\nbd-cofibration, $t$ est une $U$\nbd-équivalence. Considérons l'image par~$U$ de ce diagramme dans $\M$:
$$
\xymatrix{
&U(A)\ar[r]^{U(i)}\ar[d]_{U(j)}
&U(A')\ar[r]^{U(s)}\ar[d]^{U(j')}
&U(A'')\ar[d]^{U(j'')}
\\
&U(B)\ar[r]
&U(B')\ar[r]_{U(t)}
&U(B'')
&\kern -15pt.\kern15pt
}
$$
Comme $i$ est une $U$\nbd-cofibration, le carré de gauche est homotopiquement cocartésien. Comme $U(s)$ et $U(t)$ sont des équivalence faibles, il en est de même du carré de droite, donc aussi du carré composé. On en déduit que $j$ est une $U$\nbd-cofibration.
\end{proof}

\section{Catégories à cribles et cocribles}\label{cribles}

\begin{paragr}
Soit $\C$ une catégorie. On dit qu'une flèche $f:X\to Y$ de $\C$ est \emph{quarrable} si pour tout morphisme $g:Y'\to Y$ de $\C$, il existe un carré cartésien de la forme 
$$
\xymatrix{
&X'\ar[r]\ar[d]
&X\ar[d]^f
\\
&Y'\ar[r]_g
&Y
&\hskip -15pt,\hskip 15pt
}$$
autrement dit, si le produit fibré $Y'\times_YX$ est représentable dans $\C$. Dualement, on dit que $f$ est \emph{coquarrable} si elle est quarrable comme flèche de la catégorie opposée à~$\C$.
\end{paragr}

\begin{paragr}
Soit $\C$ une catégorie admettant un objet final $e$. On rappelle qu'un \emph{segment} de $\C$ est un triplet $(I,\del{0},\del{1})$, où $I$ est un objet de $\C$ et $\del{0},\del{1}:e\to I$ sont des morphismes de $\C$ de source l'objet final. Si $\C$ admet aussi un objet initial $\varnothing$, on dit que le segment $(I,\del{0},\del{1})$ est \emph{séparant} si le carré
$$
\xymatrix{
\varnothing\ar[r]\ar[d]
&e\ar[d]^{\del{0}}
\\
e\ar[r]_{\del{1}}
&I
}$$
est cartésien.
\end{paragr}

\begin{paragr}
Soit $\C$ une catégorie admettant un objet final $e$ et un objet initial \emph{strict}~$\varnothing$ (on dit qu'un objet initial est strict si tout morphisme de but cet objet est un isomorphisme). Un \emph{segment criblant} de $\C$ est un segment séparant $(I,\del{0},\del{1})$ de $\C$ satisfaisant aux propriétés suivantes:
\ts
\begin{enumerate}[label=SCR\arabic*), ref=SCR\arabic*, wide]
\item\label{item:SCR1}
le morphisme $\del{\e}$, $\e=0,1$, est quarrable et le foncteur 
$$\delet{\e}:\cm{\C}{I}\toto\cm{\C}{e}\simeq\C\ ,\quad \chi:X\to I\ \mapsto\ \chi^{-1}(\e)\ ,$$
où $\chi^{-1}(\e)$ est le produit fibré $X\times_Ie$ défini par le carré cartésien
$$
\xymatrix{
&\chi^{-1}(\e)\ar[r]\ar[d]_{i^{\chi}_{\e}}
&e\ar[d]^{\del{\e}}
\\
&X\ar[r]_{\chi}
&I
&\hskip -15pt,\hskip 15pt
}$$
transforme carrés cocartésiens en carrés cocartésiens;
\tb
\item\label{item:SCR2}
pour toute flèche $\chi:X\to I$ et $\e=0,1$, le morphisme $i^{\chi}_{\e}$ ci-dessus est coquarrable;
\tb
\item\label{item:SCR3}
l'application $\chi\mapsto i^{\chi}_{\e}$, $\e=0,1$, définit une injection de $\Hom_\C(X,I)$ dans l'ensemble des classes d'isomorphisme de sous-objets de $X$.
\end{enumerate}
\tb
On dit qu'un morphisme $X'\to X$ de $\C$ est un \emph{crible} (resp. un
\emph{cocrible}) s'il est l'image inverse de $\del{0}$ (resp. de $\del{1}$)
par une flèche $\chi:X\to I$ de $\C$. En vertu de la
propriété~\eqref{item:SCR3}, une telle flèche $\chi$ est unique; on dit
qu'elle est la flèche \emph{définissant} le crible (resp. le cocrible)
$X'\to X$. En particulier, tout crible (resp.~cocrible) définit un cocrible
(resp. un crible), image inverse de $\del{1}$ (resp. de $\del{0}$) par
$\chi$, appelé \emph{cocrible complémentaire} (resp. \emph{crible
complémentaire}). Les cribles et les cocribles sont stables par images
inverses. La condition~\eqref{item:SCR2} affirme que les cribles et les
cocribles sont des morphismes coquarrables de $\C$.  \tb

Comme le segment $(I,\del{0},\del{1})$ est séparant et $\varnothing$ un objet initial strict, si $X'\to X$ est un crible et $X''\to X$ le cocrible complémentaire, alors il résulte de la commutation des limites projectives entre elles que le carré
$$\xymatrix{
\varnothing\ar[r]\ar[d]
&X''\ar[d]
\\
X'\ar[r]
&X
}$$
est cartésien.
\tb

Une \emph{catégorie à cribles et cocribles} est une catégorie $\C$ admettant un objet final et un objet initial strict, \emph{munie} d'un segment criblant. On dit alors que ce segment criblant \emph{définit la structure de catégorie à cribles et cocribles} de $\C$.
\end{paragr}

\begin{prop}\label{imdircribles}
Les cribles \emph{(resp.} les cocribles\emph{)} sont stables par images directes. Autrement dit, pour tout carré cocartésien de $\C$
$$
\raise 20pt
\vbox{
\xymatrix{
&A\ar[r]^a\ar[d]_i
&A'\ar[d]^{i'}
\\
&B\ar[r]_b
&B'
&\hskip -15pt,\hskip 15pt
}
}
\leqno(\star)
$$
si le morphisme $i$ est un crible \emph{(resp.} un cocrible\emph{)}, il en est de même du morphisme $i'$. De plus, le carré $(\star)$ est alors aussi cartésien et le morphisme $b$ induit un isomorphisme du cocrible \emph{(resp.} du crible\emph{)} complémentaire à $i$ avec celui complémentaire à $i'$.
\end{prop}

\begin{proof}
\newcommand\dstar{\star\mspace{1mu}\star}
Considérons un carré cocartésien $(\star)$ et supposons que $i$ soit un crible (resp. un cocrible) de sorte que l'on ait un carré cartésien
$$
\xymatrix{
&A\ar[r]^p\ar[d]_i
&e\ar[d]^{\del{\e}}
\\
&B\ar[r]_{\chi}
&I
&\hskip -15pt,\hskip 15pt
}$$
avec $\e=0$ (resp. $\e=1$). Si l'on note $p'$ l'unique flèche $A'\to e$, comme on a les égalités $\del{\e}p'a=\del{\e}p=\chi i$, la propriété universelle des sommes amalgamées implique l'existence d'une unique flèche $\chi':B'\to I$ telle que $\chi'b=\chi$ et $\chi'i'=\del{\e}p'$. 
$$
\raise 29pt
\vbox{
\xymatrixrowsep{.8pc}
\xymatrixcolsep{2.6pc}
\xymatrix{
A\ar[r]^a\ar[dd]_i
&A'\ar[dd]_{i'}\ar[rd]^{p'}
\\
&&e\ar[dd]^{\del{\e}}
\\
B\ar[r]^b\ar@<-1ex>[rrd]_{\chi}
&B'\ar@{-->}[rd]^{\chi'}
\\
&&I
}}\leqno(\dstar)$$
En vertu de la condition \eqref{item:SCR1}, l'image inverse du carré cocartésien $(\star)$ par $\del{\e}$
$$
\xymatrix{
A\ar[r]\ar[d]
&A'\ar[d]
\\
\chi^{-1}(\e)\ar[r]
&\chi'^{-1}(\e)
}
$$
est un carré cocartésien, et comme la flèche verticale de gauche est, par hypothèse, un isomorphisme il en est de même de celle de droite, ce qui prouve que $i'$ est un crible (resp. un cocrible). On en déduit que le carré oblique du diagramme $(\dstar)$ est cartésien, et comme le carré composé l'est aussi, il en est de même du carré de gauche, autrement dit, le carré $(\star)$ est aussi cartésien. À nouveau, en vertu de la condition~\eqref{item:SCR1}, l'image inverse du carré cocartésien $(\star)$ par $\del{\e'}$, où $\e'=1-\e$, est cocartésien, et comme le segment $(I,\del{0},\del{1})$ est séparant et l'objet initial de $\C$ strict, ce dernier carré cocartésien  est de la forme
$$
\xymatrix{
&\varnothing\ar[r]\ar[d]
&\varnothing\ar[d]
\\
&\chi^{-1}(\e')\ar[r]
&\chi'^{-1}(\e')
&\hskip -15pt.\hskip 15pt
}$$
On en déduit que le morphisme $\chi^{-1}(\e')\to\chi'^{-1}(\e')$ est un isomorphisme, ce qui achève la démonstration.
\end{proof}

\begin{lemme}
Soient $\chi:X\to I$ un morphisme de $\C$, $p$ l'unique morphisme de $X$ vers $e$, $\e=0,1$ et $\e'=1-\e$. Pour que le carré
$$
\raise 20pt
\vbox{
\xymatrix{
\varnothing\ar[r]\ar[d]
&e\ar[d]^{\del{\e}}
\\
X\ar[r]_{\chi}
&I
}
}
\leqno(\star)
$$
soit cartésien, il faut et il suffit que $\chi=\del{\e'}p$. 
\end{lemme}

\begin{proof}
On a un diagramme de carrés cartésiens
$$\xymatrix{
\varnothing\ar[r]\ar[d]
&\varnothing\ar[r]\ar[d]
&e\ar[d]^{\del{\e}}
\\
X\ar[r]_p
&e\ar[r]_{\del{\e'}}
&I
}$$
(le carré de gauche étant cartésien puisque l'objet initial $\varnothing$ est strict, et celui de droite puisque le segment $(I,\del{0},\del{1})$ est
séparant). Si $\chi=\del{\e'}p\,$, le carré $(\star)$ est donc cartésien. La réciproque résulte alors de \eqref{item:SCR3}.
\end{proof}

\begin{prop}\label{propefficace}
Soient $i^{}_1:U^{}_1\to X$ et $i^{}_2:U^{}_2\to X$ deux cribles de même but, et $j^{}_1:F^{}_1\to X$ et $j^{}_2:F^{}_2\to X$ les cocribles complémentaires. Les conditions suivantes sont équivalentes:
\begin{enumerate}
\item le crible $i^{}_1:U^{}_1\to X$ se factorise par le cocrible $j^{}_2:F^{}_2\to X$;
\item le crible $i^{}_2:U^{}_2\to X$ se factorise par le cocrible $j^{}_1:F^{}_1\to X$;
\item $U^{}_1\times_XU^{}_2\simeq\varnothing$.
\end{enumerate}
\end{prop}

\begin{proof}
Soient $\chi^{}_1:X\to I$ et $\chi^{}_2:X\to I$ les morphismes définissant les cribles $i^{}_1:U^{}_1\to X$ et $i^{}_2:U^{}_2\to X$, de sorte qu'on ait des carrés cartésiens
$$
\raise 38pt
\vbox{
\xymatrix{
U^{}_1\ar[d]_{i^{}_1}\ar[r]^{p^{}_1}
&e\ar[d]^{\del{0}}
\\
X\ar[r]_{\chi^{}_1}
&I
}
},\kern 13pt
\raise 38pt
\vbox{
\xymatrix{
U^{}_2\ar[d]_{i^{}_2}\ar[r]^{p^{}_2}
&e\ar[d]^{\del{0}}
\\
X\ar[r]_{\chi^{}_2}
&I
}
},\kern 13pt
\raise 38pt
\vbox{
\xymatrix{
F^{}_1\ar[d]_{j^{}_1}\ar[r]^{q^{}_1}
&e\ar[d]^{\del{1}}
\\
X\ar[r]_{\chi^{}_1}
&I
}
},\kern 13pt
\raise 38pt
\vbox{
\xymatrix{
F^{}_2\ar[d]_{j^{}_2}\ar[r]^{q^{}_2}
&e\ar[d]^{\del{1}}
\\
X\ar[r]_{\chi^{}_2}
&I
}
}.
$$
Pour que le crible $i^{}_1:U^{}_1\to X$ se factorise par le cocrible $j^{}_2:F^{}_2\to X$, il faut et il suffit que $\chi^{}_2i^{}_1=\del{1}p^{}_1$,
$$
\xymatrixrowsep{.4pc}
\xymatrixcolsep{.6pc}
\xymatrix{
U^{}_1\ar@/^1em/[rrrrrdd]^{p^{}_1}\ar@/_1em/[rrddddd]_{i^{}_1}\ar@{-->}[rrdd]
\\
\\
&&F^{}_2\ar[rrr]_{q^{}_2}\ar[ddd]^{j^{}_2}
&&&e\ar[ddd]^{\del{1}}
\\
\\
\\
&&X\ar[rrr]_{\chi^{}_2}
&&&I
}
$$
autrement dit, en vertu du lemme précédent que le carré
$$
\raise 22pt
\vbox{
\xymatrixcolsep{2.6pc}
\xymatrix{
\varnothing\ar[r]\ar[d]
&e\ar[d]^{\del{0}}
\\
U^{}_1\ar[r]_{\chi^{}_2i^{}_1}
&I
}
}\leqno(\star)
$$
soit cartésien. Or, dans le diagramme
$$
\xymatrix{
&\varnothing\ar[r]\ar[d]
&U^{}_2\ar[r]^{p^{}_2}\ar[d]_{i^{}_2}
&e\ar[d]^{\del{0}}
\\
&U^{}_1\ar[r]_{i^{}_1}
&X\ar[r]_{\chi^{}_2}
&I
&\hskip -15pt,\hskip 15pt
}
$$
le carré de droite est cartésien. Ainsi, le carré de gauche est cartésien si et seulement le carré composé l'est. Ce carré composé étant le carré $(\star)$, cela prouve l'équivalence des conditions (\emph{a}) et (\emph{c}). L'équivalence des conditions (\emph{b}) et (\emph{c}) s'en déduit par symétrie, ce qui achève la démonstration.
\end{proof}

\begin{exemple}\label{nCatcribles}
On note $\nCat{n}$ la catégorie des petites $n$\nbd-catégories strictes et $n$\nbd-foncteurs stricts, $1\leq n\leq\infty$, et $e$ la $n$\nbd-catégorie ponctuelle. Soient $\smp{1}$ la catégorie librement engendrée par le graphe $\{0\to1\}$, considérée comme $n$\nbd-catégorie (les $i$\nbd-flèches étant des identités pour $i>1$) et $\del{\e}:e\to\smp{1}$ le $n$\nbd-foncteur défini par l'objet $\e$ de $\smp{1}$, $\e=0,1$. Alors $(\smp{1},\del{0},\del{1})$ est un segment criblant de $\nCat{n}$, faisant de $\nCat{n}$ une catégorie à cribles et cocribles. En effet, la catégorie $\nCat{n}$ étant complète et cocomplète, ayant un objet initial strict, et la condition \eqref{item:SCR3} étant immédiate, la seule chose à vérifier est la propriété d'exactitude des foncteurs $\delet{\e}$, affirmée dans la condition \eqref{item:SCR1}, qui résulte du lemme suivant.
\end{exemple}

\begin{lemme}\label{deladj}
Le foncteur $\delet{\e}:\cm{\nCat{n}}{\smp{1}}\to\cm{\nCat{n}}{e}\simeq\nCat{n}$, $\e=0,1$, admet un adjoint à droite.
\end{lemme}

\begin{proof}
Montrons l'assertion pour $\e=0$. À toute $n$\nbd-catégorie $\C$, on associe une $n$\nbd-catégorie $\C^{\star}$, obtenue par adjonction formelle d'un objet final, définie par
$$
\begin{aligned}
&\ob\C^{\star}=\ob\C\,\amalg\,\{\star\}\\
\noalign{\vskip 3pt}
&\sHom^{}_{\,\C^{\star}}(x,y)=\left\{
\begin{aligned}
&\hbox{$\sHom^{}_{\,\C}(x,y)$ si $x,y\in\ob\C\,,$}\\
&\hbox{$(n-1)$\nbd-catégorie ponctuelle si $y=\star\,,$}\\
&\hbox{$(n-1)$\nbd-catégorie vide sinon\,,}
\end{aligned}
\right.
\end{aligned}
$$
les compositions étant définies de façon évidente. On remarque qu'on a un $n$\nbd-foncteur $\C^{\star}\to\smp{1}$
$$x\mapsto \left\{
\begin{aligned}
&0\,,\kern 12pt x\in\ob\C\\
&1\,,\kern 12pt x=\star
\end{aligned}
\right.$$
définissant ainsi un objet de $\cm{\nCat{n}}{\smp{1}}$. Cette construction est fonctorielle, et on vérifie facilement que le foncteur $\nCat{n}\to\cm{\nCat{n}}{\smp{1}}$ ainsi défini est un adjoint à droite du foncteur $\delet{0}$. L'assertion pour $\delet{1}$ en résulte, en utilisant l'automorphisme de $\nCat{n}$ de passage à la $n$\nbd-catégorie $1$\nbd-opposée, obtenue en inversant les $1$\nbd-flèches.
\end{proof}

\begin{paragr} Dans la suite de cet article, on considérera toujours la catégorie $\nCat{n}$ des petites $n$\nbd-catégories, munie de sa structure de catégorie à cribles et cocribles définie par le segment criblant $(\smp{1},\del{0},\del{1})$.
\end{paragr}

\begin{prop}\label{carcriblesncat}
Si $C'\to C$ est un crible \emph{(resp.} un cocrible\emph{)} de $\nCat{n}$ alors $C'$ s'identifie à une sous\nbd-$n$\nbd-catégorie de $C$. Si $C'$ est une sous\nbd-$n$\nbd-catégorie de $C$, les conditions suivantes sont équivalentes:

\begin{enumerate}
\item l'inclusion $C'\to C$ est un crible \emph{(resp.} un cocrible\emph{);}
\item pour tout $i$, $1\leq i\leq n$, toute $i$\nbd-flèche de $C$ dont l'objet but itéré \emph{(resp.} source itérée\emph{)} est dans $C'$ est elle-même dans $C'$\emph{;}
\item toute $n$\nbd-flèche de $C$ dont l'objet but itéré \emph{(resp.} source itérée\emph{)} est dans $C'$ est elle-même dans $C'$.
\end{enumerate}
De plus, si ces conditions sont satisfaites, le cocrible \emph{(resp.} le crible\emph{)} complémentaire s'identifie à l'inclusion $C''\to C$, où $C''$ est la sous\nbd-$n$\nbd-catégorie de $C$ dont les objets sont les objets de $C$ qui ne sont pas dans $C'$, et dont les $i$\nbd-flèches, $i\geq1$, sont les $i$\nbd-flèches de $C$ dont l'objet source itérée \emph{(resp.} but itéré\emph{)} n'est pas dans $C'$.
\end{prop}

\begin{proof}
La première assertion résulte aussitôt du fait que $\del{0}$ et $\del{1}$ sont des monomorphismes de $\nCat{n}$. Soit donc $C'$ une sous\nbd-$n$\nbd-catégorie de $C$ et supposons, par exemple, que $C'\to C$ soit un crible. Il existe alors un $n$\nbd-foncteur $\chi:C\to\smp{1}$ tel que $C'$ soit la fibre de $\chi$ au-dessus de $0$. Soit $\alpha$ une $i$\nbd-flèche de $C$ dont l'objet but itéré est dans $C'$. On en déduit que l'objet but itéré de $\chi(\alpha)$ est $0$, et par suite que $\chi(\alpha)$ est une identité itérée de $0$. La $i$\nbd-flèche $\alpha$ de $C$ est donc dans la fibre de~$\chi$ au-dessus de~$0$, autrement dit dans $C'$, ce qui prouve l'implication (\emph{a})~$\Rightarrow$~(\emph{b}). Réciproquement, supposons que la condition~(\emph{b}) soit satisfaite. On définit alors un $n$\nbd-foncteur $\chi:C\to\smp{1}$ comme suit. Si $x$ est un objet de $C$, on pose $\chi(x)=0$ ou~$1$ suivant que $x$ est ou n'est pas dans $C'$. Soit $\alpha$ une $i$\nbd-flèche d'objet source itérée $x$ (resp.~d'objet but itéré $y$). Si $y$ est dans $C'$ (et donc $x$ aussi), $\chi(\alpha)$ est la $i$\nbd-flèche identité itérée de $0$. Si $x$ n'est pas dans $C'$ (et donc $y$ non plus), $\chi(\alpha)$ est la $i$\nbd-flèche identité itérée de $1$. Si $x$ est dans $C'$ et $y$ ne l'est pas, alors $\chi(\alpha)$ est la flèche $0\to1$, si $i=1$, ou la $i$\nbd-flèche identité itérée de cette dernière si $i>1$. On vérifie facilement qu'on définit ainsi un $n$\nbd-foncteur, et la condition~(\emph{b}) implique aussitôt que $C'$ s'identifie à la fibre de $\chi$ au-dessus de $0$, ce qui prouve l'implication (\emph{b})~$\Rightarrow$~(\emph{a}). L'implication (\emph{b})~$\Rightarrow$~(\emph{c}) est évidente, et la réciproque résulte de l'application de la condition (\emph{c}) aux $n$\nbd-flèches identités itérées des $i$\nbd-flèches considérées dans la condition (\emph{b}). La dernière assertion résulte facilement de ce qui précède.
\end{proof}

\begin{definition}\label{defefficace} 
Une \emph{catégorie à cribles et cocribles efficaces} est une catégorie à cribles et cocribles telle que pour tout couple de cribles $U^{}_1\to X$ et $U^{}_2\to X$ de même but, tel que $U^{}_1\times_XU^{}_2\simeq\varnothing$, le carré cartésien
$$
\xymatrix{
&F^{}_1\times_XF^{}_2\ar[r]\ar[d]
&F^{}_1\ar[d]
\\
&F^{}_2\ar[r]
&X
&\hskip -15pt,\hskip15pt
}
$$
où $F^{}_1\to X$ et $F^{}_2\to X$ désignent les cocribles complémentaires à $U^{}_1\to X$ et $U^{}_2\to X$ respectivement, est aussi cocartésien.
\end{definition}

\begin{prop}\label{nCatcribleseff}
La catégorie à cribles et cocribles $\nCat{n}$ \emph{(\emph{cf.}~exemple~\ref{nCatcribles})} est une catégorie à cribles et cocribles efficaces.
\end{prop}

\begin{proof}
Soient $C$ une petite $n$\nbd-catégorie, $U^{}_1\to C$ et $U^{}_2\to C$ deux cribles, $U^{}_1$ et $U^{}_2$  étant identifiés à des sous\nbd-$n$\nbd-catégories de $C$, et $F^{}_1\to C$ et $F^{}_2\to C$ les cocribles complémentaires. En vertu de la proposition~\ref{carcriblesncat}, $F^{}_1$ (resp. $F^{}_2$) s'identifie à la sous\nbd-$n$\nbd-catégorie de $C$ dont les objets sont les objets de $C$ qui ne sont pas dans $U^{}_1$ (resp. dans $U^{}_2$), et dont les $i$\nbd-flèches, $i\geq1$, sont les $i$\nbd-flèches de $C$ dont l'objet source itérée n'est pas dans $U^{}_1$ (resp. dans $U^{}_2$). On suppose que $U^{}_1\times_CU^{}_2\simeq\varnothing$, ce qui signifie que les ensembles des objets de $U^{}_1$ et $U^{}_2$ sont disjoints, ou encore que tout objet de $C$ est un objet de $F^{}_1$ ou de $F^{}_2$. Si $\alpha$ est une $i$\nbd-flèche de $C$, $i\geq1$, son objet source itérée est donc dans $F^{}_1$ ou $F^{}_2$, et en vertu de la proposition~\ref{carcriblesncat}, elle est elle-même dans~$F^{}_1$ ou $F^{}_2$. Soient $D$ une petite $n$\nbd-catégorie, et $H^{}_1:F^{}_1\to D$ et $H^{}_2:F^{}_2\to D$ deux $n$\nbd-foncteurs coïncidant sur $F^{}_1\times_CF^{}_2$, intersection des sous\nbd-$n$\nbd-catégories $F^{}_1$ et $F^{}_2$. Il s'agit de montrer qu'il existe un unique $n$\nbd-foncteur $H:C\to D$ tel que  $H\,|\,F^{}_1=H^{}_1$ et $H\,|\,F^{}_2=H^{}_2$. Comme en vertu de ce qui précède, toute $i$\nbd-cellule de $C$, $i\geq0$, est dans $F^{}_1$ ou $F^{}_2$, l'unicité est évidente. Si $x$ est une telle $i$\nbd-cellule on définit $H(x)$ par $H(x)=H^{}_1(x)$ ou $H(x)=H^{}_2(x)$ selon que $x$ est dans $F^{}_1$ ou dans $F^{}_2$, et l'hypothèse que $H^{}_1$ et $H^{}_2$ coïncident sur l'intersection des sous\nbd-$n$\nbd-catégories $F^{}_1$ et $F^{}_2$ implique que $H$ est bien défini comme application. La compatibilité de $H$ aux sources, buts et identités résulte des propriétés correspondantes des foncteurs $H^{}_1$ et $H^{}_2$. Il reste à prouver la compatibilité aux compositions. Soient $i,j$ deux entiers, $0\leq j<i$, et $\alpha,\beta$ deux $i$\nbd-flèches de $C$. On suppose que la $j$\nbd-cellule but itéré de $\alpha$ est égale à la $j$\nbd-cellule source itérée de $\beta$. Cela implique que l'objet source itérée de $\beta$ est égal à l'objet but ou source itéré de $\alpha$ selon que $j=0$ ou $j>0$. Dans les deux cas cela implique, en vertu de la proposition~\ref{carcriblesncat}, que $\alpha$ et $\beta$ sont toute deux dans $F^{}_1$ ou dans $F^{}_2$. La compatibilité de $H$ aux compositions résulte donc de celle de $H^{}_1$ ou de~$H^{}_2$.
\end{proof}

\begin{prop}\label{lemmeDwyer}
Soient $\C$ une catégorie à cribles et cocribles efficaces, $\M$ une catégorie de modèles fermée, et $U:\C\to\M$ un foncteur. On suppose que l'image par~$U$ de tout carré cocartésien de $\C$ formé de cocribles est un carré homotopiquement cocartésien de $\M$. Soit $i:A\to B$ un crible. Si $i$ se décompose en $i=jk$, $k$ étant une $U$\nbd-cofibration et $j$ un cocrible, alors $i$ est une $U$\nbd-cofibration.
\end{prop}

\begin{proof}
Il s'agit de montrer que l'image par $U$ d'un carré cocartésien de~$\C$ de la forme 
% underfull vbox
$$
\raise 20pt
\vbox{
\xymatrix@R=1.9pc@C=1.9pc{
A\ar[d]_i\ar[r]^a
&A'\ar[d]^{i'}
\\
B\ar[r]_b
&B'
}}\leqno(\star)$$
est un carré homotopiquement cocartésien de $\M$. Formons le diagramme de carrés cocartésiens
$$
\xymatrix@R=1.7pc@C=1.7pc{
&A\ar[r]^a\ar[d]^k\ar@/^-1.3pc/[dd]_i
&A'\ar[d]_{k'\kern -3pt}\ar@/^1.3pc/[dd]^{i'}
\\
&W\ar[r]^w\ar[d]^j
&W'\ar[d]_{j'\kern -3pt}
\\
&B\ar[r]_b
&B'
&\hskip -15pt.\hskip 15pt
}
$$
Comme $k$ est une $U$-cofibration, l'image par $U$ du carré du haut est un carré homotopiquement cocartésien de $\M$, il suffit donc de prouver qu'il en est de même de l'image de celui du bas. En vertu de la proposition~\ref{imdircribles}, le carré du bas est aussi cartésien. Cette même proposition, appliquée au carré cocartésien $(\star)$, implique que si $l:V\to B$ désigne le cocrible complémentaire au crible $i:A\to B$, le composé $bl:V\to B'$ s'identifie au cocrible complémentaire au crible $i':A'\to B'$. Formons le diagramme de carrés cartésiens
$$
\xymatrix{
&V\times_BW\ar[r]\ar[d]
&W\ar[r]^w\ar[d]_j
&W'\ar[d]^{j'}
\\
&V\ar[r]_l
&B\ar[r]_b
&B'
&\hskip -15pt.\hskip 15pt
}
$$
On remarque que le carré de gauche, ainsi que le carré composé sont formés de cocribles. Comme le crible $i$ (complémentaire du cocrible $l$) se factorise par le cocrible~$j$, la proposition~\ref{propefficace} et l'hypothèse que $\C$ est une catégorie à cribles et cocribles efficaces impliquent que le carré de gauche est également cocartésien. On en déduit que l'image par $U$ de chacun de ces deux carrés est un carré homotopiquement cartésien. Il en résulte qu'il en est de même de l'image du carré de droite, ce qui achève la démonstration.
\end{proof}

\section{Le théorème de transfert}\label{sec:transfert}

\begin{paragr}\label{acc}
Soit $\kappa$ un cardinal. On rappelle qu'un ensemble \emph{$\kappa$\nbd-petit} est une ensemble dont le cardinal est strictement plus petit que $\kappa$. Une petite catégorie est \emph{$\kappa$\nbd-petite} si son ensemble des flèches est $\kappa$\nbd-petit. Une limite inductive est \emph{$\kappa$\nbd-petite} si elle est indexée par une catégorie $\kappa$\nbd-petite. On dit que le cardinal $\kappa$ est \emph{régulier} si la sous-catégorie pleine de $\ens$ formée des ensembles $\kappa$\nbd-petits est stable par limites inductives $\kappa$\nbd-petites. 
\tb

Soit $\kappa$ un cardinal régulier. Une catégorie est \emph{$\kappa$\nbd-filtrante} si tout foncteur de but cette catégorie et de source une catégorie $\kappa$\nbd-petite s'insère dans un cône inductif, autrement dit, est source d'un morphisme de foncteurs de but un foncteur constant. Une limite inductive est \emph{$\kappa$\nbd-filtrante} si elle est indexée par une catégorie $\kappa$\nbd-filtrante. Soit $\C$ une catégorie admettant des petites limites inductives $\kappa$\nbd-filtrantes. On dit qu'un objet $X$ de $\C$ est \emph{$\kappa$\nbd-présentable} si le foncteur $\Hom_{\C}(X,?):\C\to\ens$ commute aux petites limites inductives $\kappa$\nbd-filtrantes. L'objet $X$ est de \emph{présentation finie} s'il est $\aleph_0$\nbd-présentable, autrement dit, si le foncteur $\Hom_{\C}(X,?):\C\to\ens$ commute aux petites limites inductives filtrantes. On dit que la catégorie $\C$ est \emph{$\kappa$\nbd-accessible} si elle admet une petite sous-catégorie pleine $\C_0$ formée d'objets $\kappa$-présentables telle que tout objet de $\C$ soit limite inductive $\kappa$\nbd-filtrante d'objets de $\C_0$.
\tb

On dit qu'un objet d'une catégorie $\C$ est \emph{présentable} s'il existe un cardinal régulier~$\kappa$ tel que $X$ soit $\kappa$\nbd-présentable. On dit qu'une catégorie $\C$ est \emph{accessible} s'il existe un cardinal régulier $\kappa$ tel que $\C$ soit $\kappa$\nbd-accessible. Tout objet d'une catégorie accessible est présentable~\cite[remarque~2.2~(3)]{AR}. Une catégorie \emph{localement présentable} est une catégorie accessible et cocomplète.
\end{paragr}

\begin{paragr}
Soit $\C$ une catégorie. On rappelle que si $i:A\to B$ et $p:X\to Y$ sont deux flèches de $\C$, on dit que $i$ a la \emph{propriété de relèvement à gauche} relativement à $p$, ou que $p$ a la \emph{propriété de relèvement à droite} relativement à $i$, si tout carré commutatif de la forme 
$$
\xymatrix{
A\ar[r]^a\ar[d]_i
&X\ar[d]^p
\\
B\ar[r]_b
&Y
}
$$
admet un \emph{relèvement}, autrement dit s'il existe une flèche $h:B\to X$ de $\C$ telle que
$ph=b$ et $hi=a$. Si $\F$ désigne une classe de flèches de $\C$, on note $l(\F)$ (resp. $r(\F)$) la classe des flèches de $\C$ ayant la propriété de relèvement à gauche (resp. à droite) relativement à toutes les flèches de $\C$ appartenant à $\F$. La classe $l(\F)$ est stable par images directes, composés transfinis et rétractes. On note $\cell(\F)$ la classe des flèches de $\C$ qui sont des composés transfinis d'images directes de flèches appartenant à $\F$; on a donc $\cell(\F)\subset l(\F)$.
\end{paragr}

\begin{argptitob}\label{petitobjet}
Soient $\C$ une catégorie localement pré\-sen\-table, et $I$ un (petit) ensemble de flèches de $\C$. Alors toute flèche $f$ de $\C$ se décompose en $f=pi$, avec $i\in\cell(I)$ et $p\in r(I)$. De plus, $lr(I)$ est alors la classe des rétractes des flèches de $\C$ appartenant à $\cell(I)$.
\end{argptitob}

\begin{proof}
Cela résulte de \cite[proposition 10.5.16 et corollaire 10.5.23]{Hir} ou de~\cite[th\'eor\`eme 2.1.14 et corollaire 2.1.15]{Ho}.
\end{proof}

\begin{paragr}
Une \emph{catégorie de modèles combinatoire} est une catégorie de modèles fermée dont la catégorie sous-jacente est localement présentable et qui est à engendrement cofibrant. Cette dernière condition signifie alors simplement qu'il existe des (petits) ensembles $I$ et $J$ de flèches de $\M$ tels que $lr(I)$ (resp. $lr(J)$) soit la classe des cofibrations (resp. des cofibrations triviales) de $\M$, et on dit que la catégorie de modèles~$\M$ est \emph{engendrée} par $(I,J)$.
\end{paragr}

\begin{lemtransf}\label{transfert}
Soient $\C$ une catégorie localement présentable et $\M$~une catégorie de
modèles fermée à engendrement cofibrant, engendrée par $(I,J)$,
$$F:\M\toto\C\ ,\quad U:\C\toto\M$$
un couple de foncteurs adjoints, et $\W$ la classe des $U$-équivalences de
$\C$ \emph{(\emph{cf.}~\ref{UeqUcof})}. On suppose que
$$lr(F(J))\subset\W\ .$$
Alors il existe une structure de catégorie de modèles combinatoire sur $\C$ engendrée par $(F(I),F(J))$ dont les équivalences faibles sont les $U$-équivalences et les fibrations les flèches de $\C$ dont l'image par $U$ est une fibration de $\M$. De plus, si la catégorie de modèles $\M$ est propre à droite, il en est de même de $\C$.
\end{lemtransf}

\begin{proof}
Les assertions autres que la propreté à droite résultent par exemple de~\cite[proposition 1.4.23]{Ci}, de~\cite[théorème 3.3]{Cr} ou de~\cite[théorème~11.3.2]{Hir}. Celle concernant la propreté à droite résulte aussitôt du fait que le foncteur $U$ respecte les fibrations et les carrés cartésiens, et reflète les équivalences faibles.
\end{proof}

\begin{prop}\label{sansSmith}
Soient $\C$ une catégorie localement présentable, $\M$ une catégorie de modèles fermée à engendrement cofibrant, engendrée par $(I,J)$, dont les équivalences faibles sont stables par petites limites inductives filtrantes,
$$F:\M\toto\C\ ,\quad U:\C\toto\M$$
un couple de foncteurs adjoints tel que $U$ commute aux petites limites
inductives filtrantes, et $\W$ la classe des $U$-équivalences de $\C$
\emph{(\emph{cf.}~\ref{UeqUcof})}. On suppose que
\begin{enumerate}
\item $F(I)\subset\Cof{\W}$;
\item $F(J)\subset\W$.
\end{enumerate}
Alors il existe une structure de catégorie de modèles combinatoire sur $\C$, engendrée par $(F(I),F(J))$, dont les équivalences faibles sont les $U$-équivalences et les fibrations les flèches de $\C$ dont l'image par $U$ est une fibration de $\M$. De plus, cette catégorie de modèles est propre à gauche. Elle est aussi propre à droite si $\M$ l'est. 
\end{prop}

\begin{proof}
On observe d'abord que si $f\in r(F(I))$, on a par adjonction $U(f)\in r(I)$, autrement dit, $U(f)$ est une fibration triviale de $\M$, et en particulier une équivalence faible. On en déduit que $f$ est une $U$\nbd-équivalence. On a donc $r(F(I))\subset\W$.
\tb

D'autre part, comme la classe des équivalences faibles de $\M$ est stable par petites limites inductives filtrantes, il en est de même de la classe $\W$ des $U$\nbd-équivalences, puisque le foncteur $U$ commute auxdites limites. En particulier, $\W$ est stable par composés transfinis et rétractes.
\tb

Montrons que toute flèche $f$ de $\C$ se décompose en $f=si$, avec
$i\in\Cof{\W}$ et $s\in\W$. En vertu de l'argument du petit objet, il existe
une décomposition $f=si$ de $f$, avec $i$ composé transfini d'images
directes d'éléments de $F(I)$, et $s\in r(F(I))$. L'assertion résulte donc
de l'hypothèse (\emph{a}), de la stabilité de $\Cof{\W}$ par composés
transfinis (proposition~\ref{prprcofGroth}) et images directes~(\emph{cf.}~\ref{defcofGroth}), et de l'inclusion $r(F(I))\subset\W$. On en déduit que \hbox{$\Cof{\W}\cap\W$} est stable par images directes (proposition~\ref{lemmeGroth}), composés transfinis et rétractes (proposition~\ref{prprcofGroth}). 
\tb

On va prouver que $lr(F(J))\subset\W$. En vertu de ce qui précède et de
l'argument du petit objet, il suffit de montrer que
$F(J)\subset\Cof{\W}\cap\W$. Or, par l'hypothèse (\emph{b}), on a
$F(J)\subset\W$. D'autre part, les éléments de $J$ sont des cofibrations de
$\M$, autrement dit, des rétractes de composés transfinis d'images directes
d'éléments de $I$, et comme le foncteur $F$ commute à ces opérations, les
propriétés de stabilité de $\Cof{\W}$ (\emph{cf.}~\ref{defcofGroth} et~proposition~\ref{prprcofGroth}) et l'hypothèse~(\emph{a}) impliquent que $F(J)\subset\Cof{\W}$, ce qui prouve l'assertion. Les assertions de la proposition, autres que la question de la propreté à gauche résultent donc du lemme précédent.
\tb

À nouveau, les propriétés de stabilité de $\Cof{\W}$, l'hypothèse
(\emph{a}), et l'argument du petit objet impliquent que les cofibrations de
la structure de catégorie de modèles ainsi obtenue sur $\C$ sont des
$\W$-cofibrations, et par suite, cette catégorie de modèles est propre à
gauche (\emph{cf.}~exemple~\ref{excofGroth}). 
\end{proof}

\section{Catégories de modèles à la Thomason}\label{absThom}

\begin{paragr}\label{enssimpl}
On note $\ord$ la catégorie des ensembles ordonnés et applications croissantes, considérée comme sous-catégorie pleine de la catégorie $\cat$ des petites catégories. On rappelle que la \emph{catégorie des simplexes} $\cats$ est la sous-catégorie pleine de $\ord$ formée des ensembles
$$\smp{m}=\{0,1,\dots,m\}\ ,\qquad m\geq0\ ,$$
ordonnés par l'ordre naturel. On note 
$$\face{i}{m}:\smp{m-1}\toto\smp{m}\ ,\qquad m>0\ ,\quad0\leq i\leq m\ ,$$
l'unique injection croissante dont l'image ne contient pas $i$. La catégorie des \emph{ensembles simpliciaux} est la catégorie $\simpl$ des préfaisceaux sur $\cats$. On identifiera, par le plongement de Yoneda, $\cats$ à une sous-catégorie pleine de $\simpl$. On pose 
$$
\begin{aligned}
&\bord{m}=\mathop\cup\limits_{0\leq i\leq m}\im\face{i}{m}\ ,\quad m>0\ ,\qquad\bord{0}=\varnothing\ ,\\
\noalign{\vskip 3pt}
&\cornet{m}{k}=\mathop\cup\limits_{i\neq k}\im\face{i}{m}\ ,\qquad\kern 12pt m>0\ ,\quad0\leq k\leq m\ ,
\end{aligned}
$$
l'image et la réunion étant prises au sens des préfaisceaux, et on note
$$i_m:\bord{m}\hookrightarrow\smp{m}\ ,\qquad   j_{m,k}:\cornet{m}{k}\hookrightarrow\smp{m}$$
les inclusions canoniques. Une \emph{fibration de Kan} est une flèche de
$\simpl$ ayant la propriété de relèvement à droite relativement aux
inclusions $j_{m,k}:\cornet{m}{k}\hookrightarrow\smp{m}$, $m>0$, \hbox{$0\leq k\leq m$}. 
On pose
$$
\begin{aligned}
&I=\{i_m:\bord{m}\hookrightarrow\smp{m}\,|\,m\geq0\}\ ,\\
\noalign{\vskip 3pt}
&J=\{j_{m,k}:\cornet{m}{k}\hookrightarrow\smp{m}\,|\,m>0\,,\ 0\leq k\leq m\}\ .
\end{aligned}
$$
On rappelle le théorème suivant de Quillen:
\end{paragr}

\begin{thm}\label{Quillen}
La catégorie des ensembles simpliciaux admet une structure de catégorie de modèles combinatoire propre, engendrée par $(I,J)$,
dont les cofibrations sont les monomorphismes et dont les fibrations sont les fibrations de Kan (ce qui suffit pour déterminer les équivalences faibles).
De plus, les équivalences faibles de cette structure sont stables par petites limites inductives filtrantes.
\end{thm}

\begin{proof}
Voir~\cite[II, 3.14, théorème~3]{Qu1} ou \cite[théorème 2.1.42 et proposition 2.3.20]{Ci}.
\end{proof}

Dans la suite, on considèrera toujours la catégorie $\simpl$ munie de cette structure de catégorie de modèles. En particulier, quand on parlera d'équivalences faibles d'ensembles simpliciaux, il s'agira toujours des équivalences faibles de cette structure.

\begin{paragr}\label{nerfdeAG}
Le foncteur d'inclusion $\cats\hookrightarrow\cat$ définit par le procédé de Kan un couple de foncteurs adjoints
$$c:\simpl\toto\cat\ ,\qquad N:\cat\toto\simpl\ ,$$
où $c$ est le foncteur de \emph{réalisation catégorique}, unique foncteur, à isomorphisme unique près, commutant aux petites limites inductives et prolongeant l'inclusion $\cats\hookrightarrow\cat$, et $N$ est le foncteur \emph{nerf} défini par
$$C\longmapsto\bigl(\smp{m}\mapsto\Hom_{\cat}(\smp{m},C)\bigr)\ .$$
\end{paragr}

\begin{paragr}\label{xi}
Pour tout ensemble ordonné $E$, on note $\xi(E)$ l'ensemble des parties finies non vides totalement ordonnées de $E$, muni de la relation d'ordre définie par l'inclusion des parties. On définit ainsi un foncteur $\xi:\ord\to\ord$. En composant la restriction à $\cats$ de ce foncteur avec la restriction du foncteur nerf à $\ord$, on obtient un foncteur 
$$
\xymatrixcolsep{2.9pc}
\xymatrix{
\cats\ar[r]^(.47){\xi\,|\,\cats}
&\ord\ar[r]^(.53){N\,|\,\ord}
&\simpl
}
$$
qui définit par le procédé de Kan un couple de foncteurs adjoints
$$\Sd:\simpl\toto\simpl\ ,\qquad \Ex:\simpl\toto\simpl\ ,$$ où $\Sd$ est le
foncteur de \emph{subdivision}, unique foncteur, à isomorphisme unique près,
commutant aux petites limites inductives et prolongeant le foncteur
$(N\,|\,\ord)\circ(\xi\,|\,\cats)$, et $\Ex$ est le foncteur défini par 
$$
X\longmapsto\bigl(\smp{m}\mapsto\Hom_{\simpl}(N\xi(\smp{m}),X)\bigr)\ .
$$
Pour tout ensemble ordonné $E$, on a une application croissante
$$\xi(E)\toto E\ ,\qquad S\longmapsto \max S\ ,$$
qui induit un morphisme de foncteurs $\alpha:\Sd\to1_{\simpl}$, et par transposition, un morphisme de foncteurs $\beta:1_{\simpl}\to\Ex$. On rappelle le théorème suivant de Kan:
\end{paragr}

\begin{thm}\label{Kan}
Les morphismes de foncteurs
$$\alpha:\Sd\toto1_{\simpl}\ ,\qquad\beta:1_{\simpl}\to\Ex$$
sont des équivalences faibles d'ensembles simpliciaux, argument par argument.
\end{thm}

\begin{proof}
Voir~\cite{Kan}, ou pour une version moderne~\cite[corollaire~2.1.27 et proposition~2.3.19]{Ci}.
\end{proof}

\begin{lemme}\label{Exfiltr}
Le foncteur $\Ex$ commute aux petites limites inductives filtrantes.
\end{lemme}

\begin{proof}
Pour tout entier $m \geq 0$, l'ensemble ordonné $\xi(\smp{m})$ est fini, et
donc $N\xi(\smp{m})$ est un ensemble simplicial de présentation finie
(\emph{cf.}~\cite[chapitre II, §~5.4]{GZ}). Par conséquent, le foncteur
$$X\longmapsto\Hom_{\simpl} (N\xi(\smp{m}),X)=(\Ex(X))_m$$ 
commute aux petites limites inductives filtrantes, ce qui prouve le lemme.
\end{proof}

\begin{paragr}
On note $\bordCS{m}$, $m\geq0$ (resp. $\cornetCS{m}{k}$, $m>0$, $0\leq k\leq m$) l'ensemble ordonné par inclusion des parties non vides de $\smp{m}$ distinctes de $\smp{m}$ (resp. qui ne contiennent pas $\{0,\dots,k-1,k+1,\dots,m\}$). On note
$$
\begin{aligned}
&\sigma_m:\bordCS{m}\hookrightarrow\xi(\smp{m})\ ,\qquad m\geq0\ ,\\
\noalign{\vskip 3pt}
&\tau_{m,k}:\cornetCS{m}{k}\hookrightarrow\xi(\smp{m})\ ,\qquad\kern 9pt m>0\ ,\quad 0\leq k\leq m\ ,
\end{aligned}
$$
les inclusions canoniques.
\end{paragr}

\begin{lemme}\label{Sd2can}
On a des isomorphismes canoniques 
$$\Sd^2\smp{m}\simeq N\xi^2\smp{m}\ ,\quad\Sd^2\bord{m}\simeq N\xi\bordCS{m}\ ,\quad\Sd^2\cornet{m}{k}\simeq N\xi\cornetCS{m}{k}\ ,$$
identifiant
$$
\Sd^2(
\xymatrix{
\bord{m}\ar[r]^{i_m}
&\smp{m}
})
\qquad\hbox{à}\qquad
N\xi\,(
\xymatrix{
\bordCS{m}\ar[r]^{\sigma_m}
&\xi\smp{m}
})
$$
et
$$
\Sd^2(
\xymatrix{
\cornet{m}{k}\ar[r]^{j_{m,k}}
&\smp{m}
})
\qquad\hbox{à}\qquad
N\xi\,(
\xymatrix{
\cornetCS{m}{k}\ar[r]^{\tau_{m,k}}
&\xi\smp{m}
})\ .
$$
\end{lemme}

\begin{proof}
C'est une conséquence de \cite[lemme 5.2.7 et exemple 2.1.34]{Ci}.
\end{proof}

\begin{paragr}\label{ordcriblant}
Dans ce qui suit, on considérera toujours la catégorie des ensembles
ordonnés $\ord$ munie de la structure de catégorie à cribles et cocribles
efficaces, induite par celle de $\cat$, définie par le segment criblant
$(\smp{1},\del{0},\del{1})$ (\emph{cf.}~exemple~\ref{nCatcribles} et
proposition~\ref{nCatcribleseff}).
\end{paragr}

\begin{lemme}\label{Dwyer}
Soit $i:E\to F$ une inclusion pleine d'ensembles ordonnés. Alors l'image $\xi(i):\xi E\to\xi F$ de celle-ci par le foncteur $\xi$ est un crible se factorisant en
$$
\xymatrix{
\xi E\ar[r]_k\ar@/^4ex/[rr]^{\xi(i)}
&W\ar[r]_j
&\xi F
} \ ,
$$ 
où $j$ est un cocrible et $k$ est un crible admettant une rétraction $r$ qui est aussi un adjoint à droite de~$k$.
\end{lemme}

\begin{proof}
L'inclusion d'ensembles ordonnés $\xi E \to \xi F$ est trivialement un crible. Notons $W$ le sous-ensemble ordonné de $\xi F$ formé des $S \in \xi F$ tels que $S \cap E$ soit non vide. Alors on a des inclusions $k:\xi E\to W$ et $j:W\to\xi F$, telles que $\xi(i)=jk$, $k$ étant un crible et $j$ un cocrible. On définit un foncteur $r: W \to \xi E$ par $S \mapsto S \cap E$, pour $S\in W$. Une vérification immédiate montre que celui-ci est une rétraction et un adjoint à droite de  $k$.
\end{proof}

\begin{thm}[Théorème de Thomason abstrait.]\label{Thomabs}
Soient $\C$ une catégorie à cribles et cocribles efficaces dont la catégorie sous-jacente est localement présentable, 
$$F:\simpl\toto\C\ ,\quad U:\C\toto\simpl$$
un couple de foncteurs adjoints, et $\W$ la classe des $U$-équivalences de $\C$. On suppose que
\begin{enumerate}
\item le foncteur $U$ commute aux petites limites inductives filtrantes;
\item le foncteur $U$ transforme les carrés cocartésiens de $\C$ formés de cocribles en carrés cocartésiens de $\simpl$;
\item le foncteur $FN$ transforme les cribles \emph{(resp.} les cocribles\emph{)} de $\ord$ en cribles \emph{(resp.} en cocribles\emph{)} de $\C$;
\item si $E'\to E$ est un crible de $\ord$ admettant une rétraction qui est aussi un adjoint à droite, son image par le foncteur $FN$ est une $U$\nbd-cofibration;
\item pour tout ensemble ordonné $E$, le morphisme d'adjonction $N(E)\to UF(N(E))$ est une équivalence faible d'ensembles simpliciaux.
\end{enumerate}
Alors il existe une structure de catégorie de modèles combinatoire sur $\C$ engendrée par $(F\Sd^2(I),F\Sd^2(J))$ dont les équivalences faibles sont les $U$-équivalences et les fibrations les flèches de $\C$ dont l'image par $\Ex^2U$ est une fibration de Kan. De plus, cette catégorie de modèles est propre.
\end{thm}

\begin{proof}
On remarque d'abord que le théorème~\ref{Kan} implique aussitôt que les $U$\nbd-équivalences coïncident avec les $\Ex^2U$\nbd-équivalences. D'autre part, en vertu de l'hypothèse~(\emph{a}) et du lemme~\ref{Exfiltr}, le foncteur $\Ex^2U$ commute aux petites limites inductives filtrantes. Le théorème sera donc conséquence du théorème~\ref{Quillen} et de la proposition~\ref{sansSmith}, appliquée au couple de foncteurs adjoints $(F\Sd^2,\Ex^2U)$, pourvu qu'on montre que le foncteur $F\Sd^2$ satisfait aux conditions~(\emph{a}) et~(\emph{b}) de cette proposition.
\tb

Montrons d'abord que $F\Sd^2(I)\subset\Cof{\W}$, autrement dit que pour tout \hbox{$m\geq0$,} le morphisme 
\smash{$F\Sd^2(
\xymatrix{
\bord{m}\ar[r]^{i_m}
&\smp{m}
})$}
est une $\W$\nbd-cofibration. En vertu du lemme~\ref{Sd2can}, le morphisme
\smash{$\Sd^2(
\xymatrix{
\bord{m}\ar[r]^{i_m}
&\smp{m}
})$}
s'identifie à 
\smash{$N\xi\,(
\xymatrix{
\bordCS{m}\ar[r]^{\sigma_m}
&\xi\smp{m}
})$},
et comme $\sigma_m$ est une inclusion pleine d'ensembles ordonnés, le lemme~\ref{Dwyer} implique que $\xi(\sigma_m)$ est un crible se factorisant en $\xi(\sigma_m)=jk\,$,
$$\xymatrix{
\xi(\bordCS{m})\ar[r]^(.6)k
&W\ar[r]^(.4)j
&\xi^2\smp{m}
}\ ,$$
où $j$ est un cocrible et $k$ est un crible admettant une rétraction $r$ qui est aussi un adjoint à droite de~$k$. L'hypothèse (\emph{c}) implique que $FN\xi(\sigma_m)$ (resp. $FN(j)$) est un crible (resp. un cocrible) de $\C$, et l'hypothèse (\emph{d}) que $FN(k)$ est une $U$\nbd-cofibration. Comme $U$ est un adjoint à droite, il respecte les monomorphismes, et en particulier, transforme les cocribles de $\C$ en monomorphismes de $\simpl$. L'hypothèse (\emph{b}) implique donc que $U$ transforme les carrés cocartésiens de $\C$ formés de cocribles en carrés homotopiquement cocartésiens. Il résulte donc de la proposition~\ref{lemmeDwyer} que $FN\xi(\sigma_m)$ est une $U$-cofibration, et en particulier une $\W$-cofibration (proposition~\ref{UcofWcof}), ce qui prouve l'assertion.
\tb

Il reste à prouver que $F\Sd^2(J)\subset\W$, autrement dit que pour tout $m>0$ et tout $k$, $0\leq k\leq m$, le morphisme 
\smash{$
UF\Sd^2(
\xymatrix{
\cornet{m}{k}\ar@{}@<-.4ex>[r]^{j_{m,k}}\ar[r]
&\smp{m}
})$}
est une équivalence faible d'ensembles simpliciaux. En vertu du lemme~\ref{Sd2can}, le morphisme
\smash{$
\Sd^2(
\xymatrix{
\cornet{m}{k}\ar@{}@<-.4ex>[r]^{j_{m,k}}\ar[r]
&\smp{m}
})$}
s'identifie à
\smash{$N\xi\,(
\xymatrix{
\cornetCS{m}{k}\ar@{}@<-.4ex>[r]^{\tau_{m,k}}\ar[r]
&\xi\smp{m}
})$}. Or, on a un carré commutatif
$$
\xymatrixrowsep{2.8pc}
\xymatrixcolsep{4.5pc}
\xymatrix{
&N\xi(\cornetCS{m}{k})\ar[r]^(.487){N\xi(\tau_{m,k})}\ar[d]
&N\xi^2(\smp{m})\ar[d]
\\
&UFN\xi(\cornetCS{m}{k})\ar[r]^(.487){UFN\xi(\tau_{m,k})}
&UFN\xi^2(\smp{m})
&\kern -45pt,\kern 45pt
}
$$
dont les flèches verticales sont les morphismes d'adjonction, qui sont des équivalences faibles par l'hypothèse (\emph{e}). Comme $j_{m,k}$ est une équivalence faible, en vertu du théorème~\ref{Kan}, il en est de même de $\Sd^2(j_{m,k})$, donc aussi de la flèche horizontale du haut du carré. L'assertion résulte donc de la propriété du deux sur trois.
\end{proof}

\begin{rem}\label{remThomabseq}
Comme en vertu du théorème~\ref{Kan} les $U$\nbd-équivalences coïncident
avec les $\Ex^2U$\nbd-équivalences, sous les hypothèses du théorème, le
foncteur $\Ex^2U$ respecte les fibrations et les fibrations triviales, et
par suite, le couple $(F\Sd^2,\Ex^2U)$ est une adjonction de Quillen. De
plus, comme le théorème~\ref{Kan} implique aussi que le foncteur $\Ex$
induit une autoéquivalence de la catégorie homotopique des ensembles
simpliciaux, l'adjonction $(F\Sd^2,\Ex^2U)$ est une équivalence de Quillen
si et seulement si le foncteur $U$ induit une équivalence entre la catégorie
homotopique de la catégorie de modèles $\C$ du théorème précédent et celle
des ensembles simpliciaux (\emph{cf.}~\cite[Proposition 1.3.13]{Ho}). 
\end{rem}

\begin{rem}\label{remThomabs}
Le plus souvent, pour montrer la condition~(\emph{d}) du théorème précédent,
on montrera la condition plus forte suivante (\emph{cf.}~\ref{UeqUcof}):
\begin{enumerate}
\item[(d${}'$)] \emph{si $E'\to E$ est un crible de $\ord$ admettant une rétraction qui est aussi un adjoint à droite, son image par le foncteur $FN$ est une $U$\nbd-équivalence et le reste après tout cochangement de base.}
\end{enumerate}
Par ailleurs, on observe que dans la démonstration du théorème 4.11, les
conditions~(\emph{d}) et (\emph{e}) ne sont appliquées qu'à des ensembles ordonnés
\emph{finis}.
\end{rem}

\section{Nerfs simpliciaux de $n$-catégories}\label{nCat}

Dans la suite, on se fixe $n$, $0<n\leq\infty$, et un foncteur
$i:\cats\to\nCat{n}$ de la catégorie des simplexes vers celle des petites
$n$\nbd-catégories strictes. On rappelle que la catégorie $\nCat{n}$ est
localement présentable (pour $n=\infty$, cela résulte par exemple
de~\cite[proposition~1.51]{AR} et~\cite[proposition~3.14]{Ara} ; le cas
général résulte du cas $n = \infty$ et de \cite[théorème~1.39,
condition~(\emph{ii})]{AR} appliqué à l'inclusion $\nCat{n}\to\nCat{\infty}$
qui admet à la fois un adjoint à gauche et un adjoint à droite
(\emph{cf.}~\ref{tronq})). En particulier, la catégorie $\nCat{n}$ est
complète et cocomplète. On en déduit, par le procédé de Kan, un couple de
foncteurs adjoints
$$i^{}_!:\simpl\to\nCat{n}\ ,\qquad i^*:\nCat{n}\to\simpl\ ,$$
où $i^{}_!$ est l'unique foncteur, à isomorphisme unique près, commutant aux
petites limites inductives et prolongeant le foncteur $i$, et $i^*$ est
défini par
$$
C\longmapsto\bigl(\smp{m}\mapsto\Hom_{\nCat{n}}(i(\smp{m}),C)\bigr)\ .
$$
En vertu du théorème~\ref{Thomabs} (appliqué à la catégorie $\C=\nCat{n}$ et
aux foncteurs $F=i^{}_!$ et $U=i^*$), pour avoir une structure de catégorie
de modèles à la Thomason sur $\nCat{n}$, il suffit de vérifier que, pour un
foncteur $i$ convenable, le couple de foncteurs adjoints~$(i^{}_!,i^*)$
satisfait aux conditions (\emph{a})--(\emph{e}) de ce théorème. Dans
cette section, on présente des conditions suffisantes sur le foncteur $i$
pour que les conditions (\emph{a}), (\emph{b}) et~(\emph{c}) dudit théorème
soient satisfaites, et on donne un exemple d'un tel foncteur.

\begin{prop}\label{nerffiltant}\label{conditiona}
Si pour tout $m\geq0$, $i(\smp{m})$ est un objet de présentation finie de
$\nCat{n}$, et en particulier si $i(\smp{m})$ est une $n$\nbd-catégorie
finie, alors le foncteur $i^*$ commute aux limites inductives filtrantes.
\end{prop}

\begin{proof}
En vertu de la définition des objets de présentation finie (\emph{cf.}~\ref{acc}) et de celle de $i^*$, la seule chose à vérifier est qu'une $n$\nbd-catégorie finie est de présentation finie. Cette dernière assertion est conséquence facile du fait que pour tout $i\geq0$, le foncteur $\fl^{}_i:\nCat{n}\to\ens$, associant à une petite $n$\nbd-catégorie l'ensemble de ses $i$\nbd-cellules, commute aux limites inductives filtrantes, et de la commutation dans $\ens$ des limites inductives filtrantes aux limites projectives finies. 
\end{proof}

\begin{lemme}
Soit $C$ une petite $n$\nbd-catégorie admettant un objet $x^{}_0$ tel que pour tout objet $x$ de $C$, l'ensemble des $1$\nbd-flèches de $C$ de $x^{}_0$ vers $x$ soit non vide. Alors le foncteur $\Hom_{\nCat{n}}(C,?)$ transforme les carrés cocartésiens de $\nCat{n}$ formés de cocribles en carrés cocartésiens d'ensembles.
\end{lemme}

\begin{proof}
Soit
$$
\xymatrix{
D_0\ar[r]\ar[d]
&D_1\ar[d]
\\
D_2\ar[r]
&D
}$$
un carré cocartésien de $\nCat{n}$ formé de cocribles, qui sont en particulier des monomorphismes. On peut donc supposer, pour simplifier, que les $D_i$, $i=0,1,2$, sont des sous\nbd-$n$\nbd-catégories de $D$. En vertu de la proposition~\ref{imdircribles}, ce carré est aussi cartésien. Ainsi, le carré
$$
\xymatrix{
\Hom_{\nCat{n}}(C,D_0)\ar[r]\ar[d]
&\Hom_{\nCat{n}}(C,D_1)\ar[d]
\\
\Hom_{\nCat{n}}(C,D_2)\ar[r]
&\Hom_{\nCat{n}}(C,D)
}$$
est un carré cartésien d'ensembles formé d'inclusions. Pour montrer qu'il est cocartésien, il suffit donc de montrer que tout $n$\nbd-foncteur $u:C\to D$ se factorise par $D_1$ ou $D_2$. Comme le foncteur associant à une $n$\nbd-catégorie son ensemble d'objets admet un adjoint à droite, il transforme les carrés cocartésiens en carrés cocartésiens. On en déduit que $u(x^{}_0)$ est dans $D_1$ ou dans $D_2$. Supposons, par exemple, que $u(x^{}_0)$ soit un objet de $D_1$ et montrons qu'alors $u$ se factorise par $D_1$. Soient $i$, $0\leq i\leq n$, et $\gamma$ une $i$\nbd-cellule de $C$ d'objet source itérée $x$. Par hypothèse, il existe une $1$\nbd-flèche $f$ de $C$ de~$x^{}_0$ vers $x$. Comme $u(x^{}_0)$ est un objet de $D_1$, il résulte de la proposition~\ref{carcriblesncat} que $u(f)$ est dans $D_1$, et en particulier, que $u(x)$ est dans $D_1$. Une nouvelle application de la proposition~\ref{carcriblesncat} implique que $u(\gamma)$ est dans $D_1$, ce qui achève la démonstration.
\end{proof}

\begin{prop}\label{nerfcoc}\label{conditionb}
Si pour tout $m\geq0$, la $n$\nbd-catégorie $i(\smp{m})$ admet un objet $x^{}_m$ tel que pour tout objet $x$ de $i(\smp{m})$, l'ensemble des $1$\nbd-flèches de $i(\smp{m})$ de $x^{}_m$ vers $x$ soit non vide, alors le foncteur $i^*$ transforme les carrés cocartésiens de $\nCat{n}$ formés de cocribles en carrés cocartésiens de $\simpl$.
\end{prop}

\begin{proof} Vu la définition de $i^*$, il s'agit d'une conséquence immédiate du lemme précédent.
\end{proof}

\begin{lemme}\label{limindnerf}
Soit $E$ un ensemble ordonné. Alors le morphisme canonique 
\[ \limind_{S\in\xi E}NS\to NE\ , \]
induit par les inclusions $S\to E$, est un isomorphisme.
\end{lemme}

\begin{proof}
Le lemme résulte d'une vérification directe, ou de l'observation que
l'ensemble ordonné $\xi E$ (\emph{cf.}~\ref{xi}) s'identifie à une sous-catégorie cofinale de la catégorie $\cm{\cats}{NE}$ des simplexes de $NE$, en associant à un sous-ensemble totalement ordonné fini non vide $S$ de $E$ le $m$\nbd-simplexe de $NE$, $m=\card(S)-1$, défini par la suite strictement croissante des éléments de $S$.
\end{proof}

\begin{prop}\label{conditionc}
On suppose que le foncteur $i:\cats\to\nCat{n}$ satisfait aux deux conditions suivantes:
\begin{enumerate}
\item il respecte le segment $(\smp{1},\del{0},\del{1})$ \emph{(cf.~\ref{ordcriblant}),} autrement dit, $i(\smp{0})=\smp{0}$, \hbox{$i(\smp{1})=\smp{1}$} et $i(\del{\e})=\del{\e}$, \emph{$\e=0,1$;}
\item si $\smp{p}\to\smp{q}$ est une inclusion comme section commençante \emph{(resp.} finissante\emph{),} 
alors le $n$-foncteur $i(\smp{p})\to i(\smp{q})$ est un crible \emph{(resp.} un cocrible\emph{)}, et de plus, si $\chi:\smp{q}\to\smp{1}$ désigne l'application croissante définissant le crible \emph{(resp.} le cocrible\emph{)} $\smp{p}\to\smp{q}$ de $\ord$, alors le crible \emph{(resp.} le cocrible\emph{)} $i(\smp{p})\to i(\smp{q})$ de $\nCat{n}$ est défini par le $n$\nbd-foncteur $i^{}_!N(\chi)=i(\chi):i(\smp{q})\to i(\smp{1})$.
\end{enumerate}
Alors le foncteur $i^{}_!N$ transforme les cribles \emph{(resp.} les cocribles\emph{)} de $\ord$ en cribles \emph{(resp.} en cocribles\emph{)} de $\nCat{n}$.
\end{prop}

\begin{proof}
Soit donc $E'\to E$ un crible (resp. un cocrible) de $\ord$, qu'on peut supposer être une inclusion pour simplifier, défini par une flèche $\chi:E\to\smp{1}$, de sorte que l'on a un carré cartésien d'ensembles ordonnés
$$
\xymatrix{
&E'\ar[r]\ar[d]
&\smp{0}\ar[d]^{\del{\e}}
\\
&E\ar[r]_{\chi}
&\smp{1}
&\kern -15pt,\kern15pt
}
$$
où $\e=0$ (resp. $\e=1$). Pour tout $S\in\xi E$, on en déduit un carré cartésien
$$
\xymatrix{
&\,\,S\cap E'\ar[r]\ar[d]
&\smp{0}\ar[d]^{\del{\e}}
\\
&S\ar[r]_{\chi^{}_S}
&\smp{1}
&\kern -15pt,\kern15pt
}
$$
où $\chi^{}_S$ désigne la restriction de $\chi$ à $S$. Comme le foncteur $N$ commute aux limites projectives, on en déduit des carrés cartésiens de~$\simpl$
$$
\xymatrix{
&NE'\ar[r]\ar[d]
&\smp{0}\ar[d]^{\del{\e}}
&&\,\,N(S\cap E')\ar[r]\ar[d]
&\smp{0}\ar[d]^{\del{\e}}
\\
&NE\ar[r]_{N(\chi)}
&\smp{1}
&\kern -15pt,\kern15pt
&NS\ar[r]_{N(\chi^{}_S)}
&\smp{1}
&\kern -15pt.\kern15pt
}
$$
En vertu du lemme~\ref{limindnerf}, et du fait que les limites inductives sont universelles dans $\simpl$, la flèche $NE'\to NE$ s'identifie donc à la limite inductive 
$$\limind_{S\in\xi E}N(S\cap E')\toto\limind_{S\in\xi E}NS$$
des flèches $N(S\cap E')\to NS$. Comme le foncteur $i^{}_!$ commute aux limites inductives, on en déduit que le $n$\nbd-foncteur $i^{}_!NE'\to i^{}_!NE$ s'identifie à la limite inductive
$$\limind_{S\in\xi E}i^{}_!N(S\cap E')\toto\limind_{S\in\xi E}i^{}_!NS$$
des flèches $i^{}_!N(S\cap E')\to i^{}_!NS$. Or, le carré
$$
\xymatrix{
\,\,i^{}_!N(S\cap E')\ar[r]\ar[d]
&\smp{0}\ar[d]^{\del{\e}}
\\
i^{}_!NS\ar[r]_(.54){i^{}_!N(\chi^{}_S)}
&\smp{1}
}
$$
est cartésien. En effet, si $S\cap E'\neq\varnothing$, l'inclusion $S\cap E'\to S$ s'identifie à 
une inclusion de la forme $\smp{p}\to\smp{q}$, section commençante (resp. finissante), 
et l'assertion résulte de la condition (\emph{b}). Si $S\cap E'=\varnothing$, on a aussi $i^{}_!N(S\cap E')=\varnothing$,  et $\chi^{}_S$ se factorise par $\del{\e'}$, $\e'=1-\e$, donc grâce à la condition (\emph{a}), $i^{}_!N(\chi^{}_S)$ aussi, ce qui implique l'assertion. Comme en vertu du lemme~\ref{deladj} le foncteur $\delet{\e}:\cm{\nCat{n}}{\smp{1}}\to\cm{\nCat{n}}{\smp{0}}$ admet un adjoint à droite, on en déduit un carré cartésien
$$
\xymatrixrowsep{2.8pc}
\xymatrixrowsep{2.5pc}
\xymatrix{
&i^{}_!NE'\simeq\smash{\limind\limits_{S\in\xi E}}i^{}_!N(S\cap E')\ar[r]\ar@<2.3ex>[d]
&\smp{0}\ar[d]^{\del{\e}}
\\
&i^{}_!NE\simeq\limind\limits_{S\in\xi E}i^{}_!NS\ar[r]_(.68){i^{}_!N(\chi^{}_S)}
&\smp{1}
&\kern-15pt,\kern15pt
}
$$
ce qui achève la démonstration. 
\end{proof}

\begin{paragr}\label{stablecomp}
Soit $C$ une $n$\nbd-catégorie. Un \emph{ensemble multiplicatif} de cellules de $C$ est un ensemble $M$ de cellules de $C$ satisfaisant aux deux conditions suivantes:
\begin{enumerate}[label=(\emph{\alph*})]
\item pour tout $i$, $0\leq i<n$, si $f$ est une $i$\nbd-cellule de $C$ appartenant à $M$, la \hbox{$(i+1)$}\nbd-flèche identité de $f$ appartient aussi à $M$;
\item pour tous $i,j$, $0\leq j<i\leq n$, si $g$ et $h$ sont deux $i$\nbd-flèches de $C$ appartenant à~$M$ telles que la $j$\nbd-cellule source itérée de $g$ soit égale à la $j$\nbd-cellule but itéré de~$h$, alors la $i$\nbd-flèche composée $g*_jh$ appartient à $M$.
\end{enumerate}
On dit qu'un ensemble $S$ de cellules de $C$ \emph{engendre par compositions} la $n$\nbd-catégorie~$C$ si l'ensemble de toutes les cellules de $C$ est le plus petit ensemble multiplicatif de cellules de $C$ contenant $S$. Dans ce cas, $S$ contient forcément l'ensemble des objets de~$C$, et pour tout $i>0$, toute $i$\nbd-flèche de $C$ est un composé d'un nombre fini de $i$\nbd-flèches, qui sont dans $S$ ou sont des unités itérées de $j$\nbd-cellules, $0\leq j<i$, appartenant à $S$.
\end{paragr}

\begin{lemme}\label{criblestablecomp}
Soient $C$ une $n$\nbd-catégorie, $S$ un ensemble de cellules de $C$ qui engendre la $n$\nbd-catégorie $C$ par compositions, et $C'$ une sous\nbd-$n$\nbd-catégorie de $C$. Si toute cellule de $C$ appartenant à $S$ dont l'objet but itéré \emph{(resp.} source itérée\emph{)} est dans $C'$ est elle-même dans $C'$, alors l'inclusion $C'\to C$ est un crible \emph{(resp.} un cocrible\emph{)}.
\end{lemme}

\begin{proof}
Pour toute $i$\nbd-flèche $f$ de $C$, $1\leq i\leq n$, on note $l(f)$ le nombre minimal de compositions nécessaires pour exprimer $f$ comme composé de $i$\nbd-flèches qui sont dans $S$ ou sont des unités itérées de cellules appartenant à $S$. On va montrer, par récurrence sur $l(f)$, que si l'objet but itéré (resp. source itérée) de $f$ est dans $C'$, alors $f$ est elle-même dans $C'$, ce qui prouvera le lemme en vertu de la proposition~\ref{carcriblesncat}. Si $l(f)=0$, alors $f$ appartient à $S$, ou est une unité itérée d'une cellule appartenant à~$S$, cela est donc vrai par hypothèse. Supposons la propriété vérifiée si $l(f)<k$, pour un $k>0$, et montrons-la si $l(f)=k$. Un tel $f$ se décompose en \smash{$g*_jh$}, $0\leq j<i$, avec $l(g),l(h)<k$. On distingue deux cas. Si $j=0$, l'objet but itéré (resp. source itérée) de $f$ est égal à celui de $g$ (resp.~de~$h$). Par hypothèse de récurrence, $g$ (resp. $h$) est dans $C'$, donc aussi son objet source itérée (resp. but itéré), qui n'est autre que l'objet but itéré de $h$ (resp. source itérée de $g$). On en déduit, toujours par hypothèse de récurrence, que $h$ (resp. $g$) est également dans $C'$, et par stabilité de $C'$ par composition, $f$ aussi. Si $j>0$, alors l'objet but itéré (resp. source itérée) de~$f$ est égal à celui commun de~$g$ et $h$, qui sont donc par hypothèse de récurrence dans $C'$, et on conclut à nouveau par stabilité par composition, ce qui achève la démonstration.
\end{proof}

\begin{paragr}\label{tronq}
On rappelle que pour tous $m,n$, $1\leq m<n\leq\infty$, l'inclusion pleine
$\nCat{m}\to\nCat{n}$ (définie en considérant une $m$\nbd-catégorie comme
une $n$\nbd-catégorie dont toutes les $i$\nbd-flèches, pour $i>m$, sont des
identités) admet à la fois un adjoint à gauche et un adjoint à droite.
L'adjoint à droite associe à toute petite $n$\nbd-catégorie $C$, le
$m$\nbd-\emph{tronqué bête} de $C$ dont les $i$\nbd-cellules, $0\leq i\leq
m$, sont les $i$\nbd-cellules de $C$, les sources, buts, unités et
compositions venant de celles dans $C$.
L'adjoint à gauche associe à $C$ le
$m$\nbd-\emph{tronqué intelligent} de $C$ dont les $i$\nbd-cellules, $0\leq
i< m$, sont les $i$\nbd-cellules de $C$ et les $m$\nbd-flèches sont les
$m$\nbd-flèches de $C$, modulo la relation d'équivalence identifiant deux
$m$\nbd-flèches de $C$ s'il existe un zigzag de $(m+1)$\nbd-flèches de $C$
les reliant. Les sources, buts, unités et compositions sont induites par
celles de $C$.
\end{paragr}

\begin{paragr}\label{nerfdeStreet}
Street a introduit dans~\cite{S1} un foncteur $i_\infty:\cats\to\nCat{\infty}$ (voir aussi~\cite{S2,S2cor,Steiner,SteinerOr}), permettant par le procédé de Kan de définir un couple de foncteurs adjoints
$$c_\infty:\simpl\to\nCat{\infty}\ ,\qquad N_{\infty}:\nCat{\infty}\to\simpl\ ,$$
l'adjoint à droite $N_{\infty}$ étant connu sous le nom de \emph{nerf de Street}. Le foncteur $i_\infty$ \hbox{associe} à l'objet $\smp{m}$ de $\cats$ une $m$\nbd-catégorie $\Ornt{m}$ (considérée comme $\infty$\nbd-catégorie), appelée \emph{l'oriental de Street de dimension $m$}. 
\tb
Les objets de $\Ornt{m}$ sont les mêmes que ceux de la catégorie $\smp{m}$,
autrement dit, les entiers $0,1,\dots,m$. Pour $0<i\leq m$, les
$i$\nbd-flèches \emph{indécomposables} (c'est-à-dire les $i$-flèches qui ne peuvent pas s'écrire de façon non triviale comme composées de deux $i$\nbd-flèches) sont les parties à $i+1$ éléments de l'ensemble des objets de $\Ornt{m}$. Ces parties s'identifient aux suites strictement croissantes $x_0<x_1<\cdots<x_i$ d'entiers entre $0$ et~$m$, autrement dit aux applications strictement croissantes $x:\smp{i}\to\smp{m}$, $i$\nbd-simplexes non dégénérés de $\smp{m}$ (identifié à un ensemble simplicial par le plongement de \hbox{Yoneda}). 
Le but (resp. la source) d'une telle $i$\nbd-flèche est un composé des
\hbox{$(i-1)$}\nbd-cellules correspondant aux applications strictement
croissantes $x\face{k}{i}$ (\emph{cf.}~\ref{enssimpl}), $0\leq k\leq i$, $k$
pair (resp. impair), et d'unités itérées de $j$\nbd-cellules, $j<i-1$. Pour
$i$ grand, les formules définissant ces composés sont très compliquées, mais,
dans cet article, nous aurons seulement besoin de les décrire en petite
dimension. Si $0\leq x_0<x_1\leq m$, alors $\{x_0,x_1\}$ est une
$1$\nbd-flèche de $\Ornt{m}$ de source $x_0$ et but $x_1$. Si $0\leq
x_0<x_1<x_2\leq m$, alors $\{x_0,x_1,x_2\}$ est une $2$\nbd-flèche de
$\Ornt{m}$ de source $\{x_0,x_2\}$ et but $\{x_1,x_2\}*_0\{x_0,x_1\}$.
\[
\shorthandoff{;}
\xymatrix{
&x_1\ar[rd]^{\{x_1,x_2\}} \\
x_0\ar[ru]^{\{x_0,x_1\}} \ar[rr]_{\{x_0,x_2\}}
&
\ar@{}[u]^(.10){}="a"^(.75){}="b"
\ar@{=>}"a";"b"|(.45){\scriptstyle\{\!x_0,x_1,x_2\!\}}
%\ar@{=>}"a";"b"|(.45){\hbox{\vphantom{0}$\scriptstyle\{\!x_0,x_1,x_2\!\}$}}
&x_2
}
\]
\begin{comment}
$$
\xymatrixrowsep{.8pc}
\xymatrixcolsep{4.3pc}
\xymatrix{
&x_1\ar[rddddd]^{\{x_1,x_2\}}
\\
\\
&
\\
\\
&\ar@{==>}[uu]|{\scriptscriptstyle\{\!x_0,x_1,x_2\!\}}
%&\ar@{=>}[uu]_(.4){\scriptscriptstyle\{\!x_0,x_1,x_2\!\}}
\\
x_0\ar[ruuuuu]^{\{x_0,x_1\}}
\ar[rr]_{\{x_0,x_2\}}
&&x_2
}
$$
\end{comment}
Si $0\leq x_0<x_1<x_2<x_3\leq m$, alors $\{x_0,x_1,x_2,x_3\}$ est une $3$\nbd-flèche de $\Ornt{m}$ de source $(1_{\{x_2,x_3\}}*_0\{x_0,x_1,x_2\})*_1\{x_0,x_2,x_3\}$ et but $(\{x_1,x_2,x_3\}*_01_{\{x_0,x_1\}})*_1\{x_0,x_1,x_3\}$. 
Si $0\leq x_0<x_1<\cdots<x_i\leq m$, alors l'objet source itérée (resp. but itéré) de la $i$\nbd-flèche $\{x_0,x_1,\dots,x_i\}$ de $\Ornt{m}$ est $x_0$ (resp. $x_i$).
\tb

Un \emph{atome} de la $m$\nbd-catégorie $\Ornt{m}$ est un objet ou une $i$\nbd-flèche indécomposable de~$\Ornt{m}$, $1\leq i\leq m$. Ainsi, les atomes de $\Ornt{m}$ sont en bijection avec les parties non vides de l'ensemble $\{0,1,\dots,m\}$, ou les simplexes non dégénérés de $\smp{m}$. L'ensemble des atomes engendre par compositions la $m$\nbd-catégorie $\Ornt{m}$~\cite[théorème~3.15]{S1} (et l'engendre même librement~\cite[théorème~4.1]{S1}). En particulier, les $m$\nbd-catégories $\Ornt{m}$ sont finies, $\Ornt{0}$ est un singleton s'identifiant à $\smp{0}$, $\Ornt{1}$ s'identifie à $\smp{1}$ et $\Ornt{2}$ est la $2$\nbd-catégorie 
\[
\shorthandoff{;}
\xymatrix{
&1\ar[rd]^{\{1,2\}} \\
0\ar[ru]^{\{0,1\}} \ar[rr]_{\{0,2\}}
&
\ar@{}[u]^(.10){}="a"^(.75){}="b"
\ar@{=>}"a";"b"|(.45){\hbox{\vphantom{0}$\scriptstyle\{\!0,1,2\!\}$}}
&2 \pbox{.}
}
\]
\begin{comment}
$$
\xymatrixrowsep{.6pc}
\xymatrixcolsep{3.7pc}
\xymatrix{
&1\ar[rddddd]^{\{1,2\}}
\\
\\
&
\\
\\
&\ar@{=>}[uu]_(.4){\{0,1,2\}}
\\
0\ar[ruuuuu]^{\{0,1\}}
\ar[rr]_{\{0,2\}}
&&2\kern 16pt.\kern-16pt
}
$$
\end{comment}

Soit $\varphi:\smp{p}\to\smp{q}$ un morphisme de $\cats$. Le $\infty$\nbd-foncteur $i_\infty(\varphi):\Ornt{p}\to\Ornt{q}$ associe à un objet $k$, $0\leq k\leq p$, de $\Ornt{p}$ l'objet $\varphi(k)$ de $\Ornt{q}$ et à une $i$\nbd-flèche indécomposable $E\subset\{0,1,\dots,p\}$, $i=\card(E)-1$, de $\Ornt{p}$ la $i$\nbd-flèche unité itérée de la $j$\nbd-cellule $\varphi(E)$ de~$\Ornt{q}$ ($j=\card(\varphi(E))-1$). En particulier, si $\varphi$ est un monomorphisme, le $\infty$\nbd-foncteur $i_\infty(\varphi)$ associe à un atome de $\Ornt{p}$ correspondant à une partie $E$ de $\{0,1,\dots,p\}$, l'atome de $\Ornt{q}$ correspondant à la partie $\varphi(E)$ de $\{0,1,\dots,q\}$. 
\end{paragr}

\begin{paragr}\label{trnerfdeStreet}
Pour $n\geq1$, on note $N_n:\nCat{n}\to\simpl$ la restriction du nerf de
Street à la sous-catégorie pleine $\nCat{n}$ de $\nCat{\infty}$. Ce foncteur
admet un adjoint à gauche \hbox{$c_n:\simpl\to\nCat{n}$}, égal au composé du
foncteur $c_\infty:\simpl\to\nCat{\infty}$ avec le foncteur de
$n$\nbd-troncation intelligente
$\nCat{\infty}\to\nCat{n}$~(\emph{cf.}~\ref{tronq}). Le couple de foncteurs
adjoints $(c_n,N_n)$ s'obtient également par le procédé de Kan à partir du
foncteur \hbox{$i_n:\cats\to\nCat{n}$}, composé du foncteur
$i_\infty:\cats\to\nCat{\infty}$ avec le foncteur de $n$\nbd-troncation
intelligente. Ainsi, le foncteur $i_n$ associe à $\smp{m}$ le
$n$\nbd-tronqué intelligent~$\trOrnt{m}{n}$ de l'oriental de
Street~$\Ornt{m}$. Pour tout $m\geq0$, l'ensemble des objets de
$\trOrnt{m}{n}$ est égal à $\{0,1,\dots,m\}$ et, pour $n=1$, la catégorie
$\trOrnt{m}{1}$ s'identifie à $\smp{m}$ (de sorte que $N_1$ est le foncteur
nerf habituel $N:\cat\to\simpl$ et $c_1$ son adjoint à gauche $c$ (\emph{cf.}~\ref{nerfdeAG})).
\end{paragr}

\begin{paragr}\label{eqfThom}
Soit $n$, $1\leq n\leq\infty$. On dit qu'une flèche de $\nCat{n}$,
$n$\nbd-foncteur strict entre deux petites $n$\nbd-catégories strictes, est
une \emph{équivalence faible de Thomason} si son image par le nerf de Street
(ou ce qui revient au même par $N_n$ ou par $\Ex^2N_n$~(\emph{cf.}~\ref{Kan})) est une équivalence faible d'ensembles simpliciaux. On dit qu'elle est une \emph{fibration de Thomason} si son image par $\Ex^2N_n$ est une fibration de Kan. Si $n=1$, on retrouve ainsi les équivalences faibles et les fibrations de la structure de catégorie de modèles sur $\cat$, introduite par Thomason dans~\cite{Th}.
\end{paragr}

\begin{lemme}\label{orientauxcribles}
Soient $n\geq1$, et $\smp{p}\to\smp{q}$ une inclusion comme section commençante \emph{(resp.} finissante\emph{)}. Alors le $n$-foncteur $\trOrnt{p}{n}\to\trOrnt{q}{n}$ est un crible \emph{(resp.} un cocrible\emph{)}. De plus, si $\chi:\smp{q}\to\smp{1}$ désigne l'application croissante définissant le crible \emph{(resp.} le cocrible\emph{)} $\smp{p}\to\smp{q}$, alors le crible \emph{(resp.} le cocrible\emph{)} $\trOrnt{p}{n}\to\trOrnt{q}{n}$ est défini par le $n$\nbd-foncteur $c_nN(\chi)$.
\end{lemme}

\begin{proof}
Supposons par exemple que $\smp{p}\to\smp{q}$ est une inclusion comme section commençante, et montrons que $\trOrnt{p}{n}\to\trOrnt{q}{n}$ est un crible.
La $n$\nbd-catégorie $\trOrnt{q}{n}$ est engendrée par compositions par l'ensemble des cellules de $\trOrnt{q}{n}$ correspondant aux parties non vides $E$ de $\{0,1,\dots,q\}$. L'objet but itéré d'une telle cellule est le plus grand élément de $E$. Si ce dernier est dans $\trOrnt{p}{n}$, on a donc $E\subset\{0,1,\dots,p\}$ et la cellule correspondante est donc dans $\trOrnt{p}{n}$. On conclut par le lemme~\ref{criblestablecomp}. La dernière assertion résulte facilement de ce qui précède. Le cas d'une section finissante se démontre de façon analogue, ou s'en déduit par un argument de dualité.
\end{proof}

\begin{prop}\label{prop:cN_abc}
Pour tout $n\geq1$, le couple de foncteurs adjoints $(c_n,N_n)$ satisfait
aux conditions \emph{(a), (b)} et \emph{(c)} du
théorème~\emph{{\ref{Thomabs}}}, autrement dit:
\begin{enumerate}
\item le foncteur $N_n$ commute aux petites limites inductives filtrantes;
\item le foncteur $N_n$ transforme les carrés cocartésiens de $\nCat{n}$ formés de cocribles, en carrés cocartésiens de $\simpl$;
\item le foncteur $c_nN$ transforme les cribles \emph{(resp.} les cocribles\emph{)} de $\ord$ en cribles \emph{(resp.} en cocribles\emph{)} de $\nCat{n}$.
\end{enumerate}
\end{prop}

\begin{proof}
Comme pour tout $m\geq0$, la $m$\nbd-catégorie $i_n(\smp{m})=\trOrnt{m}{n}$ est finie, l'assertion~(\emph{a}) résulte de la proposition~\ref{conditiona}. Pour tout objet $i$, $0\leq i\leq m$, de $\trOrnt{m}{n}$, il existe une $1$\nbd-flèche de $\trOrnt{m}{n}$ de source $0$ et but $i$. L'assertion~(\emph{b}) résulte donc de la proposition~\ref{conditionb}. Vu que le foncteur $i_n$ respecte le segment $(\smp{1},\del{0},\del{1})$, l'assertion~(\emph{c}) résulte du lemme précédent et de la proposition~\ref{conditionc}.
\end{proof}

\begin{sch}\label{remnThom}
En vertu de la proposition précédente, du théorème~\ref{Thomabs} et de la remarque~\ref{remThomabs}, pour montrer l'existence d'une structure de catégorie de modèles à la Thomason sur $\nCat{n}$, autrement dit, d'une catégorie de modèles combinatoire propre ayant comme équivalences faibles (resp. comme fibrations) les équivalences faibles de Thomason (resp. les fibrations de Thomason), il suffit de montrer les deux propriétés suivantes:
\begin{enumerate}
\item[(d${}'$)] \emph{si $E'\to E$ est un crible de $\ord$ admettant une rétraction qui est aussi un adjoint à droite, son image par le foncteur $c_nN$ est une équivalence faible de Thomason et le reste après tout cochangement de base;}
\item[(e)] \emph{pour tout ensemble ordonné $E$, le morphisme d'adjonction $N(E)\to N_nc_nN(E)$ est une équivalence faible d'ensembles simpliciaux.}
\end{enumerate}
De plus, $(c_n\Sd^2,\Ex^2N_n)$ sera alors une adjonction de Quillen
(\emph{cf.}~remarque~\ref{remThomabseq}).
\tb

Si $n=1$, comme le foncteur nerf $N=N_1$ est pleinement fidèle, le morphisme
d'adjonction $cN\to1_\cat$ est un isomorphisme. La propriété (\emph{d}${}'$)
affirme alors simplement qu'un crible $E'\to E$ de $\ord$ admettant une
rétraction qui est aussi un adjoint à droite est une équivalence faible de
Thomason dans $\cat$, et le reste après tout cochangement de base. Cela
résulte du fait qu'un tel crible est un rétracte par déformation fort, du
fait que les rétractes par déformation fort sont stables par images
directes~\cite[proposition~5.1.8]{Ci}, et du fait que les rétractes par
déformation sont des équivalences faibles de Thomason (puisque le nerf,
commutant aux produits, transforme un rétracte par déformation de $\cat$ en
un rétracte par déformation de $\simpl$).
Par ailleurs, puisque $N$ est un foncteur pleinement fidèle, le morphisme
d'adjonction $cN(E)\to E$ est un isomorphisme, et comme le composé $N(E) \to
NcN(E) \to N(E)$ est l'identité, le morphisme d'adjonction $N(E) \to
Nc(N(E))$ est aussi un isomorphisme. La propriété~(\emph{e}) est donc
immédiate.
\tb

Les deux sections suivantes seront consacrées à la preuve de ces deux
propriétés pour $n=2$. Dans ce cas, contrairement à ce qui est affirmé
dans~\cite{Wor}, le morphisme d'adjonction $N(E)\to N_2c_2(N(E))$ n'est plus
un isomorphisme (même si $E$ est l'ensemble ordonné $\smp{2}$
(\emph{cf.}~scholie~\ref{trivfaux})). Ainsi, la preuve de la propriété
(\emph{e}) est non triviale. Néanmoins, le point le plus délicat est la
démonstration de la condition~(\emph{d}${}'$). La difficulté vient du fait
que comme le foncteur $c_2$ ne commute pas aux produits binaires, l'image
par $c_2N$ d'un crible de $\ord$ admettant une rétraction qui est aussi un
adjoint à droite n'est pas en général un rétracte par déformation fort.
Comme il a déjà été observé dans~\cite{Wor}, la notion de rétracte par
déformation fort doit pour cela être assouplie en exigeant seulement de
l'homotopie intervenant dans la définition qu'elle soit un $2$-foncteur
oplax normalisée (lax normalisé dans~\cite{Wor}, où la convention opposée
est adoptée pour le sens des $2$\nbd-flèches des $2$\nbd-tronqués
$\trOrnt{m}{2}$ des orientaux de Street), par opposition à un $2$-foncteur
strict. En revanche, contrairement à ce qui est affirmé dans~\cite{Wor}, les
cribles de $2$\nbd-catégories rétractes par déformation forts en ce sens
généralisé ne semblent pas être stables par images directes. Du moins, la
preuve exposée dans~\cite{Wor} est incorrecte
(\emph{cf.}~scholie~\ref{sch:rdf_wor}). Pour contourner cette difficulté,
nous introduirons (\emph{cf.}~définition \ref{def:rdf}) une condition
supplémentaire sur un tel crible de $2$\nbd-catégories rétracte (prenant la
forme d'une compatibilité entre le crible et les contraintes de
fonctorialité de l'homotopie) et nous montrerons que cette nouvelle notion
est stable par images directes. On conclura en prouvant que l'image par
$c_2N$ d'un crible de $\ord$ admettant une rétraction qui est aussi un
adjoint à droite est un exemple d'un tel crible rétracte.
\tb

Pour $n\geq3$, les propriétés (\emph{e}) et (\emph{d}${}'$) ne sont pas
encore établies. On vient d'observer que dans le cas $n=2$, la notion de
$2$\nbd-foncteur oplax normalisé joue un rôle crucial. Or, le concept de
$n$\nbd-foncteur lax ou oplax normalisé (ou pas) n'a pas encore été
exploré. Ce sera l'objet d'un article en préparation~\cite{DG1}. De même,
pour prouver la condition~(\emph{e}), on aura besoin de montrer que dans
$\nCat{n}$, les équivalences faibles définies à l'aide du nerf de Street
coïncident avec celles définies par le nerf $n$\nbd-simplicial, généralisant
ainsi un résultat obtenu, pour $n=2$, par M.~Bullejos et
A.~M.~Cegarra~\cite{BCeg} et, pour~$n=3$, par A.~M.~Cegarra et
B.~A.~Heredia~\cite{HCeg}. Ce sera l'objet de~\cite{DG2}.
\end{sch}

\section{Le théorème de Thomason $2$-catégorique}\label{sec:2-cat}

Le but de cette section est de démontrer les conditions~($\text{\emph{d}}'$)
et ($\emph{e}$) du scholie~\ref{remnThom} pour $n = 2$ (modulo des questions
$2$\nbd-catégoriques qui seront traitées dans la section suivante), ce qui
achèvera la preuve du théorème de Thomason $2$-catégorique.

\begin{paragr}\label{or2tronq}
Les $2$\nbd-tronqués des orientaux de Street admettent une description
particulièrement simple. Pour tout entier $m\geq0$, la $2$\nbd-catégorie
$\trOrnt{m}{2}$ s'identifie à la $2$\nbd-catégorie dont les objets sont les
entiers $0,1,\dots,m$ et dont la catégorie des flèches de $k$ vers~$l$,
$0\leq k,\,l\leq m$, est l'ensemble ordonné par inclusion des parties $S$ de
$\{0,1,\dots,m\}$ telles que $\min(S)=k$ et $\max(S)=l$. En particulier,
cette catégorie est vide si~$l<k$ et est la catégorie ponctuelle si $k=l$.
Ainsi, l'ensemble des $1$\nbd-flèches de $\trOrnt{m}{2}$ s'identifie à
l'ensemble des parties non vides de $\{0,1,\dots,m\}$, la composition
correspondant à la réunion. Il faut se garder de confondre cette bijection
entre l'ensemble de \emph{toutes} les $1$\nbd-flèches de $\trOrnt{m}{2}$
(d'ailleurs égal à l'ensemble des $1$\nbd-flèches de $\Ornt{m}$) et
l'ensemble des parties non vides de $\{0,1,\dots,m\}$, et la bijection entre
ce dernier ensemble et l'ensemble des atomes de $\Ornt{m}$, formé des objets
et des $i$\nbd-flèches \emph{indécomposables}, $1\leq i\leq m$,
de~$\Ornt{m}$ (\emph{cf.}~\ref{nerfdeStreet}).
\tb

Soit $\varphi:\smp{p}\to\smp{q}$ un morphisme de $\cats$. Le $2$\nbd-foncteur $i_2(\varphi):\trOrnt{p}{2}\to\trOrnt{q}{2}$ associe à un objet $k$, $0\leq k\leq p$, de $\trOrnt{p}{2}$ l'objet $\varphi(k)$ de $\trOrnt{q}{2}$ et à une $1$\nbd-flèche $S\subset\{0,1,\dots,p\}$ de $\trOrnt{p}{2}$ la $1$\nbd-flèche $\varphi(S)$ de $\trOrnt{q}{2}$.
\end{paragr}

\begin{paragr}\label{orntg}
Plus généralement, si $E$ est un ensemble ordonné, on note $\Orntg{E}$ la $2$\nbd-catégorie définie comme suit. Les objets de $\Orntg{E}$ sont les éléments de $E$. Les $1$\nbd-flèches sont les parties finies, totalement ordonnées, non vides de $E$, autrement dit, les éléments de~$\xi(E)$ (\emph{cf.}~\ref{xi}). Si $S\in\xi(E)$ est une $1$\nbd-flèche de $\Orntg{E}$, la source (resp. le but) de~$S$ est le minimum (resp. le maximum) de $S$. Si $S$ et $S'$ sont deux $1$\nbd-flèches composables de $\Orntg{E}$, leur composé est égal à la réunion de $S$ et $S'$ (qui est encore une partie finie, totalement ordonnée, non vide de $E$). Si $S$ et $S'$ sont deux $1$\nbd-flèches parallèles de $\Orntg{E}$ (c'est-à-dire telles que $\min(S)=\min(S')$ et $\max(S)=\max(S')$), l'ensemble des $2$\nbd-flèches de $\Orntg{E}$, de source $S$ et but $S'$, est un singleton si $S\subset S'$ et est vide sinon. Autrement dit, pour tout couple $x,y\in E$ d'objets de $\Orntg{E}$, la catégorie $\sHom_{\Orntg{E}}(x,y)$ est l'ensemble ordonné par inclusion des partie finies, totalement ordonnées, $S$ de $E$ telles que $\min(S)=x$ et $\max(S)=y$. On remarque que pour tout entier $m\geq0$, on a $\trOrnt{m}{2}=\Orntg{\smp{m}}$.
\tb

Soit $\varphi:E\to E'$ une application croissante entre ensembles ordonnés. On définit un $2$\nbd-foncteur $\Orntg{\varphi}:\Orntg{E}\to\Orntg{E'}$ en associant à tout objet $x\in E$ de $\Orntg{E}$, l'objet $\varphi(x)\in E'$ de $\Orntg{E'}$ et à toute $1$\nbd-flèche $S\in\xi(E)$ de $\Orntg{E}$, la $1$\nbd-flèche $\varphi(S)\in\xi(E')$ de~$\Orntg{E'}$ (en remarquant que pour tout couple $x,y$ d'objets de $\Orntg{E}$, on obtient ainsi une application croissante $\sHom_{\Orntg{E}}(x,y)\to\sHom_{\Orntg{E'}}(\varphi(x),\varphi(y))$). On remarque que si $E=\smp{p}$ et $E'=\smp{q}$, alors on a $\Orntg{\varphi}=i_2(\varphi)$. On définit ainsi un foncteur $\Orntgfonct:\ord\to\nCat{2}$, qui prolonge le foncteur $i_2$.
\end{paragr}

\begin{lemme}\label{limindorntg}
Soit $E$ un ensemble ordonné. Alors le morphisme canonique
\[ \limind_{S\in\xi E}\Orntg{S}\to \Orntg{E}\ , \]
induit par les inclusions $S\to E$, est un isomorphisme de $2$\nbd-catégories.
\end{lemme}

\begin{proof}
Pour tout $S\in\xi(E)$, on note $\e^{}_S$ l'inclusion $S\to E$ et pour tous $S,S'\in\xi(E)$, $S'\subset S$, on note $\e^{}_{S,S'}$ l'inclusion $S'\to S$. Il s'agit de montrer que pour toute $2$\nbd-catégorie $C$, et toute famille de $2$\nbd-foncteurs $u^{}_S:\Orntg{S}\to C$, $S\in\xi(E)$, telle que pour tout couple $S,S'\in\xi(E)$, $S'\subset S$, on ait $u^{}_{S'}=u^{}_S\,\Orntg{\e^{}_{S,S'}}$, il existe un unique $2$\nbd-foncteur $u:\Orntg{E}\to C$ tel que pour tout $S\in\xi(E)$, on ait 
\begin{equation}\label{limindorntg1}
u\,\Orntg{\e^{}_S}=u^{}_S\ .
\end{equation}
\textsc{Unicité.} Soit $x\in E$ un objet de $\Orntg{E}$. La relation~\ref{limindorntg1}, appliquée à $S=\{x\}$ et à l'objet $x$ de $\Orntg{\{x\}}$, implique que
\begin{equation}\label{limindorntg2}
u(x)=u^{}_{\{x\}}(x)\ .
\end{equation}
Soit $S\in \xi(E)$ une $1$\nbd-flèche de $\Orntg{E}$. La relation~\ref{limindorntg1}, appliquée à la $1$\nbd-flèche $S$ de~$\Orntg{S}$, implique que
\begin{equation}\label{limindorntg3}
u(S)=u^{}_S(S)\ .
\end{equation}
Soit $S'\subset S$, $\min(S)=\min(S')$, $\max(S)=\max(S')$, une $2$\nbd-flèche de $\Orntg{E}$. La relation~\ref{limindorntg1}, appliquée à la $2$\nbd-flèche $S'\subset S$ de $\Orntg{S}$, implique que
\begin{equation}\label{limindorntg4}
u(\,S'\subset S\,)=u^{}_S(\,S'\subset S\,)\ .
\end{equation}
\textsc{Existence.} On laisse au lecteur la longue et fastidieuse, mais triviale, vérification du fait que les formules~\ref{limindorntg2}, \ref{limindorntg3} et~\ref{limindorntg4} définissent un $2$\nbd-foncteur, et que ce $2$\nbd-foncteur satisfait à la relation~\ref{limindorntg1}.
\end{proof}

\begin{prop}\label{orntgadj}
Pour tout ensemble ordonné $E$, les $2$\nbd-catégories $\Orntg{E}$ et
$c_2N(E)=c_2N_2(E)$ sont canoniquement isomorphes. De plus, le morphisme
d'adjonction $c_2N_2(E)\to E$ s'identifie à travers cet isomorphisme
canonique à l'unique $2$\nbd-foncteur $\Orntg{E}\to E$ induisant l'identité
sur les objets. 
\end{prop}

\begin{proof}
Vu que les foncteurs $\Orntgfonct$ et $c_2N_2$ coïncident par définition sur les~$\smp{m}$, $m\geq0$, la proposition résulte des lemmes~\ref{limindnerf} et~\ref{limindorntg} et du fait que le foncteur $c_2$ commute aux limites inductives (en tant qu'adjoint à gauche). Un $2$\nbd-foncteur $\Orntg{E}\to E$ qui est l'identité sur les objets  associe forcément à une $1$\nbd-flèche $S\in\xi(E)$ de $\Orntg{E}$ l'unique flèche de $E$ de source $\min(S)$ et but $\max(S)$ et à une $2$\nbd-flèche de~$\Orntg{E}$ l'identité de l'image de sa source (qui est également l'image de son but). Un tel $2$\nbd-foncteur est donc unique. Pour montrer la deuxième assertion, il suffit donc de prouver que le morphisme d'adjonction $c_2N_2(E)\to E$ induit l'identité sur les objets. Cela résulte de la fonctorialité de ce morphisme, appliquée à l'inclusion $\{x\}\to E$, pour $x\in E$.
\end{proof}

\begin{paragr}\label{nerfbisimpl}
Soit $C$ une petite $2$\nbd-catégorie. Le \emph{nerf bisimplicial} de $C$ est l'ensemble bisimplicial dont les $(p,q)$\nbd-simplexes sont les $q$\nbd-simplexes du nerf de la catégorie 
$$\textstyle\coprod\limits_{x_0,\dots,x_p\in\ob(C)}\sHom_C(x_0,x_1)\times\cdots\times\sHom_C(x_{p-1},x_p)\ .$$
On vérifie facilement qu'on définit ainsi un foncteur $\Nbisimpl:\nCat{2}\to\pref{\cats\times\cats}$.
\end{paragr}

\begin{paragr}\label{eqfbisimpl}
On note $\delta:\cats\to\cats\times\cats$ le foncteur diagonal et $\delta^*:\pref{\cats\times\cats}\to\simpl$ le foncteur image inverse déduit de $\delta$.
Les \emph{équivalences faibles diagonales} de $\pref{\cats\times\cats}$ sont les flèches de $\pref{\cats\times\cats}$ dont l'image par $\delta^*$ est une équivalence faible simpliciale. On a le lemme classique suivant dû à Quillen~\cite{QuK} (pour une preuve purement simpliciale, voir~\cite[lemme 5.3.1]{Ho}):
\end{paragr}

\begin{lemme}\label{lemmebisimpl}
Les équivalences faibles argument par argument de $\pref{\cats\times\cats}$ sont des équivalences faibles diagonales. Autrement dit, pour qu'un morphisme d'ensembles bisimpliciaux $X_{\bullet,\bullet}\to Y_{\bullet,\bullet}$ soit une équivalence faible diagonale, il suffit que pour tout $p\geq0$, $X_{p,\bullet}\to Y_{p,\bullet}$ soit une équivalence faible simpliciale.
\end{lemme}

\begin{thm}\label{Cegarra}
Une flèche de $\nCat{2}$ est une équivalence faible de Thomason si et seulement si son nerf bisimplicial est une équivalence faible diagonale.
\end{thm}

\begin{proof}
C'est une conséquence immédiate de~\cite[théorème 1.1]{BCeg}.
\end{proof}

\noindent
Le lemme suivant est classique (voir \cite[lemme 2.3]{BCF}  ou
\cite[proposition 3.3]{Hoy}) :

\begin{lemme}\label{delHoyo}
Soient $C,D$ deux petites $2$\nbd-catégories, et $u:C\to D$ un $2$\nbd-foncteur induisant une bijection des ensembles des objets. Si pour tout couple d'objets $x,y$ de~$C$, le foncteur $\sHom_C(x,y)\to\sHom_D(u(x),u(y))$, induit par $u$, est une équivalence faible de Thomason, alors $u$ est une équivalence faible de Thomason.
\end{lemme}

\begin{proof}
Comme le foncteur nerf $N:\cat\to\simpl$ commute aux sommes disjointes et aux produits, et comme les équivalences faibles simpliciales sont stables par produits finis et sommes, l'hypothèse implique que pour tout $p\geq0$, le morphisme $(\Nbisimpl(C))_{p,\bullet}\to(\Nbisimpl(D))_{p,\bullet}$ est une équivalence faible simpliciale. Le lemme~\ref{lemmebisimpl} implique donc que le morphisme d'ensembles bisimpliciaux $\Nbisimpl(u)$ est une équivalence faible diagonale. Le lemme résulte alors du théorème précédent.
\end{proof}

\begin{prop}\label{prop:c_2N_2_eq_thom}
Pour tout ensemble ordonné $E$, le morphisme d'adjonction $c_2N_2(E)\to E$ est une équivalence faible de Thomason.
\end{prop}

\begin{proof}
En vertu de la proposition~\ref{orntgadj}, ce morphisme d'adjonction s'identifie à l'unique $2$\nbd-foncteur $\Orntg{E}\to E$ induisant l'identité sur les objets. Or, pour tout couple $x,y\in E$, le foncteur $\sHom_{\Orntg{E}}(x,y)\to\sHom_E(x,y)$, induit par ce $2$\nbd-foncteur, est une équivalence faible de Thomason. En effet, si $x\leq y$, alors $\sHom_E(x,y)$ est la catégorie ponctuelle et $\sHom_{\Orntg{E}}(x,y)$ est l'ensemble ordonné par inclusion des parties finies, totalement ordonnées $S$ de $E$ telles que $\min(S)=x$ et $\max(S)=y$. Cet ensemble ordonné admet un plus petit élément, à savoir l'ensemble $\{x,y\}$, et son nerf est donc contractile. Si $x\not\leq y$, les catégories $\sHom_{\Orntg{E}}(x,y)$ et $\sHom_E(x,y)$ sont vides, ce qui prouve l'assertion. La proposition résulte alors du lemme précédent.
\end{proof}

\begin{cor}\label{conditione}
Pour tout ensemble ordonné $E$, le morphisme d'adjonction $N(E)\to N_2c_2(N(E))$ est une équivalence faible d'ensembles simpliciaux.
\end{cor}

\begin{proof}
On a $N(E)=N_2(E)$, et comme dans l'égalité du triangle
$$
\xymatrix{
N_2(E)\ar[r]\ar@/_2em/[rr]_{1_{N_2(E)}}
&N_2c_2N_2(E)\ar[r]
&N_2(E)\ \ 
}
$$
la flèche de droite est, en vertu de la proposition précédente, une équivalence faible d'ensembles simpliciaux, il en est de même de la flèche de gauche.
\end{proof}

\begin{sch}\label{trivfaux}
En revanche, contrairement à ce qui est affirmé dans~\cite{Wor}, le morphisme d'adjonction $N(E)\to N_2c_2(N(E))$ n'est en général \emph{pas} un isomorphisme. Par exemple, si $E=\smp{2}$, alors $c_2(N(E))$ est par définition la $2$\nbd-catégorie $\trOrnt{2}{2}=\Ornt{2}$
$$
\xymatrixrowsep{.3pc}
\xymatrix{
&1\ar[dddddr]^{u^{}_{21}\kern 50pt\textstyle{u_{20}\neq u_{21}\circ u_{10}}}
\\
\\
&
\\
&
\\
&\ar@{=>}[uu]^{\alpha}
\\
0\ar[uuuuur]^{u^{}_{10}}\ar[rr]_{u^{}_{20}}
&&2
}\kern 25pt
$$
et les $1$\nbd-simplexes non dégénérés de $N_2c_2(N(E))$ sont les
$2$\nbd-foncteurs non constants de $\Ornt{1}=\smp{1}$ vers $\Ornt{2}$. Il y
en a quatre, correspondant aux $1$\nbd-flèches de $\Ornt{2}$ qui ne sont pas
des identités, à savoir $u^{}_{10},\,u^{}_{21},\,u^{}_{20}$ et
$u^{}_{21}\circ u^{}_{10}$. Ainsi, $N_2c_2(N(E))$ n'est \emph{pas} isomorphe
à $N(E)=\smp{2}$ qui n'a que trois $1$\nbd-simplexes non dégénérés.
\end{sch}

\begin{definition}\label{def:fonct_lax}
Un \ndef{$2$-foncteur oplax normalisé} $F$ d'une $2$-catégorie $C$ vers une
$2$-catégorie $D$ est donné par
\begin{itemize}
  \item 
   pour tout objet $x$ de $C$, un objet $F(x)$ de $D$ ;
  \item pour toute $1$-flèche $f : x \to y$ de $C$, une
    $1$-flèche $F(f) : F(x) \to F(y)$ de $D$ ;
  \item pour toute $2$-flèche $\alpha : f \Rightarrow g$ de $C$, une
    $2$-flèche $F(\alpha) : F(f) \Rightarrow F(g)$ de $D$ ;
  \item pour tout couple $x \xrightarrow{f} y \xrightarrow{g} z$ de
    $1$-flèches composables de $C$, une \ndef{contrainte de fonctorialité},
    c'est-à-dire une $2$-flèche
    \[ F(g, f):F(gf)\Rightarrow F(g)F(f) \]
    de $D$,
\end{itemize}
satisfaisant aux conditions suivantes :
\begin{itemize}
  \item \ndef{conditions de normalisation} :
    \begin{itemize}
      \item pour tout objet $x$ de $C$, on a $F(\id{x}) = \id{F(x)}$ ;
      \item pour toute $1$-flèche $f : x \to y$ de $C$, on a
        $F(\id{y}, f) = \id{f} = F(f, \id{x})$ ;
    \end{itemize}
  \item \ndef{condition de cocycle} : pour tout triplet
    $x \xrightarrow{f} y \xrightarrow{g} z \xrightarrow{h}t $ de $1$-flèches
    composables de $C$, on a
    \[
    (F(h) \comp F(g, f))F(h, gf) = (F(h, g) \comp F(f))F(hg, f)\ ;
    \]
  \item \ndef{compatibilité à la composition verticale} : pour tout couple
    $f \xRightarrow{\alpha\ } g \xRightarrow{\beta\,} h$ de $2$\nbd-flèches
    composables verticalement de $C$, on a $F(\beta\alpha) =
    F(\beta)F(\alpha)$ ;
  \item \ndef{compatibilité aux identités des $1$-flèches} : pour toute $1$-flèche $f$ de $C$, on a \hbox{$F(\id{f}) = \id{F(f)}$} ;
  \item \ndef{compatibilité à la composition horizontale} :
    pour tout couple
    \[
\UseAllTwocells
\xymatrix@C=3pc{
x \rtwocell^f_g{\,\alpha}
&
y \rtwocell^k_l{\,\,\beta}
&
z
}
    \]
    de $2$-flèches composables horizontalement de $C$, on a
    \[
    F(l, g)F(\beta \comp \alpha) = (F(\beta) \comp F(\alpha))F(k, f)\ .
    \]
\end{itemize}

On notera $\dcatnorm$ la catégorie dont les objets sont les $2$-catégories
strictes et les morphismes sont les $2$-foncteurs oplax normalisés entre
celles-ci. Tout $2$-foncteur strict étant canoniquement un $2$-foncteur
oplax normalisé, on dispose d'un foncteur canonique de $\dcat$
vers~$\dcatnorm$, identifiant $\dcat$ à une sous-catégorie non pleine de
$\dcatnorm$.
On vérifie facilement que ce foncteur d'inclusion commute aux petites limites projectives.
\tb

Dans la suite, on omettra systématiquement l'adjectif « normalisé ». En
particulier, on appellera foncteur oplax un foncteur oplax normalisé. De
même, ci-dessous, on appellera homotopie oplax (resp. déformation oplax) ce
qu'on devrait appeler homotopie oplax normalisée (resp. déformation oplax
normalisée).
\end{definition}

\begin{definition}
Soient $A$ une $2$-catégorie et $\eps = 0, 1$. On notera $\DeltaeX{\eps}{A}$
le $2$\nbd-foncteur $\Deltae{\eps} \times \id{A} : A \simeq e \times A \to
\Deltan1 \times A$. 
(Comme le foncteur d'inclusion de $\dcat$ dans $\dcatnorm$ commute aux
limites projectives, le produit $\Deltan1 \times A$ dans $\dcat$ est aussi
un produit dans $\dcatnorm$.)
\tb

Soient $F, G : A \to B$ deux $2$-foncteurs oplax (le cas qui nous intéresse
vraiment est celui où $F$ et $G$ sont stricts). Une \ndef{homotopie oplax}
de $F$ vers $G$ est un $2$-foncteur oplax $H : \Deltan1
\times A \to A$ satisfaisant
\[ H\DeltaeX{0}{A} = F \quad\text{et}\quad H\DeltaeX{1}{A} = G\ . \]

On dira que deux tels $2$-foncteurs $F$ et $G$ sont \ndef{élémentairement
oplax homotopes} s'il existe une homotopie oplax de $F$ vers $G$. On
dira que $F$ et $G$ sont \ndef{oplax homotopes} s'ils sont dans la même
classe d'équivalence pour la relation d'équivalence engendrée par la
relation «~être élémentairement oplax homotopes~». Autrement dit, $F$
et $G$ sont oplax homotopes s'ils sont reliés par un zigzag
d'homotopies oplax.
\end{definition}

\begin{definition}\label{def:rdf}
Soit $i : A \to B$ un crible de $2$-catégories (qu'on considèrera comme une
inclusion). Une \ndef{structure de rétracte par déformation oplax fort}
sur $i$ est donnée par
\begin{itemize}
  \item une rétraction $r : B \to A$ de $i$ (c'est-à-dire un $2$\nbd-foncteur strict tel que $ri=1_A$) ;
  \item une homotopie oplax $H$ de $ir$ vers $\id{B}$,
\end{itemize}
satisfaisant aux conditions suivantes :
\begin{enumerate}[label=RDF\arabic*), ref=RDF\arabic*, labelindent=\parindent, leftmargin=*]
  \item l'homotopie $H$ est \ndef{relative à $A$} au sens où on a
    \[ H(1_{\Deltan1} \times i) = ip^{}_A\ , \]
    où $p^{}_A : \Deltan1 \times A \to A$ désigne la deuxième projection ;
    \label{item:RDF1}
  \item pour tout couple \hbox{$x \xrightarrow{f} y \xrightarrow{g} z$} de
    $1$-flèches composables de $B$ avec $f$ dans~$A$, on a 
     \[ H(0 \to 1, gf) = H(0 \to 1, g)H(0, f)\ , \]
    et la contrainte de fonctorialité correspondante est triviale.
    \label{item:RDF2}
\end{enumerate}

On dira qu'un $2$-foncteur $i : A \to B$ est un \ndef{crible rétracte par
déformation oplax fort} si $i$ est un crible sur lequel il existe une
structure de rétracte par déformation oplax fort. On appellera également un
tel $2$-foncteur un \ndef{$2$-crible rétracte par déformation oplax fort}.
\end{definition}

\begin{paragr}\label{paragr:Ntilde}
Notons $\tilde{i}_2 : \cDelta \to \dcatnorm$ le foncteur obtenu en composant
les inclusions canoniques
\[
\cDelta \to \Cat \to \dcat \to \dcatnorm\ .
\]
On définit un foncteur $\wt{N}_2:\dcatnorm\to\simpl$ par la formule
\[
C\longmapsto\bigl(\smp{m}\mapsto\Hom_{\dcatnorm}(\tilde{i}_2(\smp{m}),C)\bigr)\ .
\]
Il est immédiat que ce foncteur commute aux limites projectives.
\tb

Si $C$ est une $2$-catégorie, on vérifie facilement (voir l'avant-dernier
exercice de~\cite{StreetHandbook}) que la donnée d'un $2$-foncteur oplax
(normalisé) de $\tilde{i}_2(\Deltan{m}) = \Deltan{m}$ vers $C$ est
naturellement équivalente à celle d'un $2$-foncteur strict de
$i_2(\Deltan{m}) = \trOrnt{m}{2}$ (\emph{cf.}~\ref{trnerfdeStreet})
vers~$C$. On en déduit qu'on a
\[ \wt{N}_2\,|\,\dcat = N_2\ . \]
\end{paragr}

\begin{prop}
Soient $F$ et $G$ deux $2$-foncteurs oplax qui sont oplax homotopes.  Alors
les morphismes d'ensembles simpliciaux $\wt{N}_2(F)$ et $\wt{N}_2(G)$ sont
homotopes (au sens où ils sont reliés par un zigzag d'homotopies
simpliciales).
\end{prop}

\begin{proof}
Il suffit de traiter le cas où $F$ et $G$ sont élémentairement oplax
homotopes. Vu que $\wt{N}_2(\Deltan1)\simeq\Deltan1$, ce cas résulte immédiatement du fait que $\wt{N}_2$ commute aux
produits et envoie donc une homotopie oplax sur une homotopie simpliciale.
\end{proof}

\begin{cor}\label{coro:eq_hom_faible}
Soit $F : A \to B$ un $2$-foncteur oplax. Supposons qu'il existe un
$2$-foncteur oplax $G$ tel que $GF$ et $FG$ soient oplax homotopes à
$\id{A}$ et $\id{B}$ respectivement. Alors $\wt{N}_2(F)$ est une équivalence
faible d'ensembles simpliciaux. En particulier, si le $2$-foncteur $F$ est
strict, alors $F$ est une équivalence faible de Thomason
\emph{(cf.~\ref{eqfThom})}.
\end{cor}

\begin{proof}
La proposition précédente implique immédiatement que $\wt{N}_2(F)$ est
inversible d'inverse $\wt{N}_2(G)$ dans la catégorie homotopique des
ensembles simpliciaux, et donc que $\wt{N}_2(F)$ est une équivalence faible
d'ensembles simpliciaux.
Le second point résulte du fait que si $F$ est strict, alors $N_2(F) =
\wt{N}_2(F)$.
\end{proof}

\begin{prop}\label{prop:2rdf_thom}
Les $2$-cribles rétractes par déformation oplax forts sont des équivalences
faibles de Thomason.
\end{prop}

\begin{proof}
C'est un cas particulier du corollaire \ref{coro:eq_hom_faible}.
\end{proof}

\begin{prop}\label{prop:2rdf_img_dir}
Les $2$-cribles rétractes par déformation oplax forts sont stables par
images directes dans $\dcat$. Plus précisément, si $i:A\to B$ est un crible de $2$\nbd-catégories et 
$$
\xymatrix{
&A\ar[r]^u\ar[d]_i
&A'\ar[d]^{i'}
\\
&B\ar[r]_{v}
&B'
&
}
$$
est un carré cocartésien de $\dcat$, alors pour toute structure de rétracte
par déformation oplax fort~$(r,H)$ sur~$i$, il existe une unique structure
de rétracte par déformation oplax fort~$(r',H')$ sur $i'$ compatible à $(r,
H)$ au sens où on a
$$r'v=ur\quad\hbox{et}\quad H'(1_{\smp{1}}\times v)=vH\ .$$
\end{prop}

\begin{proof}
La section \ref{sec:2rdf}, entièrement indépendante des autres résultats de l'article, sera consacrée à la preuve de cette proposition. 
\end{proof}

\begin{rem}\label{rem:rdf_wor}
L'exemple suivant montre que la condition \eqref{item:RDF2} de la définition de la
notion de $2$\nbd-crible rétracte par déformation oplax fort est essentielle
pour la véracité de la proposition précédente (du moins dans sa forme
précise).

\tb

Considérons le crible de $2$\nbd-catégories
$$
A=
\left\{
\raise 22pt
\vbox{
\xymatrixrowsep{.7pc}
\xymatrix{
x\ar[rd]^f
\\
&z
\\
y\ar[ru]_g
}
}
\right\}
\xymatrix{\ar[r]^i&}
\left\{
\raise 22pt
\vbox{
\xymatrixrowsep{.7pc}
\xymatrix{
x\ar[rd]^f\ar@/^3ex/[rrd]^p\drrtwocell<\omit>{<-.5>\alpha}
\\
&z\ar[r]^h
&t\dlltwocell<\omit>{<-.5>\kern 25pt\beta\vrule depth 6pt width 0pt}
\\
y\ar[ru]_g\ar@/_3ex/[rru]_q
}
}
\right\} = B
$$
(où $p\neq hf$ et $q\neq hg$), le cocrible
$$
A=
\left\{
\raise 22pt
\vbox{
\xymatrixrowsep{.7pc}
\xymatrix{
x\ar[rd]^f
\\
&z
\\
y\ar[ru]_g
}
}
\right\}
\xymatrix{\ar[r]^u&}
\left\{
\raise 22pt
\vbox{
\xymatrixrowsep{.7pc}
\xymatrix{
&x\ar[rd]^f
\\
a\ar[ru]^k\ar[rd]_l
&&z
\\
&y\ar[ru]_g
}
}
\right\}=A'\ ,
$$
où $A'$ est la catégorie obtenue à partir de $A$ par adjonction formelle d'un objet initial~$a$ (de sorte que $fk=g\,l$), et le crible image directe de $i$ le long de $u$
$$
A'=
\left\{
\raise 22pt
\vbox{
\xymatrixrowsep{.7pc}
\xymatrix{
&x\ar[rd]^f
\\
a\ar[ru]^k\ar[rd]_l
&&z
\\
&y\ar[ru]_g
}
}
\right\}
\xymatrix{\ar[r]^{i'}&}
\left\{
\raise 22pt
\vbox{
\xymatrixrowsep{.7pc}
\xymatrix{
&x\ar[rd]^f\ar@/^3ex/[rrd]^p\drrtwocell<\omit>{<-.5>\alpha}
\\
a\ar[ru]^k\ar[rd]_l
&&z\ar[r]^h
&t\dlltwocell<\omit>{<-.5>\kern 25pt\beta\vrule depth 6pt width 0pt}
\\
&y\ar[ru]_g\ar@/_3ex/[rru]_q
}
}
\right\} = B'\ .
$$
On a donc un carré cocartésien dans $\nCat{2}$
$$
\xymatrix{
&A\ar[r]^u\ar[d]_i
&A'\ar[d]^{i'}
\\
&B\ar[r]_{v}
&B'
&\kern-16pt.\kern15pt
}
$$
Le crible $i$ admet une unique rétraction $r:B\to A$. Cette rétraction envoie $t$ sur $z$, $h$~sur $1_z$, $p$ et $hf$ sur $f$, $q$ et $hg$ sur $g$, $\alpha$ sur $1_f$ et $\beta$ sur $1_g$. De même, le crible $i'$  admet une unique rétraction $r':B'\to A'$ et cette dernière satisfait à la relation $rv=ur$.
\tb

Il existe plusieurs homotopies oplax relatives à $A$ de $ir$ vers $1_B$.
Comme les restrictions d'une telle homotopie sur $\{0\}\times B$,
$\{1\}\times B$ et $\Deltan1\times A$ sont imposées et comme $B$ est une
catégorie enrichie en ensembles ordonnés, pour définir une telle
homotopie~$H$, il suffit de la définir sur les $1$\nbd-flèches de
$\Deltan1\times B$ de la forme $(0\to1,w)$, avec $w$ une $1$-flèche de $B$
qui n'est pas dans $A$ (c'est-à-dire $w\in\{1_t,h,p,q,hf,hg\}$). La
compatibilité aux sources et buts implique que 
$$
H(0\to1,1_t) = h \quad\hbox{et}\quad H(0\to1,h)=h\ ,
$$ 
et celle à la contrainte de fonctorialité relative aux couples de morphismes
composables
\[
(0, x) \xrightarrow{(0 \rightarrow 1, 1_x)} 
(1, x) \xrightarrow{\hpzz(1, p)\hpzz}
(1, t)
\quad\text{et}\quad
(0, y) \xrightarrow{(0 \rightarrow 1, 1_y)}
(1, y) \xrightarrow{\hpzz(1, q)\hpzz}
(1, t)
\]
que
$$H(0\to1,p)=p\quad\hbox{et}\quad H(0\to1,q)=q\ .$$ 
Enfin, à nouveau, la compatibilité aux sources et buts implique que
$$H(0\to1,hf)=hf\ \hbox{ou}\ p\qquad\hbox{et}\qquad H(0\to1,hg)=hg\ \hbox{ou}\ q\ .$$
On vérifie facilement que chacun de ces quatre choix définit une homotopie oplax relative à $A$ de $ir$ vers $1_B$, et que parmi ces quatre homotopies, seule celle qui est définie par
$$H(0\to1,hf)=hf\qquad\hbox{et}\qquad H(0\to1,hg)=hg\ $$
satisfait à la condition \eqref{item:RDF2}.
\tb

On va montrer que celle qui est définie par 
$$H(0\to1,hf)=p\qquad\hbox{et}\qquad H(0\to1,hg)=q\ $$
ne se prolonge \emph{pas} en une homotopie oplax $H'$ relative à $A'$ de $i'r'$ vers $1_{B'}$. 
En effet, si $H'$ était une telle homotopie, on aurait, entre autres, des
$2$\nbd-flèches de contrainte de fonctorialité
\[
\begin{aligned}
H'(0 \to 1, hfk) & \Rightarrow 
H'(0 \to 1, hf)\,H'(0, k)
=vH(0 \to 1, hf)\,i'r'(k)=pk\ , \\
H'(0 \to 1, h\,g\,l) & \Rightarrow 
H'(0 \to 1, h\,g)\,H'(0, l)
=vH(0 \to 1, h\,g)\,i'r'(l)=ql\ .
\end{aligned}
\]
Étant donné que $hfk=h\,g\,l$, il faudrait donc qu'il existe deux
$2$-flèches de $B'$ de même source et de buts respectifs $pk$ et $ql$. Or il
n'existe pas de telles $2$-flèches dans~$B'$, ce qui prouve qu'un telle
homotopie $H'$ n'existe pas.
\end{rem}

\begin{sch}\label{sch:rdf_wor}
Dans \cite{Wor}, les auteurs introduisent une notion de «~\forlang{skew
immersion}~» dans le but d'obtenir une variante de notre
condition~($\text{\emph{d}}'$) (ou plus précisément de la conséquence que
nous tirons de cette condition dans la preuve du théorème~\ref{Thomabs} en
utilisant la proposition~\ref{lemmeDwyer}). Nous traduirons ce terme par
«~immersion tordue~». Une \ndef{immersion tordue} est un crible $j:A\to C$
de $\dcat$ admettant une décomposition
$$
\xymatrix{
A\ar[r]^i\ar@/_1.5em/[rr]_{j}
&B\ar[r]^k
&C\ \ 
}\ \ ,
$$
où $k$ est un cocrible et $i$ un crible admettant une rétraction $r$, telle qu'il existe une homotopie oplax (lax dans~\cite{Wor}) relative à $A$ de $ir$ vers $1_B$ (mais sans qu'on demande que cette homotopie satisfasse à la condition \eqref{item:RDF2}).
\tb

Les auteurs de \cite{Wor} affirment (proposition~4.4.1) que cette notion
d'immersion tordue est stable par images directes dans $\dcat$. Pour le
montrer, ils considèrent un $2$-foncteur $u:A\to A'$, et ils forment le
diagramme de carrés cocartésiens 
$$
\xymatrix{
&A\ar[r]^u\ar[d]^i\ar@/^-1.3pc/[dd]_j
&A'\ar[d]_{i'\kern -3pt}\ar@/^1.3pc/[dd]^{j'}
\\
&B\ar[r]^{v}\ar[d]^k
&B'\ar[d]_{k'\kern -3pt}
\\
&C\ar[r]_{w}
&C'
&\hskip -15pt.\hskip 15pt
}
$$
Comme les cocribles sont stables par images directes, ils expliquent que
pour prouver leur affirmation, il suffit de construire une rétraction $r'$
de $i'$, et une homotopie oplax~$H'$ relative à $A'$ de $i'r'$ vers
$1_{B'}$. La construction de la rétraction $r'$ ne pose aucun problème; elle
résulte de la propriété universelle de la somme amalgamée $B'$. En revanche,
ils définissent l'homotopie $H'$ par des formules explicites qui impliquent
que $H'(1_{\smp{1}}\times v)= vH$. L'exemple de la remarque précédente
montre pourtant que cela n'est pas toujours possible. De fait, si $H$
désigne l'homotopie définie à la fin de la remarque précédente, un examen de
la preuve de la proposition~4.4.1 de \cite{Wor} révèle que, parmi les
équations définissant leur homotopie $H'$, celle de la ligne~10, page~226,
implique les égalités
\[
\begin{aligned}
H'(0 \to 1,hfk) & = 
H(0 \to 1, hf) \, k 
= pk\ ,
\\
H'(0 \to 1, h\,g\,l) & = 
H(0 \to 1, h\,g) \, l
= ql\ .
\end{aligned}
\]
Or celles-ci sont contradictoires, puisque $hfk=h\,g\,l$, mais $pk\neq ql$. 

\tb

Bien que cet exemple montre que la preuve de la proposition~4.4.1
de~\cite{Wor} est incorrecte, il ne dément pas la stabilité des immersions
tordues par images directes. En effet, il a été observé dans la remarque
précédente que dans cet exemple, il existe une autre homotopie oplax
relative à $A$ de $ir$ vers $1_A$ qui, elle, satisfait à la
condition~\eqref{item:RDF2}, et qui se prolonge donc, en vertu de la
proposition~\ref{prop:2rdf_img_dir}, en une homotopie oplax relative à $A'$
de $i'r'$ vers $1_B'$. Il rend néanmoins cette stabilité très douteuse.
\end{sch}

\begin{cor}\label{coro:2rdf_couniv}
Les $2$-cribles rétractes par déformation oplax forts sont des équivalences
faibles de Thomason et le restent après tout cochangement de base. 
\end{cor}

\begin{proof}
C'est une conséquence immédiate de la proposition \ref{prop:2rdf_thom} et de
la proposition précédente.
\end{proof}

\begin{paragr}
Soit $i : E'\to E$ un crible de $\ord$ admettant une rétraction $r:E\to E'$
qui est aussi un adjoint à droite. Notons que la condition d'adjonction
signifie exactement que pour tout $x$ dans $E$, on a $ir(x) \leq x$.
\tb

Considérons le $2$-foncteur $c_2N(i)$. En vertu de la proposition~\ref{prop:cN_abc},
ce $2$-foncteur est un crible. Il admet de plus $c_2N(r)$ comme rétraction.
Nous allons maintenant définir une homotopie oplax $H$ de
$c_2N(ir) = c_2N(i)c_2N(r)$ vers $\id{c_2NE}$. La $2$-catégorie $c_2NE$
étant en fait une catégorie enrichie en ensembles ordonnés (\emph{cf.}~proposition~\ref{orntgadj} et paragraphe~\ref{orntg}), les données pour
définir une telle homotopie se réduisent à la donnée de l'action de $H$ sur
les objets et les $1$-flèches de $\Deltan1 \times c_2NE$. 
\tb

Définissons ces données en utilisant la description explicite de $c_2NE$
donnée dans le paragraphe~\ref{orntg}
(\emph{cf.}~proposition~\ref{orntgadj}). Si $x$ est un objet de $c_2NE$,
c'est-à-dire un élément de $E$, on pose
\[
H(0, x) = ir(x) \quad\text{et}\quad H(1, x) = x\ ,
\]
et si $S$ est une $1$\nbd-flèche de $c_2NE$, c'est-à-dire un élément de $\xi(E)$, on pose
\[
H(0, S) = ir(S) \quad\text{et}\quad H(1, S) = S\ .
\]
Enfin, si $S$ est un élément de $\xi(E)$, on pose
\[
H(0 \to 1, S) =
\begin{cases}
  \{ir(\min S),\max S\}, & \text{si $S\subset E-E'$,} \\
  (S\cap E')\cup\{\max S\}, & \text{sinon.} \\
\end{cases}
\]
\end{paragr}

\begin{prop}\label{prop:c_2Ni_RDF}
Le crible $c_2N(i)$ est un rétracte par déformation oplax fort. Plus
précisément, $c_2N(r)$ et $H$ définissent une structure de rétracte par
déformation oplax fort sur $c_2N(i)$.
\end{prop}

\begin{proof}
Vérifions que $H$ est un $2$-foncteur oplax. Nous devons tout d'abord nous
assurer que $H$ est bien défini sur les $2$-flèches. Il s'agit de voir que
si $S \subset S'$ sont deux éléments de $\xi E$ tels que $\min S = \min S'$
et $\max S = \max S'$, alors on a
\[
H(0, S) \subset H(0, S'), \quad
H(1, S) \subset H(1, S') \quad\text{et}\quad
H(0 \to 1, S) \subset H(0 \to 1, S')\ .
\]
Les deux premières inclusions sont évidentes. Pour la troisième, il suffit
de remarquer qu'on a $S \subset E-E'$ si et seulement si $S' \subset E-E'$,
car $E'$ étant un crible, cette condition ne dépend en fait que de l'entier
$\min S = \min S'$.
\tb

Vérifions maintenant que les contraintes de fonctorialité de $H$ sont bien
définies. Comme la restriction de $H$ à $\{\e\}\times c_2NE$, $\e=0,1$, est
un $2$\nbd-foncteur strict, il suffit de prouver que pour $S,S'$ deux
éléments de $\xi E$ tels que $\max S=\min S'$, on a des inclusions
\[
\begin{aligned}
H\bigl((0\to1,S')\circ(0,S)\bigr)=H(0\to1,S'\cup S)\subset H(0\to1,S')\circ H(0,S)\ ,\\
H\bigl((1,S')\circ(0\to1,S)\bigr)=H(0\to1,S'\cup S)\subset H(1,S')\circ H(0\to1,S)\ .
\end{aligned}
\]
On remarque que $\min (S\cup S')=\min S$ et $\max (S\cup S')=\max S'$.
On distingue trois cas :
\begin{enumerate}[label=(\arabic*), wide]
  \item $S\subset E'$, ce qui implique $S'\not\subset E-E'$ et $S'\cup S\not\subset E-E'$.

   On a alors
   \[
\begin{aligned}
&H(0\to1,S'\cup S)=((S'\cup S)\cap E')\cup\{\max (S'\cup S)\}=(S'\cap E')\cup S\cup\{\max S'\}\ ,\\
&H(0\to1,S')\circ H(0,S)=(S'\cap E')\cup\{\max S'\}\cup ir(S)=(S'\cap E')\cup\{\max S'\}\cup S\ ,\\
&H(1,S')\circ H(0\to1,S)=S'\cup(S\cap E')\cup\{\max S\}=S'\cup S\ .
\end{aligned}
   \]

  \item $S\cap E'\neq\varnothing\neq S\cap(E-E')$, ce qui implique
    $S'\subset E-E'$ et $S'\cup S\not\subset E-E'$.
    
    On a alors
    \[
\begin{aligned}
&H(0\to1,S'\cup S)=((S'\cup S)\cap E')\cup\{\max (S'\cup S)\}=(S\cap E')\cup \{\max S'\}\ ,\\
&H(0\to1,S')\circ H(0,S)=\{ir(\min S'),\max S'\}\cup ir(S)\ ,\\
&H(1,S')\circ H(0\to1,S)=S'\cup(S\cap E')\cup\{\max S\}\ .
\end{aligned}
    \]

  \item  $S\subset E-E'$, ce qui implique $S'\subset E-E'$ et $S'\cup
    S\subset E-E'$.
    
    On a alors
    \[
\begin{aligned}
&H(0\to1,S'\cup S)=\{ir(\min (S'\cup S)),\max (S'\cup S)\}=\{ir(\min S),\max S'\}\ ,\\
&H(0\to1,S')\circ H(0,S)=\{ir(\min S'),\max S'\}\cup ir(S)\ ,\\
&H(1,S')\circ H(0\to1,S)=S'\cup\{ir(\min S),\max S\}\ .
\end{aligned}
    \]
\end{enumerate}

\noindent
On a donc bien, dans tous les cas, ces deux inclusions. Il résulte alors du fait que $c_2NE$ est
une catégorie enrichie en ensembles ordonnés que $H$ est bien un foncteur
oplax.

\tb
Il est par ailleurs immédiat que $H$ est une homotopie oplax de
$c_2N(i)\,c_2N(r)$ vers~$1_{c_2NE}$ et qu'elle satisfait à la
condition~\eqref{item:RDF1}. Enfin, on remarque que dans le cas~(1)
ci-dessus, on a montré que si $S\subset E'$, on a non seulement une
inclusion, mais même une égalité
\[
H(0 \to 1, S' \cup S) = H(0 \to 1, S') \circ H(0, S)\ ,
\]
ce qui prouve que $H$ satisfait à la condition~\eqref{item:RDF2} et achève la démonstration.
\end{proof}

\begin{cor}\label{coro:d'}
Si $E'\to E$ est un crible de $\ord$ admettant une rétraction qui est aussi
un adjoint à droite, son image par le foncteur $c_2N$ est une équivalence
faible de Thomason et le reste après tout cochangement de base.
\end{cor}

\begin{proof}
En vertu de la proposition précédente, $c_2N(E' \to E)$ est un $2$\nbd-crible rétracte
par déformation oplax fort. Le résultat est donc une conséquence immédiate
du corollaire~\ref{coro:2rdf_couniv}.
\end{proof}

\begin{thm}\label{thm:cmf_2cat}
Il existe une structure de catégorie de modèles combinatoire propre sur
$\dcat$ engendrée par $(c_2\Sd^2(I), c_2\Sd^2(J))$ \emph{(cf.
\ref{enssimpl})} dont les équivalences faibles sont les équivalences faibles
de Thomason et les fibrations sont les fibrations de Thomason
\emph{(cf.~\ref{eqfThom})}.
De plus, $(c_2\Sd^2,\Ex^2N_2)$ est une adjonction de Quillen.
\end{thm}

\begin{proof}
En vertu du scholie~\ref{remnThom}, il suffit de vérifier les
conditions~($\text{\emph{d}}'$) et~($\emph{e}$) dudit scholie. C'est
exactement le contenu des corollaires~\ref{conditione} et \ref{coro:d'}.
\end{proof}

Nous appellerons cette structure de catégorie de modèles la \ndef{structure
de catégorie de modèles à la Thomason sur $\dcat$}.

\begin{paragr}
Nous terminons cette section en expliquant comment on peut déduire
immédiatement de résultats classiques et d'un théorème de J.~Chiche que 
l'adjonction de Quillen
\[ c_2\Sd^2 : \simpl \to \dcat\ , \qquad \Ex^2N_2 : \dcat \to \simpl\ \]
est une équivalence de Quillen.
\end{paragr}

\begin{thm}[Illusie, Quillen]
Le foncteur $N : \cat \to \simpl$ induit une équivalence de catégories entre
les catégories homotopiques (où on considère $\cat$ munie des équivalences
faibles de Thomason).
\end{thm}

\begin{proof}
Voir \cite[chapitre VI, corollaire 3.3.1]{IL}.
\end{proof}

\begin{thm}[J.~Chiche]\label{thm:Chiche}
Le foncteur d'inclusion $\cat \to \dcat$ induit une équivalence de
catégories entre les catégories homotopiques (où on considère $\cat$
et~$\dcat$ munies des équivalences faibles de Thomason).
\end{thm}

\begin{proof}
Voir \cite[théorème 7.9]{Chiche}.
\end{proof}

\begin{cor}
Le foncteur $N_2 : \dcat \to \simpl$ induit une équivalence de catégories
entre les catégories homotopiques (où on considère $\dcat$ munie des
équivalences faibles de Thomason).  \end{cor}

\begin{proof}
Notons $i : \cat \to \dcat$ le foncteur d'inclusion. On a $N_2i = N$.  Ces
trois foncteurs respectant les équivalences faibles, on a également
$\overline{N_2}\,\overline{i} = \overline{N}$ (où
$\overline{N_2},\,\overline{i}$ et $\overline{N}$ désignent les foncteurs
entre les catégories localisées induits par $N_2,\,i$ et~$N$
respectivement). Mais en vertu des deux théorèmes précédents, $\overline{i}$ et
$\overline{N}$ sont des équivalences de catégories. Il en est donc de même
de $\overline{N_2}$.
\end{proof}

\begin{cor}
L'adjonction de Quillen 
\[ c_2\Sd^2 : \simpl \to \dcat\ , \qquad \Ex^2N_2 : \dcat \to \simpl \]
est une équivalence de Quillen.
\end{cor}

\begin{proof}
Cela résulte immédiatement de la remarque \ref{remThomabseq} et du
corollaire précédent.
\end{proof}

\section{Stabilité par images directes des 2-cribles rétractes par\\
déformation oplax forts}\label{sec:2rdf}

Cette section est consacrée à la démonstration de la
proposition~\ref{prop:2rdf_img_dir}. En vertu de la proposition
\ref{imdircribles} et de l'exemple \ref{nCatcribles}, les cribles de
$2$-catégories sont stables par images directes. Néanmoins, afin de montrer
la proposition~\ref{prop:2rdf_img_dir}, nous aurons besoin de décrire
explicitement l'image directe d'un tel crible. Nous commençons par le cas
$1$-catégorique.

\subsection{Image directe de cribles dans $\Cat$}

Soient $i : A \to B$ un crible de catégories et $u : A \to A'$ un foncteur
entre catégories. Le but de cette sous-section est de décrire
l'image directe de $i$ le long de $u$. Il s'agit donc de définir une
catégorie~$B'$ munie de foncteurs $i' : A' \to B'$ et $v : B \to B'$ 
tels que le carré
\[
\xymatrix{
A \ar[r]^u \ar[d]_i & A' \ar[d]^{i'}\\
B \ar[r]_v & B' \\
}
\]
soit un carré cocartésien de $\Cat$.
\tb

Dans la suite, pour simplifier, nous supposerons que $i$ est une inclusion.

\begin{dparagr}
Commençons par définir le graphe sous-jacent à $B'$. Les objets de $B'$ sont
donnés par
\[ \Ob(B') = \Ob(A') \amalg \Ob(B \sauf A)\ . \]
Pour $x$ et $y$ deux objets de $B'$, on pose
\[
\Hom_{B'}(x, y) = 
         \begin{cases}
           \Hom_{A'}(x, y), & \text{si $x, y \in \Ob(A')$ ;} \\
           \Hom_B(x, y), & \text{si $x, y \in \Ob(B \sauf A)$ ;}\\
           \varnothing, & \text{si $x \in \Ob(B \sauf A)$ et $y \in \Ob(A')$ ;}\\
           \pi_0(\cm{A}{y} \times_A \mc{x}{A}), & \text{si $x \in \Ob(A')$ et $y \in
           \Ob(B \sauf A)$.}
         \end{cases}
\]
Soient $x$ un objet de $A'$ et $y$ un objet de $B \sauf A$. Explicitons
l'ensemble $\Hom_{B'}(x, y)$. Un élément de cet ensemble est représenté par un
triplet $(f_2, a, f_1)$, où $a$ est un objet de~$A$, $f_1 : x \to u(a)$ est
un morphisme de $A'$ et $f_2 : a \to y$ est un morphisme de $B$. Si $(f_2,
a, f_1)$ et $(f'_2, a', f'_1)$ sont deux tels triplets, on note $(f_2, a,
f_1) \sim (f'_2, a', f'_1)$ s'il existe une flèche $h : a \to a'$ de $A$
faisant de
\[
\xymatrix@C=1.5pc@R=.5pc{
& x \ar[ddl]_{f_1} \ar[ddr]^{f'_1} &
& &
a \ar[rr]^{h} \ar[ddr]_{f_2} & & a' \ar[ddl]^{f'_2} \\
& & & \text{et}\\
u(a) \ar[rr]_{u(h)} & & u(a')
& &
& y
}
\]
des triangles commutatifs.
Deux triplets représentent le même élément de $\Hom_{B'}(x, y)$ s'ils sont
dans la même classe d'équivalence pour la relation d'équivalence engendrée
par $\sim$. On notera $\overline{(f_2, a, f_1)}$ l'image d'un tel triplet dans
$\Hom_{B'}(x, y)$. L'objet $a$ d'un tel triplet étant uniquement déterminé
par $f_2$, on représentera souvent les éléments de $\Hom_{B'}(x, y)$ par des
couples $(f_2, f_1)$.
\end{dparagr}
  
\begin{dparagr}
Définissons maintenant une structure de catégorie sur $B'$. Les identités
sont héritées de manière évidente des identités de $A'$ et $B$. 
Soient $x \xrightarrow{f} y \xrightarrow{g} z$ deux flèches composables de
$B'$. Définissons leur composé.
\begin{itemize}
  \item Si $x$, $y$ et $z$ sont dans $A'$, le composé est induit par celui
    de $A'$.
  \item Si $x$, $y$ et $z$ sont dans $B\sauf A$, le composé est induit par
    celui de $B$.
  \item Si $x$ est dans $A'$ et $y, z$ sont dans $B\sauf A$, on a 
    $f = \overline{(f_2, f_1)}$ et on pose
    \[ g\overline{(f_2, f_1)} = \overline{(gf_2, f_1)}\ . \]
    Il est immédiat que $g$ induit un foncteur
    \[ \cm{A}{y} \times_A \mc{x}{A} \to \cm{A}{z} \times_A \mc{x}{A} \]
    et la formule précédente est donc bien définie.
  \item Si $x$ et $y$ sont dans $A'$ et $z$ est dans $B\sauf A$, on a 
    $g = \overline{(g_2, g_1)}$ et on pose
    \[ \overline{(g_2, g_1)}f = \overline{(g_2, g_1f)}\ . \]
    On vérifie comme ci-dessus que cette formule est bien définie.
\end{itemize}
On vérifie immédiatement qu'on obtient ainsi une catégorie. 
\end{dparagr}
  
\begin{dparagr}
Par définition, $A'$ est une sous-catégorie de $B'$. On notera $i' :
A' \to B'$ le foncteur d'inclusion. On définit un foncteur $v : B \to B'$
de la manière suivante. Si $x$ est un objet de $B$, on pose
\[
v(x) =
\begin{cases}
  x, & \text{si $x \in \Ob(B\sauf A$),}\\
  u(x), & \text{si $x \in \Ob(A)$.}
\end{cases}
\]
Si $f : x \to y$ est une flèche de $B$, on pose
\[
v(f) =
\begin{cases}
  f, & \text{si $x, y \in \Ob(B\sauf A$)},\\
  u(f), & \text{si $x, y \in \Ob(A)$},\\
  \overline{(f, \id{u(x)})}, & \text{si $x \in \Ob(A)$ et
  $y \in \Ob(B\sauf A)$}. 
\end{cases}
\]
On vérifie immédiatement que $v$ est un foncteur.
\tb

On a donc un carré
\[
\xymatrix{
A \ar[r]^u \ar[d]_i & A' \ar[d]^{i'}\\
B \ar[r]_v & B' \pbox{.} \\
}
\]
On vérifie immédiatement que ce carré est commutatif.
\end{dparagr}

\begin{dprop}
Le carré
\[
\xymatrix{
A \ar[r]^u \ar[d]_i & A' \ar[d]^{i'}\\
B \ar[r]_v & B'
}
\]
est un carré cocartésien de $\Cat$.
\end{dprop}

\begin{proof}
Soient $C$ une catégorie et $j : A' \to C$ et $w : B \to C$ des foncteurs
tels que $ju = wi$. Il s'agit de montrer qu'il existe un unique foncteur $l
: B' \to C$ rendant commutatif le diagramme
\[
\xymatrix@R=.4pc@C=.6pc{
A \ar[rrr]^u \ar[ddd]_i & & & A' \ar[ddd]_{i'} \ar@/^1em/[dddddrr]^{j}\\
\\
\\
B \ar[rrr]^v \ar@/_1em/[ddrrrrr]_{w} & & & B' \ar@{-->}[ddrr]^{l} \\
\\
& & & & & C \pbox{.}
}
\]

\noindent\textsc{Existence.}
Définissons un tel foncteur $l$. Pour $x$ un objet de $B'$, on pose
\[
l(x) = 
\begin{cases}
j(x), & \text{si $x \in \Ob(A')$,}\\ 
w(x), & \text{si $x \in \Ob(B\sauf A)$.}\\ 
\end{cases}
\]
Pour $f : x \to y$ un morphisme de $B'$, on pose
\[
l(f) =
\begin{cases}
j(f), & \text{si $x, y \in \Ob(A')$,}\\
w(f), & \text{si $x, y \in \Ob(B\sauf A)$,}\\
w(f_2)j(f_1), & \text{si $x \in \Ob(A')$, $y \in \Ob(B\sauf A)$ et $f = \overline{(f_2,
f_1)}$.}
\end{cases}
\]
Montrons que la dernière formule est bien définie. Supposons qu'on ait deux
couples $(f_2, f_1)$ et $(f'_2, f'_1)$ en relation \forlang{via} un
morphisme $h$ de $A$ (autrement dit, $h$ vérifie $f'_1 = u(h)f^{}_1$ et
$f^{}_2 = f'_2h$). On a alors
\[
\begin{split}
w(f_2)j(f_1) & = w(f'_2h)j(f_1) = w(f'_2)w(h)j(f_1)
= w(f'_2)ju(h)j(f_1) \\
&
= w(f'_2)j(u(h)f_1)
= w(f'_2)j(f'_1)\ .
\end{split}
\]
On vérifie immédiatement que $l$ définit bien un foncteur et qu'il satisfait
aux égalités $li' = j$ et $lv = w$.
\tb

\noindent\textsc{Unicité.}
Les formules définissant $l$ sont imposées par la compatibilité à $j$
et~$w$. C'est évident pour toutes les formules sauf éventuellement pour la
dernière. Pour celle-ci, cela résulte des égalités
\[
\overline{(f_2, f_1)} = \overline{(f_2, \id{})}f_1 = v(f_2)i'(f_1)\ .
\]
Le foncteur $l$ est donc unique, ce qui achève la démonstration.
\end{proof}

\subsection{Image directe de cribles dans $\ncat2$}\label{sec:cribles2Cat}

Soient $i : A \to B$ un crible de $2$-catégories et $u : A \to A'$ un
$2$-foncteur strict entre $2$-catégories. Le but de cette sous-section est de
décrire l'image directe de $i$ le long de $u$. Il s'agit donc de définir une
$2$-catégorie $B'$ munie de $2$-foncteurs stricts $i' : A' \to B'$ et $v :
B \to B'$ tels que le carré
\[
\xymatrix{
A \ar[r]^u \ar[d]_i & A' \ar[d]^{i'}\\
B \ar[r]_v & B' \\
}
\]
soit un carré cocartésien de $\ncat2$.
\tb

Dans la suite, pour simplifier, nous supposerons que $i$ est une inclusion
(ce qui est licite en vertu de la proposition~\ref{carcriblesncat}).

\begin{dparagr}\label{paragr:relR}
Commençons par définir le graphe enrichi en catégories sous-jacent à $B'$.
Le graphe sous-jacent à ce graphe enrichi ne sera rien d'autre que le graphe
sous-jacent à la catégorie $\Uo(B) \amalg_{\Uo(A)} \Uo(A')$, où $\Uo$
désigne $1$-tronqué bête du paragraphe~\ref{tronq}, qu'on a décrite dans la
sous-section précédente.
\tb

Les objets de $B'$ sont donnés par
\[ \Ob(B') = \Ob(A') \amalg \Ob(B \sauf A)\ . \]
Pour $x$ et $y$ deux objets de $B'$, on pose
\[
\Homi_{B'}(x, y) = 
         \begin{cases}
           \Homi_{A'}(x, y), & \text{si $x, y \in \Ob(A')$ ;} \\
           \Homi_B(x, y), & \text{si $x, y \in \Ob(B \sauf A)$ ;}\\
           \varnothing, & \text{si $x \in \Ob(B \sauf A)$ et $y \in \Ob(A')$ ;}\\
           B'_{x, y}, & \text{si $x \in \Ob(A')$ et $y \in
           \Ob(B \sauf A)$,}
         \end{cases}
\]
où $B'_{x, y}$ est la catégorie que nous allons maintenant définir.
\tb

Rappelons que si $C$ est une $2$-catégorie, on peut lui associer son
$1$-tronqué bête~$\Uo(C)$, mais aussi une $1$-catégorie
$\Ut(C)$ dont les objets sont ceux de $C$ et dont les flèches sont les
$2$\nbd-flèches de $C$ (les opérations source et but étant induites par les
opérations source et but itérés de $C$). De plus, les opérations source et
but pour les $2$-flèches de $C$ induisent des foncteurs source et but
$\Ut(C) \rightrightarrows \Uo(C)$ (qui sont l'identité sur les objets).
On utilisera la notation $\alpha : x \toz y$ pour indiquer que $\alpha$ est
une flèche de $\Ut(C)$ de source $x$ et but $y$, autrement dit, que $\alpha$
est une $2$-flèche de $C$ dont la source itérée est l'objet $x$ et le but
itéré est l'objet $y$.
\tb

Fixons maintenant un objet $x$ de $A'$ et un objet $y$ de $B \sauf A$ et
décrivons la catégorie~$B'_{x, y}$. En utilisant les foncteurs définis
ci-dessus, on obtient un graphe de catégories
\[ 
\cm{\Ut(A)}{y} \times_{\Ut(A)} \mc{x}{\Ut(A)}
\rightrightarrows
\cm{\Uo(A)}{y} \times_{\Uo(A)} \mc{x}{\Uo(A)}\ ,
\]
et donc des graphes (d'ensembles)
\[
\begin{split}
G^\flat_{x, y} & =
\Ob(\cm{\Ut(A)}{y} \times_{\Ut(A)} \mc{x}{\Ut(A)})
\rightrightarrows
\Ob(\cm{\Uo(A)}{y} \times_{\Uo(A)} \mc{x}{\Uo(A)})\ , \\
G_{x, y} & =
\Ob(\cm{\Ut(A)}{y} \times_{\Ut(A)} \mc{x}{\Ut(A)})
\rightrightarrows
\pi_0(\cm{\Uo(A)}{y} \times_{\Uo(A)} \mc{x}{\Uo(A)})\ .
\end{split}
\]
Explicitement, les objets de $G^\flat_{x, y}$ sont les triplets $(f_2, a,
f_1)$ où $f_1 : x \to u(a)$ est une $1$-flèche de $A'$ et $f_2 : a \to y$
est une $1$-flèche de~$B$, et ses flèches sont les triplets $(\alpha_2, a,
\alpha_1)$ où $\alpha_1 : x \toz u(a)$ est une $2$-flèche de $A'$ et
$\alpha_2 : a \toz y$ est une $2$-flèche de~$B$. Les objets du graphe $G_{x,
y}$ sont les classes d'équivalences de triplets $(f_2, a, f_1)$ comme
ci-dessus (pour la même relation d'équivalence que dans le cas
$1$-catégorique) et ses flèches sont les mêmes que celles du graphe
$G^\flat_{x, y}$.
\tb

La catégorie $B'_{x,y}$ est la catégorie engendrée par le graphe $G_{x, y}$
et les relations suivantes :
\begin{enumerate}[label=R\arabic*),ref=R\arabic*]
  \item\label{item:R1}
    pour toutes flèches composables 
    \[\xrightarrow{(\alpha_2, a, \alpha_1)}\xrightarrow{(\beta_2, b,
    \beta_1)}\] de $G^\flat_{x, y}$ (ces flèches sont \forlang{a fortiori}
    composables dans $G_{x, y}$ et on a $a = b$), on a
    \[
    (\beta_2, a, \beta_1)(\alpha_2, a, \alpha_1) \sim
    (\beta_2\beta_1, a, \alpha_2\alpha_1)\ ;
    \]
  \item\label{item:R2}
   pour tout triplet $(f_2, a, f_1)$ représentant un objet de $G_{x, y}$, on a
    \[ \id{\overline{(f_2, a, f_1)}} \sim (\id{f_2}, a, \id{f_1})\ ;\]
  \item\label{item:R3}
  pour toutes flèches $(\alpha_2, a, \alpha_1)$ et $(\alpha'_2, a',
    \alpha'_1)$ de $G_{x, y}$ telles qu'il existe une $2$-flèche $\gamma : a
    \toz a'$ de $A$ faisant de
\[
\xymatrix@C=1.5pc@R=.5pc{
& x \ar[ddl]_{\alpha_1} \ar[ddr]^{\alpha'_1} &
& &
a \ar[rr]^{\gamma} \ar[ddr]_{\alpha_2} & & a' \ar[ddl]^{\alpha'_2} \\
& & & \text{et}\\
u(a) \ar[rr]_{u(\gamma)} & & u(a')
& &
& y
}
\]
des triangles commutatifs, on a
    \[ (\alpha_2, a, \alpha_1) \sim (\alpha'_2, a', \alpha'_1)\ . \]
\end{enumerate}

Comme dans le cas $1$-catégorique, on notera $\overline{\sstrut(\alpha_2, a,
\alpha_1)}$, ou simplement $\overline{\sstrut(\alpha_2, \alpha_1)}$, la
$2$-flèche de $B'$ correspondant à un triplet $(\alpha_2, a, \alpha_1)$
comme ci-dessus.
\tb

Comme annoncé, le $1$-graphe sous-jacent à $B'$ n'est rien d'autre que le
graphe sous-jacent à la catégorie $\Uo(B) \amalg_{\Uo(A)} \Uo(A')$. En
particulier, en vertu du cas $1$-catégorique, ce graphe est munie d'une
structure de catégorie.
\end{dparagr}

\begin{dparagr}
Pour faire de $B'$ une $2$-catégorie, il nous reste
à définir, pour tous objets $x$, $y$ et $z$ de $B'$, un foncteur de
composition horizontale 
\[ \Homi_{B'}(y, z) \times \Homi_{B'}(x, y) \to \Homi_{B'}(x, z)\ . \]
Soient donc $x$, $y$ et $z$ trois objets de $B'$.
\tb
\begin{enumerate}[label=(\arabic*), wide]
  \item Si $x$, $y$ et $z$ sont dans $A'$, la composition horizontale de
    $B'$ est héritée de celle de~$A'$.
    \tb
  \item Si $x$, $y$ et $z$ sont dans $B\sauf A$, la composition horizontale
    de $B'$ est héritée de celle de~$B$.
    \tb
  \item Supposons que $x$ et $y$ sont dans $A'$ et que $z$ est dans $B\sauf
    A$. Considérons le morphisme de graphes 
    \[
    \comp : G_{y, z} \to \Homi(\Homi_{B'}(x, y), \Homi_{B'}(x, z))
    \]
    donné sur les objets par 
    \[
    \overline{(f_2, f_1)} \longmapsto
    \begin{cases}
    g \mapsto
    \overline{(f_2, f_1)} \comp g \coloneqq \overline{(f_2,
    f_1g)}, &
    \text{si $g \in \Ob(\Homi_{B'}(x, y))$ ;} 
    \\
    \delta \mapsto
    \overline{(f_2, f_1)} \comp \delta \coloneqq \overline{(\id{f_2},
    f_1 \comp \delta)}, &
    \text{si $\delta \in \fl(\Homi_{B'}(x, y))$,} 
    \end{cases}
    \] 
    (la première formule est bien définie par le cas $1$-catégorique et la
    relation \eqref{item:R3} implique facilement qu'il en est de même de la seconde) et
    sur les flèches par
    \[
    (\alpha_2, \alpha_1) \longmapsto
    \big( g \mapsto
     (\alpha_2, \alpha_1) \comp  g \coloneqq \overline{(\alpha_2, \alpha_1 \comp
    g)}\big)\ .
    \]
    Si $(f_2, f_1)$ représente un objet de $G_{y, z}$, $g$ est un objet
    de $\Homi_{B'}(x, y)$ et $\delta, \delta'$ sont des flèches composables
    de $\Homi_{B'}(x, y)$, alors les égalités
        \[
        \overline{(f_2, f_1)} \comp \id{g} =
        \overline{(\id{f_2}, \id{f_1g})} =
        \id{\overline{(f_2, f_1g)}} =
        \id{\overline{(f_2, f_1)} \comp g}
        \]
    (où la deuxième égalité résulte de la relation~\eqref{item:R2}) et
        \[
        \begin{split}
        \overline{(f_2, f_1)} \comp  \delta'\delta & = 
        \overline{(\id{f_2}, f_1 \comp  \delta'\delta)} = 
        \overline{(\id{f_2}, (f_1 \comp  \delta')(f_1\comp
        \delta))} \\
        & =
        \overline{(\id{f_2}, f_1 \comp  \delta')}\,\,
        \overline{(\id{f_2}, f_1 \comp  \delta)} \\
        & \phantom{=1} \text{(en vertu de la relation~\eqref{item:R1})} \\
        & =
        \big(\overline{(f_2, f_1)} \comp \delta'\big)
        \big(\overline{(f_2, f_1)} \comp \delta\big)
        \end{split}
        \]
     montrent que $\overline{(f_2, f_1)} \comp \hbox{?}$ est bien un
     foncteur ; si $(\alpha_2, \alpha_1)$ est une flèche de $G_{y, z}$,
     $\delta$~est une
     flèche de $\Homi_{B'}(x, y)$ et si on note $f_1$, $f_2$ et $g$ (resp.
     $f'_1$, $f'_2$ et $g'$) les sources (resp. les buts) respectifs de
     $\alpha_1$, $\alpha_2$ et $\delta$, de sorte qu'on a
\[
\UseAllTwocells
\xymatrix@C=3pc{
x \rtwocell^g_{g'}{\,\delta}
&
y \rtwocell^{f_1}_{f'_1}{\;\;\alpha_1}
&
\zbox{$u(a)$\ ,}\phantom{a}
}
\qquad\quad
\xymatrix@C=3pc{
a \rtwocell^{f_2}_{f'_2}{\;\;\alpha_2}
&
z
\zbox{\ ,}
}
\]
    alors les égalités
        \[
        \begin{split}
          \big((\alpha_2, \alpha_1) \comp g' \big)
          \big(\overline{(f_2, f_1)} \comp \delta\big)
          & =
          \overline{(\alpha_2, \alpha_1 \comp g')}\,\,
          \overline{(\id{f_2}, f_1 \comp \delta)} \\
          & =
          \overline{(\alpha_2, (\alpha_1 \comp g')(f_1 \comp \delta))}
          \\
          & \phantom{=1} \text{(en vertu de la relation~\eqref{item:R1})} \\
          & =
          \overline{(\alpha_2, (f'_1 \comp \delta)(\alpha_1 \comp g))}
          \\
          & =
          \overline{(\id{f'_2}, f'_1 \comp \delta)}\,\,
          \overline{(\alpha_2, \alpha_1 \comp g)} \\
          & \phantom{=1} \text{(en vertu de la relation~\eqref{item:R1})} \\
          & =
          \big(\overline{(f'_2, f'_1)} \comp \delta\big)
          \big((\alpha_2, \alpha_1) \comp g \big)
        \end{split}
        \]
    montrent que $(\alpha_2, \alpha_1) \comp \hbox{?}$ est bien une
    transformation naturelle.
    \tb

    Ce morphisme de graphes passe au quotient par les relations \eqref{item:R1}, \eqref{item:R2}
    et \eqref{item:R3}. En effet, si $(\alpha_2, \alpha_1)$ et $(\alpha'_2,
    \alpha'_1)$ sont des générateurs de $\Homi_{B'}(x, y)$ satisfaisant aux
    hypothèses de \eqref{item:R1}, on a
           \[
        \begin{split}
        ((\alpha'_2, \alpha'_1) \comp g)
        ((\alpha_2, \alpha_1) \comp g)
        & =
        \overline{(\alpha'_2, \alpha'_1 \comp g)}\,\,
        \overline{(\alpha_2, \alpha_1 \comp g)} \\
        & =
        \overline{(\alpha'_2\alpha^{}_2, (\alpha'_1 \comp g)(\alpha_1 \comp
        g))} \\
        & \phantom{=1} \text{(en vertu de la relation~\eqref{item:R1})} \\
        & =
        \overline{(\alpha'_2\alpha^{}_2, \alpha'_1\alpha^{}_1 \comp g)} \\
        & =
        (\alpha'_2\alpha^{}_2, \alpha'_1\alpha^{}_1) \comp g\ ,
        \end{split}
        \]
    ce qui établit la compatibilité à \eqref{item:R1}. La compatibilité à \eqref{item:R2} résulte de
    l'égalité
        \[
        \overline{(\id{f_2}, \id{f_1})} \comp g =
        \overline{(\id{f_2}, \id{f_1g})} =
        \id{\overline{(f_2, f_1g)}}\ ,
        \]
    où $\overline{(f_2, f_1)}$ est un objet de $G_{y, z}$ et $g$ est un objet
    de $\Homi_{B'}(x, y)$ (la deuxième égalité résultant de la
    relation~\eqref{item:R2}).
    \tb

    Enfin, si $(\alpha_2, \alpha_1)$ et $(\alpha'_2, \alpha'_1)$ sont des
    flèches de $G_{y, z}$ satisfaisant aux hypothèses de \eqref{item:R3}, on vérifie
    immédiatement qu'il en est de même
    de $(\alpha_2, \alpha_1 \comp g)$ et $(\alpha'_2, \alpha'_1 \comp g)$
    pour tout objet $g$ de $\Homi_{B'}(x, y)$ et on a donc
           \[
        (\alpha_2, \alpha_1) \comp g
        =
        \overline{(\alpha_2, \alpha_1 \comp g)}
        =
        \overline{(\alpha'_2, \alpha'_1 \comp g)}
        = 
        (\alpha'_2, \alpha'_1) \comp g\ ,
        \]
    la deuxième égalité résultant de la relation~\eqref{item:R3}, ce qui prouve la
    compatibilité à~\eqref{item:R3}.
    \tb

    Le morphisme de graphes considéré induit donc par la propriété
    universelle définissant $\Homi_{B'}(y, z)$ et par adjonction un foncteur
    \[ 
    \Homi_{B'}(y, z) \times \Homi_{B'}(x, y) \to \Homi_{B'}(x, z)\ ,
    \]
    et on définit la composition horizontale par ce foncteur.
    \tb

  \item Si $x$ est dans $A'$ et $y, z$ sont dans $B\sauf A$, la composition
    horizontale est définie comme dans le cas précédent à partir du
    morphisme de graphes
    \[ 
     G_{x, y} \to \Homi(\Homi_{B'}(y, z), \Homi_{B'}(x, z)) \\
    \]
    donné sur les objets par
    \[
       g \comp \overline{(f_2, f_1)} = \overline{(gf_2, f_1)}
       \quad\text{et}\quad
       \delta \comp \overline{(f_2, f_1)} = \overline{(\delta \comp
       f_2, \id{f_1})}
    \] 
    et sur les flèches par
    \[
    g \comp (\alpha_2, \alpha_1) = \overline{(g \comp
    \alpha_2, \alpha_1)}\ .
    \]
\tb
On vérifie immédiatement que la composition ainsi définie est associative et
compatible aux identités des objets.
\end{enumerate}
\end{dparagr}

\begin{dparagr}\label{paragr:def_v}
Par définition, $A'$ est une sous-$2$-catégorie de $B'$. On notera $i' : A'
\to B'$ le $2$\nbd-foncteur d'inclusion. On définit un $2$-foncteur strict
$v : B \to B'$ de la manière suivante. Si $x$ est un objet de $B$, on pose
\[
v(x) =
\begin{cases}
  x, & \text{si $x \in \Ob(B\sauf A)$,}\\
  u(x), & \text{si $x \in \Ob(A)$.}
\end{cases}
\]
Si $f : x \to y$ est une $1$-flèche de $B$, on pose
\[
v(f) =
\begin{cases}
  f, & \text{si $x, y \in \Ob(B\sauf A)$},\\
  u(f), & \text{si $x, y \in \Ob(A)$},\\
  \overline{(f, \id{u(x)})}, & \text{si $x \in \Ob(A)$
  et $y \in \Ob(B\sauf A)$}. 
\end{cases}
\]
Enfin, si $\alpha : x \toz y$ est une $2$-flèche de $B$, on pose
\[
v(\alpha) =
\begin{cases}
  \alpha, & \text{si $x, y \in \Ob(B\sauf A)$},\\
  u(\alpha), & \text{si $x, y \in \Ob(A)$},\\
  \overline{(\alpha, \id{u(x)})}, & \text{si $x \in \Ob(A)$
  et $y \in \Ob(B\sauf A)$}. 
\end{cases}
\]
Il est immédiat que $v$ est un morphisme de $2$-graphes. De plus, en vertu
du cas $1$\nbd-catégorique, $u$ est fonctoriel sur les $1$-flèches. Il nous
reste donc à montrer la $2$\nbd-fonctorialité. La compatibilité aux
identités et à la composition verticale est évidente. Montrons la
compatibilité à la composition horizontale.
\tb

Soient donc $x \xrightarrowz{\alpha} y \xrightarrowz{\beta} z$ deux
$2$-flèches de $B$. Si les objets $x$, $y$ et $z$ sont tous les trois dans
$A$, ou tous les trois dans $B\sauf A$, la compatibilité est évidente.
Supposons que $x$ et $y$ sont dans $A$ et que $z$ est dans $B\sauf A$. On a
alors
\[
\begin{split}
  v(\beta) \comp v(\alpha) & = \overline{(\beta, \id{u(y)})} \comp u(\alpha)
  = \overline{(\beta, u(\alpha))} = \overline{(\beta \comp \alpha, \id{u(x)})} =
  v(\beta \comp \alpha)\ ,
\end{split}
\]
où l'avant-dernière égalité résulte de la relation \eqref{item:R3}. Le cas restant se
traite de manière similaire.
\tb

On a donc un carré
\[
\xymatrix{
A \ar[r]^u \ar[d]_i & A' \ar[d]^{i'}\\
B \ar[r]_v & B' \pbox{.} \\
}
\]
On vérifie immédiatement que ce carré est commutatif.
\end{dparagr}

\begin{dprop}
Le carré
\[
\xymatrix{
A \ar[r]^u \ar[d]_i & A' \ar[d]^{i'}\\
B \ar[r]_v & B'
}
\]
est un carré cocartésien de $\ncat2$.
\end{dprop}

\begin{proof}
Soient $C$ une $2$-catégorie et $j : A' \to C$
et $w : B \to C$ deux $2$\nbd-foncteurs stricts tels que $ju = wi$. Il
s'agit de montrer qu'il existe un unique $2$\nbd-foncteur strict $l : B' \to
C$ rendant commutatif le diagramme
\[
\xymatrix@R=.4pc@C=.6pc{
A \ar[rrr]^u \ar[ddd]_i & & & A' \ar[ddd]_{i'} \ar@/^1em/[dddddrr]^{j}\\
\\
\\
B \ar[rrr]^v \ar@/_1em/[ddrrrrr]_{w} & & & B' \ar@{-->}[ddrr]^{l} \\
\\
& & & & & C \pbox{.}
}
\]

\noindent\textsc{Existence.}
Définissons un tel foncteur $l$. Pour $x$ un objet de $B'$, on pose
\[
l(x) = 
\begin{cases}
j(x), & \text{si $x \in \Ob(A')$,}\\ 
w(x), & \text{si $x \in \Ob(B\sauf A)$.}\\ 
\end{cases}
\]
Pour $f : x \to y$ une $1$-flèche de $B'$, on pose
\[
l(f) =
\begin{cases}
j(f), & \text{si $x, y \in \Ob(A')$,}\\
w(f), & \text{si $x, y \in \Ob(B\sauf A)$,}\\
w(f_2)j(f_1), & \text{si $x \in \Ob(A')$, $y \in \Ob(B\sauf A)$} \\
\phantom{w(f_2)j(f_1),} & \text{et $(f_2, f_1)$ est un représentant de $f$.}
\end{cases}
\]
(La troisième formule est bien définie en vertu du cas $1$-catégorique.)
Enfin, pour $\alpha : x \toz y$ une $2$-flèche de $B'$, on pose
\[
l(\alpha) =
\begin{cases}
j(\alpha), & \text{si $x, y \in \Ob(A')$,}\\
w(\alpha), & \text{si $x, y \in \Ob(B\sauf A)$,}\\
w(\alpha_2) \comp j(\alpha_1), & \text{si $x \in \Ob(A')$,
$y \in \Ob(B\sauf A)$} \\
\phantom{w(\alpha_2) \comp j(\alpha_1),} & \text{et $(\alpha_2, \alpha_1)$
est un représentant de $\alpha$.}
\end{cases}
\]
Dans le troisième cas, nous avons défini $l(\alpha)$ uniquement sur les
générateurs de $\Homi_{B'}(x, y)$. Nous devons donc vérifier que les
relations \eqref{item:R1}, \eqref{item:R2} et \eqref{item:R3} sont
satisfaites. Si $(\beta_2, \beta_1)$ et $(\alpha_2, \alpha_1)$ sont des
générateurs de $\Homi_{B'}(x, y)$ satisfaisant aux hypothèses de
\eqref{item:R1}, on a
\[
\begin{split}
l(\overline{\beta_2, \beta_1})l(\overline{\sstrut\alpha_2, \alpha_1})
& =
(w(\beta_2) \comp j(\beta_1))(w(\alpha_2) \comp j(\alpha_1)) \\
& = 
(w(\beta_2)w(\alpha_2)) \comp (j(\beta_1)j(\alpha_1)) \\
& =
w(\beta_2\alpha_2) \comp j(\beta_1\alpha_1) \\
& =
l(\overline{\beta_2\alpha_2, \beta_1\alpha_1})\ ,
\end{split}
\]
ce qui établit la compatibilité à \eqref{item:R1}. La compatibilité à
\eqref{item:R2} résulte des égalités
\[
l(\overline{\id{f_2}, \id{f_1}}) = w(\id{f_2}) \comp j(\id{f_1})
= \id{w(f_2)j(f_1)} = \id{l\overline{(f_2, f_1)}}\ ,
\]
où $(f_2, f_1)$ est un représentant d'un objet de $\Homi_{B'}(x, y)$.
Enfin, si $(\alpha_2, \alpha_1)$ et $(\alpha'_2, \alpha'_1)$ sont
deux générateurs de $\Homi_{B'}(x, y)$ satisfaisant aux hypothèses de \eqref{item:R3}
\forlang{via} une $2$\nbd-flèche $\gamma$ de $A$, on a
\[
\begin{split}
l(\overline{\sstrut\alpha_2, \alpha_1})
& =
w(\alpha_2) \comp j(\alpha_1)
=
w(\alpha'_2 \comp \gamma) \comp j(\alpha_1) \\
& =
w(\alpha'_2) \comp w(\gamma) \comp j(\alpha_1)
=
w(\alpha'_2) \comp ju(\gamma) \comp j(\alpha_1) \\
& =
w(\alpha'_2) \comp j(u(\gamma) \comp \alpha_1)
=
w(\alpha'_2) \comp j(\alpha'_1) \\
& =
l(\overline{\sstrut\alpha'_2, \alpha'_1})\ ,
\end{split}
\]
ce qui achève la démonstration du fait que $l$ est bien défini.
\tb

Montrons maintenant que $l$ est un $2$-foncteur strict. Le cas
$1$-catégorique montre que $l$ est $1$-fonctoriel. Il est immédiat que $l$
est compatible aux identités des $1$-flèches et à la composition verticale.
Montrons la compatibilité à la composition horizontale. Comme dans la
vérification de la propriété analogue pour $v$
(\emph{cf.}~\ref{paragr:def_v}), il y a quatre configurations à traiter.
Deux sont triviales et deux sont essentiellement symétriques. Traitons un
des cas non triviaux, l'autre cas se traitant de façon analogue.
\tb

Soient donc $x \xrightarrowz{\alpha} y \xrightarrowz{\beta} z$ deux
$2$-flèches de $B'$ avec $x, y$ dans $A'$ et $z$ dans $B\sauf A$. On peut
supposer que $\beta$ est un générateur $\Homi_{B'}(y, z)$
donné par un couple $(\beta_2, \beta_1)$. On a alors
\[
\begin{split}
l\overline{(\beta_2, \beta_1)} \comp l(\alpha)
& =
w(\beta_2) \comp j(\beta_1) \comp j(\alpha) \\
& =
w(\beta_2) \comp j(\beta_1 \comp \alpha) \\
& =
l\overline{(\beta_2, \beta_1 \comp \alpha)} \\
& =
l(\overline{(\beta_2, \beta_1)} \comp \alpha)\ .
\end{split}
\]

On a donc bien défini un $2$-foncteur strict $l$. On vérifie immédiatement
qu'on a $li' = j$ et $lv = w$.
\tb

\noindent\textsc{Unicité.}
Les formules définissant $l$ sont imposées par la
compatibilité à~$j$ et~$w$. C'est évident pour toutes les formules sauf
éventuellement le troisième cas dans la définition de l'action de $l$ sur
les $1$-flèches et les $2$-flèches. Le cas des $1$-flèches résulte du cas
$1$\nbd-catégorique et celui des $2$-flèches des égalités
\[
\overline{(\alpha_2, \alpha_1)} = \overline{(\alpha_2, \id{})}\comp\alpha_1 =
v(\alpha_2)\comp i'(\alpha_1)\ .
\]
Le foncteur $l$ est donc unique, ce qui achève la démonstration.
\end{proof}

\subsection{Quelques identités préliminaires}

Fixons un crible $i : A \to B$ de $2$\nbd-catégories et une structure de
rétracte par déformation oplax fort sur $i$ donnée par $r : B \to A$ et $H
: \Deltan1 \times B \to B$ (\emph{cf.} définition~\ref{def:rdf}).

\begin{dprop}
Soit $\gamma$ un objet, une $1$-flèche ou une $2$-flèche de $A$. On a
\[ H(0, \gamma) = \gamma\ . \]
\end{dprop}

\begin{proof}
En effet, on a
\[
H(0, \gamma) = ir(\gamma) = iri(\gamma) = i(\gamma) = \gamma\ .\qedhere \]
\end{proof}

\begin{dprop}
Soient $x \xrightarrow{f} y \xrightarrow{g} z$ deux $1$-flèches composables de $B$.
Si $f$ est dans $A$, on a
\begin{align*}
  H(0 \to 1, gf) & = H(0 \to 1, g)f\ .\tag{A1}\label{eq:A1}
\end{align*}
Soient $x \xrightarrowz{\alpha} y \xrightarrowz{\beta} z$ deux $2$-flèches composables de $B$.
Si $\alpha$ est dans $A$, on a
\begin{align*}
H(0 \to 1, \beta\comp\alpha) & = H(0 \to 1, \beta) \comp \alpha\
.\tag{A2}\label{eq:A2}
\end{align*}
\end{dprop}

\begin{proof}
La première identité est une conséquence immédiate de \eqref{item:RDF2} et
de la proposition précédente. (C'est également un cas particulier de la
seconde identité.) Montrons la seconde. Soient
\[
\UseAllTwocells
\xymatrix@C=3pc{
x \rtwocell^f_g{\,\alpha}
&
y \rtwocell^k_l{\,\,\beta}
&
z
}
\]
deux $2$-flèches de $B$ avec $\alpha$ dans $A$. En vertu de la compatibilité
de $H$ à la composition horizontale de $(0 \to 1, \beta)$ et $(0, \alpha)$,
on a
\[
\begin{split}
\MoveEqLeft
H((0 \to 1, l), (0, g))H(0 \to 1, \beta\comp\alpha) \\
& = \big(H(0 \to 1, \beta) \comp H(0, \alpha)\big)
H((0 \to 1, k), (0, f))\ .
\end{split}
\]
Puisque $\alpha$ appartient à $A$, il en est de même de $f$ et $g$, et par
conséquent, en vertu de~\eqref{item:RDF2}, les contraintes $H((0 \to 1, l), (0, g))$ et
$H((0 \to 1, k), (0, f))$ sont triviales. On obtient donc
\[
\begin{split}
H(0 \to 1, \beta\comp\alpha) & = H(0 \to 1, \beta) \comp H(0, \alpha) \\
& = H(0 \to 1, \beta) \comp \alpha\ ,
\end{split}
\]
ce qu'on voulait démontrer.
\end{proof}

\begin{dprop}
Soient $x \xrightarrow{f} y \xrightarrow{g} z \xrightarrow{h} t$ trois
$1$-flèches composables de $B$. Si $f$ est dans $A$, on a
\begin{align*}
H((0 \to 1, h), (0, gf)) & = H((0 \to 1, h), (0, g)) \comp f\
,\tag{B1}\label{eq:B1}\\
H((1, h), (0 \to 1, gf)) & = H((1, h), (0 \to 1, g)) \comp f\
.\tag{B2}\label{eq:B2}
\end{align*}
Si de plus $g$ est dans $A$, on a
\begin{align*}
H((1, hg), (0 \to 1, f)) & = H((1, h), (0 \to 1, gf))\
.\tag{B3}\label{eq:B3}
\end{align*}
En particulier, sous cette dernière hypothèse, on a
\begin{align*}
H((1, hg), (0 \to 1, f)) & = H((1, h), (0 \to 1, \id{z})) \comp gf\
.\tag{B4}\label{eq:B4}
\end{align*}
\end{dprop}

\begin{proof}
 Démontrons ces identités.
  \begin{enumerate}[label=(B\arabic*), wide]
  \item
  En vertu de la condition de cocycle satisfaite par $H$ appliquée à
  \[ \xrightarrow{\hpz(0, f)\hpz} \xrightarrow{\hpz(0, g)\hpz}
  \xrightarrow{(0 \rightarrow 1, h)}\ , \]
  on a
  \[
  \begin{split}
  \MoveEqLeft
  \big(H(0 \to 1, h) \comp H((0, g), (0, f))\big)H((0 \to 1, h),
  (0, gf)) \\
  & = \big(H((0 \to 1, h), (0, g)) \comp H(0, f)\big)
      H((0 \to 1, hg), (0, f))\ .
  \end{split}
  \]
  Les contraintes $H((0, g), (0, f))$ et $H((0 \to 1, hg), (0, f))$ étant
  triviales (la première car $H\vert \{0\} \times B = ir$ est un
  $2$-foncteur strict et la seconde en vertu de \eqref{item:RDF2}), on obtient
  \[
  \begin{split}
  H((0 \to 1, h), (0, gf)) 
  & = H((0 \to 1, h), (0, g)) \comp H(0, f) \\
  & = H((0 \to 1, h), (0, g)) \comp f\ ,
  \end{split}
  \]
  ce qu'on voulait démontrer.

  \tb 
  \item
  En vertu de la condition de cocycle appliquée à
  \[ \xrightarrow{\hpz(0, f)\hpz} \xrightarrow{(0 \rightarrow 1, g)}
  \xrightarrow{\hpz(1, h)\hpz}\ , \]
  on a
  \[
  \begin{split}
  \MoveEqLeft
  \big(H(1, h) \comp H((0 \to 1, g), (0, f))\big)H((1, h),
  (0 \to 1, gf)) \\
  & = \big(H((1, h), (0 \to 1, g)) \comp H(0, f)\big)
      H((0 \to 1, hg), (0, f))\ .
  \end{split}
  \]
  Les contraintes $H((0 \to 1, g), (0, f))$ et $H((0 \to 1, hg), (0,
  f))$ étant triviales en vertu de \eqref{item:RDF2}, on obtient 
  \[
  \begin{split}
  H((1, h), (0 \to 1, gf)) 
  & = H((1, h), (0 \to 1, g)) \comp H(0, f) \\
  & = H((1, h), (0 \to 1, g)) \comp f\ ,
  \end{split}
  \]
  ce qu'on voulait démontrer.
  \tb
  \item
  En vertu de la condition de cocycle appliquée à
  \[ \xrightarrow{(0 \rightarrow 1, f)} \xrightarrow{\hpz(1, g)\hpz}
  \xrightarrow{\hpz(1, h)\hpz}\ , \]
  on a
  \[
  \begin{split}
  \MoveEqLeft
  \big(H(1, h) \comp H((1, g), (0 \rightarrow 1, f))\big)H((1, h),
  (0 \rightarrow 1, gf)) \\
  & = \big(H((1, h), (1, g)) \comp H(0 \rightarrow 1, f)\big)
      H((1, hg), (0 \rightarrow 1, f))
  \end{split}
  \]
  Les contraintes $H((1, g), (0 \rightarrow 1, f))$ et $H((1, h), (1,
  g))$ étant triviales (la première en vertu de~\eqref{item:RDF1} et la
  seconde car $H\,|\,\{1\} \times B = \id{B}$ est un $2$-foncteur strict),
  on obtient immédiatement \eqref{eq:B3}.

  \tb
  \item
  L'identité \eqref{eq:B4} est une conséquence immédiates des identités
  \eqref{eq:B3} et \eqref{eq:B2}.
\end{enumerate}
\end{proof}

\subsection{Image directe de $2$-cribles rétractes par déformation oplax
forts}

Fixons un crible $i : A \to B$ de $2$-catégories et une structure de
rétracte par déformation oplax fort sur $i$ donnée par $r : B \to A$ et $H
: \Deltan1 \times B \to B$. Soit $u : A \to A'$ un $2$-foncteur strict.
Considérons le carré cocartésien
\[
\xymatrix{
A \ar[r]^u \ar[d]_i & A' \ar[d]^{i'}\\
B \ar[r]_v & B' \\
}
\]
de $\dcat$ associé à $i$ et $u$. Le but de cette sous-section est
de montrer qu'il existe une unique structure de $2$-crible rétracte
par déformation oplax fort $(r', H')$ sur le crible $i'$ compatible à $(r,
H)$ au sens où on a $r'v=ur$ et $H'(\id{\Deltan1} \times v)=vH$, ce qui
établira la proposition~\ref{prop:2rdf_img_dir}. Plus précisément, on dira
que $r'$ est compatible à $r$ si on a $r'v = ur$ et que $H'$ est compatible
à $H$ si on a $H'(\id{\Deltan1} \times v)=vH$.

\tb

Dans cette sous-section, nous utiliserons librement les descriptions de $i'$,
$v$ et $B'$ données dans la sous-section~\ref{sec:cribles2Cat}. Les
identités (A$m$) et (B$n$) (pour différentes valeurs de $m$ et $n$)
auxquelles nous ferons référence sont celles qui ont été établies dans la
sous-section précédente.

\begin{dparagr}
En vertu de la propriété universelle des carrés cocartésiens, il existe une
unique rétraction $r' : B' \to A'$ de $i'$
compatible à~$r$ : l'unique $2$-foncteur $r'$ rendant commutatif le diagramme
\[
\xymatrix@R=.4pc@C=.6pc{
A \ar[rrr]^u \ar[ddd]_i & & & A' \ar[ddd]_{i'} \ar@/^1em/[dddddrr]^{\id{A'}}\\
\\
\\
B \ar[rrr]^v \ar@/_1em/[ddrrrrr]_{ur} & & & B' \ar[ddrr]^{r'}  \\
\\
& & & & & A' \pbox{.}
}
\]
\end{dparagr}

\begin{dparagr}\label{paragr:H'_uniq_cell}
Nous allons commencer par montrer qu'il existe au plus une homotopie
oplax~$H'$ de $i'r'$ vers $\id{B'}$ compatible à $H$ faisant de $i'$ un
$2$-crible rétracte par déformation oplax fort.

\tb

Soit $H'$ une telle homotopie. Commençons par déterminer l'action de $H'$
sur les cellules de $\Deltan1 \times B'$.  Pour $\gamma$ un objet, une
$1$-flèche ou une $2$-flèche de $B'$, puisque $H'$ est une homotopie de
$i'r'$ vers $\id{B'}$, on a nécessairement
\[
H'(0, \gamma) = i'r'(\gamma) \quad\text{et}\quad H'(1, \gamma) = \gamma\ .
\]

\tb

Si $f : x \to y$ est une $1$-flèche de $B'$, on a
\[
H'(0 \to 1, f) =
\begin{cases}
f, & \text{si $x, y \in \Ob(A')$,}\\
vH(0 \to 1, f), & \text{si $x, y \in \Ob(B\sauf A)$,}\\
vH(0 \to 1, f_2)f_1, & \text{si $x \in \Ob(A')$, $y \in \Ob(B\sauf A)$} \\
\phantom{vH(0 \to 1, f_2)f_1,} & \text{et $(f_2, f_1)$ est un représentant
de $f$.}
\end{cases}
\]
En effet, la première égalité résulte de la condition~\eqref{item:RDF1}, la deuxième de
la compatibilité de $H'$ à $H$ et la troisième du calcul suivant :
\[
\begin{split}
H'(0 \to 1, \overline{(f_2, f_1}))
& = H'(0 \to 1, v(f_2)f_1) \\
& = H'(0 \to 1, v(f_2))f_1 \\
& = vH(0 \to 1, f_2)f_1\ ,
\end{split}
\]
où l'avant-dernière égalité résulte de l'identité~\eqref{eq:A1} et la dernière de
la compatibilité de $H'$ à $H$.

\tb

Enfin, si $\alpha : x \toz y$ est une $2$-flèche de $B'$, on obtient comme
ci-dessus (en remplaçant l'identité~\eqref{eq:A1} par
l'identité~\eqref{eq:A2}) les formules
\[
H'(0 \to 1, \alpha) =
\begin{cases}
\alpha, & \text{si $x, y \in \Ob(A')$,}\\
vH(0 \to 1, \alpha), & \text{si $x, y \in \Ob(B\sauf A)$,}\\
vH(0 \to 1, \alpha_2) \comp \alpha_1, & \text{si $x \in \Ob(A')$,
$y \in \Ob(B\sauf A)$}\\
\phantom{vH(0 \to 1, \alpha_2) \comp \alpha_1,} & \text{et $(\alpha_2,
\alpha_1)$ est un représentant de $\alpha$.}
\end{cases}
\]
Notons que dans le troisième cas, nous avons défini $H'(0 \to 1, \alpha)$ uniquement
si $\alpha$ est un générateur de $\Homi_{B'}(x, y)$. Ceci détermine
évidemment la valeur de $H'(0 \to 1, \alpha)$ pour un $\alpha$ quelconque.
\end{dparagr}

\begin{dparagr}\label{paragr:H'_uniq_contr}
Soit $H'$ comme dans le paragraphe précédent. Nous allons déterminer les
contraintes de fonctorialité de $H'$.  Soient 
\[ x \xrightarrow{f} y \xrightarrow{g} z \] 
deux $1$-flèches composables de $B'$. De telles $1$-flèches déterminent
quatre suites de $1$\nbd-flèches composables de $\Delta_1 \times B'$ :
\begin{enumerate}[label=(\alph*)]
\item 
$(0, x) \xrightarrow{(0 \rightarrow 1, f)} (1, y) \xrightarrow{\hpz(1,
g)\hpz} (1, z)$\ ;
\item 
$(0, x) \xrightarrow{\hpz(0, f)\hpz} (0, y) \xrightarrow{(0 \rightarrow 1,
g)} (1, z)$\ ;
\item
$(0, x) \xrightarrow{\hpz(0, f)\hpz} (0, y) \xrightarrow{\hpz(0, g)\hpz}
(1, z)$\ ;
\item
$(1, x) \xrightarrow{\hpz(1, f)\hpz} (1, y) \xrightarrow{\hpz(1, g)\hpz}
(1, z)$\ ;
\end{enumerate}
et il s'agit de déterminer les contraintes de fonctorialité
associées à ces quatre suites. Dans les deux derniers cas, puisque les
$2$-foncteurs
\[
H\,|\,\{0\} \times B' = i'r'
\quad\text{et}\quad
H\,|\,\{1\} \times B' = \id{B'}
\]
sont stricts, on a nécessairement
\[
  H'((0, g), (0, f)) = \id{i'r'(gf)}
  \quad\text{et}\quad
  H'((1, g), (1, f)) = \id{gf}\ .
\]
Traitons les deux premiers cas.  On distingue les configurations suivantes :
\ts
\begin{enumerate}[label=(\arabic*), wide]
 \item $x, y, z \in \Ob(A')\ .$
 \ts

   Dans ce cas, la condition~\eqref{item:RDF1} entraîne qu'on a
   \[
   H'((1, g),(0 \to 1, f)) = \id{gf}
   \quad\text{et}\quad
   H'((0 \to 1, g),(0, f)) = \id{gf}\ .
   \]
 \tb
 \item $x, y, z \in \Ob(B\sauf A)\ .$
 \ts

 Dans ce cas, la compatibilité de $H'$ à $H$ impose les relations
 \[
 \begin{split}
 H'((1, g),(0 \to 1, f)) & = vH((1, g),(0 \to 1, f)) \ , \\
 H'((0 \to 1, g),(0, f)) & = vH((0 \to 1, g),(0, f)) \ .
 \end{split}
 \]
 \tb
 \goodbreak
 \item $x, y \in \Ob(A'), \quad z \in \Ob(B\sauf A), \quad g =
   \overline{(g_2, a, g_1)}\ .$
 \ts
 \begin{enumerate}[label=(\alph*), wide=2\parindent]
  \item 
    On a nécessairement
    \[
    H'((1, g), (0 \to 1, f)) = vH((1, g_2), (0 \to 1, \id{a})) \comp g_1f\ .
    \]
    En effet, on a
    \[
    \begin{split}
    H'((1, g), (0 \to 1, f))
    & = H'((1, v(g_2)g_1), (0 \to 1, f)) \\
    & = H'((1, v(g_2)), (0 \to 1, \id{u(a)}))\comp g_1f \\
    & = H'((1, v(g_2)), (0 \to 1, v(\id{a})))\comp g_1f \\
    & = vH((1, g_2), (0 \to 1, \id{a}))\comp g_1f\ ,
    \end{split}
    \]
    où la deuxième égalité résulte de l'identité \eqref{eq:B4} et la dernière de
    la compatibilité de~$H'$ à $H$.
  \tb
  \item La condition~\eqref{item:RDF2} entraîne qu'on a
   \[ H'((0 \to 1, g),(0, f)) = \id{vH(0 \to 1, g_2)g_1f}\ . \]
  \end{enumerate}
  \tb
\item $x \in \Ob(A'), \quad y, z \in \Ob(B\sauf A), \quad f =
   \overline{(f_2, f_1)}\ .$
 \ts
 \begin{enumerate}[label=(\alph*), wide=2\parindent]
  \item 
    On a nécessairement
    \[
    H'((1, g), (0 \to 1, f)) = vH((1, g), (0 \to 1, f_2)) \comp f_1\ .
    \]
    En effet, on a
    \[
    \begin{split}
    H'((1, g), (0 \to 1, f))
    & = H'((1, g), (0 \to 1, v(f_2)f_1)) \\
    & = H'((1, g), (0 \to 1, v(f_2))) \comp f_1 \\
    & = vH((1, g), (0 \to 1, f_2)) \comp f_1\ ,
    \end{split}
    \]
    où l'avant-dernière égalité résulte de l'identité \eqref{eq:B2} et la dernière de
    la compatibilité de $H'$ à $H$.
  \tb
  \item
  \tb
    On a nécessairement
    \[
    H'((0 \to 1, g), (0, f)) = vH((0 \to 1, g), (0, f_2)) \comp f_1\ .
    \]
    En effet, on a
    \[
    \begin{split}
    H'((0 \to 1, g), (0, f))
    & = H'((0 \to 1, g), (0, v(f_2)f_1)) \\
    & = H'((0 \to 1, g), (0, v(f_2))) \comp f_1 \\
    & = vH((0 \to 1, g), (0, f_2)) \comp f_1\ ,
    \end{split}
    \]
    où l'avant-dernière égalité résulte de l'identité~\eqref{eq:B1} et la dernière de
    la compatibilité de $H'$ à $H$.
  \end{enumerate}
\end{enumerate}
\tb
\end{dparagr}

\begin{dparagr}\label{paragr:H'_def_cell}
Le paragraphe précédent achève de démontrer l'unicité de $H'$. Nous allons
maintenant montrer que les formules obtenues dans les deux paragraphes
précédents définissent effectivement une homotopie oplax de $i'r'$ vers
$\id{B'}$ compatible à $H$ faisant de $i'$ une $2$-crible rétracte par
déformation oplax fort.

\tb

Commençons par définir l'action de $H'$ sur les cellules de $\Deltan1 \times
B'$ suivant les formules obtenues dans le
paragraphe~\ref{paragr:H'_uniq_cell}.  Pour $\gamma$ un objet, une
$1$-flèche ou une $2$-flèche de $B'$, on pose
\[
H'(0, \gamma) = i'r'(\gamma) \quad\text{et}\quad H'(1, \gamma) = \gamma\ .
\]
\tb
Pour $f : x \to y$ une $1$-flèche de $B'$, on pose
\[
H'(0 \to 1, f) =
\begin{cases}
f, & \text{si $x, y \in \Ob(A')$,}\\
vH(0 \to 1, f), & \text{si $x, y \in \Ob(B\sauf A)$,}\\
vH(0 \to 1, f_2)f_1, & \text{si $x \in \Ob(A')$, $y \in \Ob(B\sauf A)$} \\
\phantom{vH(0 \to 1, f_2)f_1,} & \text{et $(f_2, f_1)$ est un représentant
de $f$.}
\end{cases}
\]
Vérifions que la troisième formule est bien définie. Supposons qu'on ait
deux couples $(f_2, f_1)$ et $(f'_2, f'_1)$ en relation \forlang{via} une
$1$-flèche $h$ de $A$ (autrement dit, $h$ vérifie \hbox{$f'_1 = u(h)f^{}_1$} et
$f^{}_2 = f'_2h$). On a alors
\[
\begin{split}
H'(0 \to 1, \overline{(f_2, f_1)})
&
= vH(0 \to 1, f_2)f_1
= vH(0 \to 1, f'_2h)f_1 \\
&
= v\big(H(0 \to 1, f'_2)h\big)f_1 
= vH(0 \to 1, f'_2)v(h)f_1 \\
&
= vH(0 \to 1, f'_2)u(h)f_1 
= vH(0 \to 1, f'_2)f'_1 \\
& 
= H'(0 \to 1, \overline{(f'_2, f'_1)})\ ,
\end{split}
\]
où la troisième égalité résulte de l'identité~\eqref{eq:A1}.

\tb

Enfin, pour $\alpha : x \toz y$ une $2$-flèche de $B'$, on pose
\[
H'(0 \to 1, \alpha) =
\begin{cases}
\alpha, & \text{si $x, y \in \Ob(A')$,}\\
vH(0 \to 1, \alpha), & \text{si $x, y \in \Ob(B\sauf A)$,}\\
vH(0 \to 1, \alpha_2) \comp \alpha_1, & \text{si $x \in \Ob(A')$,
$y \in \Ob(B\sauf A)$}\\
\phantom{vH(0 \to 1, \alpha_2) \comp \alpha_1,} & \text{et $(\alpha_2,
\alpha_1)$ est un représentant de $\alpha$.}
\end{cases}
\]
Dans le troisième cas, nous avons défini $H'(0 \to 1, \alpha)$ uniquement
sur les générateurs de $\Homi_{B'}(x, y)$. Nous devons donc vérifier que les
relations \eqref{item:R1}, \eqref{item:R2} et \eqref{item:R3} du paragraphe~\ref{paragr:relR} sont satisfaites. Si
$(\beta_2, \beta_1)$ et $(\alpha_2, \alpha_1)$ sont des générateurs de
$\Homi_{B'}(x, y)$ satisfaisant aux hypothèses de \eqref{item:R1}, on a
\[
\begin{split}
\MoveEqLeft H'(0 \to 1, \overline{(\beta_2, \beta_1)}) 
H'(0 \to 1, \overline{(\alpha_2, \alpha_1)}) \\
& =
\big(vH(0 \to 1, \beta_2) \comp \beta_1\big)
\big(vH(0 \to 1, \alpha_2) \comp \alpha_1\big) \\
& =
\big(vH(0 \to 1, \beta_2)vH(0 \to 1, \alpha_2)\big)
\comp (\beta_1\alpha_1) \\
& =
\big(vH(0 \to 1, \beta_2\alpha_2)\big)
\comp (\beta_1\alpha_1) \\
& =
H'(0 \to 1, \overline{(\beta_2\alpha_2, \beta_1\alpha_1)})\ ,
\end{split}
\]
ce qui établit la compatibilité à \eqref{item:R1}. La compatibilité à
\eqref{item:R2} résulte des égalités
\[
\begin{split}
H'(0 \to 1, \overline{(\id{f_2}, \id{f_1})}) & = vH(0 \to 1,
\id{f_2}) \comp \id{f_1} =  \id{vH(0 \to 1, f_2)f_1} \\
& = \id{H'(0 \to 1, \overline{(f_2, f_1)})}\ .
\end{split}
\]
Enfin, si $(\alpha_2, \alpha_1)$ et $(\alpha'_2, \alpha'_1)$ sont des
générateurs de $\Homi_{B'}(x, y)$ satisfaisant aux hypothèses de
\eqref{item:R3} \forlang{via} une $2$-flèche $\gamma$ de $A$, on a
\[
\begin{split}
H'(0 \to 1, \overline{\sstrut(\alpha_2, \alpha_1)})
& =
vH(0 \to 1, \alpha_2) \comp \alpha_1
= vH(0 \to 1, \alpha'_2 \comp \gamma) \comp \alpha_1 \\
&
= v\big(H(0 \to 1, \alpha'_2) \comp \gamma\big) \comp \alpha_1
= vH(0 \to 1, \alpha'_2) \comp v(\gamma) \comp \alpha_1 \\
& \phantom{=1} \text{(en vertu de l'identité~\eqref{eq:A2})} \\
&
= vH(0 \to 1, \alpha'_2) \comp u(\gamma) \comp \alpha_1 
= vH(0 \to 1, \alpha'_2) \comp \alpha'_1 \\
& 
= H'(0 \to 1, \overline{(\sstrut\alpha'_2, \alpha'_1)})\ ,
\end{split}
\]
ce qui achève la démonstration du fait que la dernière formule est bien
définie.

\tb

On vérifie immédiatement qu'on définit ainsi un morphisme de $2$-graphes
$H'$ et que de plus, celui-ci est compatible aux identités et à la
composition verticale.
\end{dparagr}

\begin{dparagr}\label{paragr:comp_H'_H_gr}
Le paragraphe précédent achève la définition du morphisme de $2$-graphes
sous-jacent à $H'$. Vérifions que ce morphisme de $2$-graphes est compatible
à $H$ au sens où on a l'égalité de morphismes de $2$-graphes
\[
H'(1_{\smp{1}}\times v)=vH\ .
\]
Soit $\gamma$ un objet, une $1$-flèche ou une $2$-flèche
de $B$. On a
\[
H'(0, v(\gamma)) = vH(0, \gamma)\ .
\]
En effet, on a
\[
H'(0, v(\gamma)) = i'r'v(\gamma) = i'ur(\gamma) = vir(\gamma) 
= vH(0, \gamma)\ .
\]
Par ailleurs, on a
\[ H'(1, v(\gamma)) = v(\gamma) = vH(1, \gamma)\ . \]
Supposons maintenant que $\alpha$ est une $1$-flèche ou une $2$-flèche de
$B$ d'objet source $x$ et d'objet but $y$, et montrons que $H'(0 \to 1,
v(\gamma)) = vH(0 \to 1, \gamma)$. On distingue trois cas~:
\tb
\begin{enumerate}[label=(\arabic*)]
  \item Si $x, y \in \Ob(A)$, on a
    \[
      \begin{split}
      H'(0 \to 1, v(\gamma)) & = H'(0 \to 1, u(\gamma)) \\
      & = u(\gamma) = v(\gamma) \\
      & = vH(0 \to 1, \gamma)\ ,
      \end{split}
    \]
    où la dernière égalité résulte du fait que $H$ satisfait à la
    condition~\eqref{item:RDF1}.
  \tb
  \item Si $x, y \in \Ob(B\sauf A)$, on a
   \[ 
    H'(0 \to 1, v(\gamma)) = H'(0 \to 1, \gamma) = vH(0 \to 1, \gamma)\ .
   \]
  \tb
  \item Si $x \in \Ob(A)$ et $y \in \Ob(B\sauf A)$, on a
    \[
    \begin{split}
    H'(0 \to 1, v(\gamma))
    & = H'(0 \to 1, \overline{(\gamma, \id{u(x)})})
      = vH(0 \to 1, \gamma) \comp \id{u(x)} \\
    & = vH(0 \to 1, \gamma)\ .
    \end{split}
    \]
\end{enumerate}
\end{dparagr}

\begin{dparagr}\label{paragr:eq_C}
On aura besoin dans la suite des identités suivantes. Soient 
$x$ un objet de~$A'$ et $y$ un objet de $B\sauf A$. Si $f : x \to y$ 
est une $1$-flèche de $B'$ représentée par un couple $(f_2, f_1)$, on a
\begin{align*}
H'(0, \overline{(f_2, f_1)}) = vH(0, f_2)f_1\ .
\tag{C1}\label{eq:C1}
\end{align*}
De même, si $(\alpha_2, \alpha_1)$ est un générateur de $\Hom_{B'}(x, y)$,
on a
\begin{align*}
H'(0, \overline{(\alpha_2, \alpha_1)}) = vH(0, \alpha_2) \comp \alpha_1\ .
\tag{C2}\label{eq:C2}
\end{align*}
Le première égalité est un cas particulier de la seconde et il suffit donc
de prouver cette dernière. Or, en utilisant le fait que 
le $2$-foncteur $H'\,|\,\{0\}\times B' = i'r'$ est strict et la
compatibilité de $H'$ à $H$ qu'on a établie au paragraphe précédent, on a
\[
H'(0, \overline{(\alpha_2, \alpha_1)})
=
H'(0, v(\alpha_2) \comp \alpha_1)
=
H'(0, v(\alpha_2)) \comp H'(0, \alpha_1)
=
vH(0, \alpha_2) \comp \alpha_1\ .
\]
\end{dparagr}

\begin{dparagr}\label{paragr:H'_def_contr}
Définissons maintenant les contraintes de fonctorialité de $H'$ suivant les
formules obtenues dans le paragraphe~\ref{paragr:H'_uniq_contr}. Soient
\[ x \xrightarrow{f} y \xrightarrow{g} z \]
deux $1$-flèches composables de $B'$. On pose
\[
  H'((0, g), (0, f)) = \id{i'r'(gf)}
  \quad\text{et}\quad
  H'((1, g), (1, f)) = \id{gf}\ .
\]
Examinons les différentes configurations restantes~:
\tb
\begin{enumerate}[label=(\arabic*), wide]
  \item $x, y, z \in \Ob(A')$\ .
    \tb
    On a alors $H'(0 \to 1, gf) = gf$.
    \tb
    \begin{enumerate}[label=(\alph*), wide=2\parindent]
      \item De plus, on a $H'(0 \to 1, g)H'(0, f) = gf$ et on pose
        \[ H'((0 \to 1, g),(0, f)) = \id{gf}\ . \]
      \tb
      \item De même, on a $H'(1, g)H'(0 \to 1, f) = gf$ et on pose
        \[ H'((1, g),(0 \to 1, f)) = \id{gf}\ .  \]
    \end{enumerate}
    \tb
  \item $x, y, z \in \Ob(B\sauf A)$\ .
  \tb
   On a alors $H'(0 \to 1, gf) = vH(0 \to 1, gf)$.
   \tb
    \begin{enumerate}[label=(\alph*), wide=2\parindent]
      \item De plus, par compatibilité de $H'$ à $H$
        (voir le paragraphe~\ref{paragr:comp_H'_H_gr}), on a
        \[
        \begin{split}
        H'(0 \to 1, g)H'(0, f) & = vH(0 \to 1, g)vH(0, f) \\
        & = v\big(H(0 \to 1, g)H(0, f)\big)\ ,
        \end{split}
        \]
        et on pose
        \[ H'((0 \to 1, g), (0, f)) = vH((0 \to 1, g), (0, f))\ . \]
        \tb
      \item On vérifie comme ci-dessus qu'on a
        \[ H'(1, g)H'(0 \to 1, f) = v\big(H(1, g)H(0 \to 1, f)\big)\ , \]
        et on pose
        \[ H'((1, g), (0 \to 1, f)) = vH((1, g), (0 \to 1, f))\ . \]
        \tb
    \end{enumerate} 
    \tb
  \item $x, y \in \Ob(A'), \quad z \in \Ob(B\sauf A),\quad g =
    \overline{(g_2, a, g_1)}$\ .
  \tb
    On a alors
    \[ 
    \begin{split}
    H'(0 \to 1, \overline{(g_2, g_1)}f) 
    & = H'(0 \to 1, \overline{(g_2, g_1f)}) \\
    & = vH(0 \to 1, g_2)g_1f\ .
    \end{split}
    \]
    \tb
    \begin{enumerate}[label=(\alph*), wide=2\parindent]
      \item De plus, on a $H'(0 \to 1, \overline{(g_2, g_1)})H'(0, f) = vH(0 \to 1,
        g_2)g_1f$ et on pose
        \[ H'((0 \to 1, \overline{(g_2, g_1)}), (0, f)) = \id{vH(0 \to 1,
        g_2)g_1f}\ , \]
        ce qui est bien défini puisque $\id{vH(0 \to 1, g_2)g_1f} = \id{H'(0
        \to 1, \overline{(g_2, g_1)}f)}$.
      \tb
      \item Par ailleurs, on a $H'(1, \overline{(g_2, g_1)})H'(0 \to 1, f) =
        \overline{(g_2, g_1)}f = v(g_2)g_1f$. On dispose d'une $2$-flèche
        \[
           H((1, g_2), (0 \to 1, \id{a})) : H(0 \to 1, g_2) \Rightarrow H(1,
           g_2)H(0 \to 1, \id{a}) = g_2
        \]
        (le calcul du but de cette $2$-flèche utilise le fait que $H$
        vérifie la condition~\eqref{item:RDF1}), et on pose
        \[
           H'((1, \overline{(g_2, g_1)}), (0 \to 1, f)) = vH((1, g_2), (0
           \to 1, \id{a})) \comp g_1f\ .
        \]
        Vérifions que cette formule est bien définie. Soit $(g'_2, a',
        g'_1)$ un second triplet en relation avec $(g_2, a, g_1)$ \emph{via}
        une $1$-flèche $h$ de $A$ (autrement dit, $h$ vérifie $g'_1 =
        u(h)g^{}_1$ et \hbox{$g^{}_2 = g'_2h$}). On a alors
        \[
        \begin{split}
           H((1, g_2), (0 \to 1, \id{a})) 
           & = H((1, g'_2h), (0 \to 1, \id{a})) \\
           & = H((1, g'_2), (0 \to 1, \id{a'})) \comp h\ ,
        \end{split}
        \]
        où la dernière égalité résulte de l'identité \eqref{eq:B4}.
        On obtient donc
        \[
        \begin{split}
        vH((1, g_2), (0 \to 1, \id{a})) \comp g_1f
        & = v\big(H((1, g'_2), (0 \to 1, \id{a'})) \comp h\big) \comp g_1f \\
        & = vH((1, g'_2), (0 \to 1, \id{a'})) \comp v(h)g_1f \\
        & = vH((1, g'_2), (0 \to 1, \id{a'})) \comp u(h)g_1f \\
        & = vH((1, g'_2), (0 \to 1, \id{a'})) \comp g'_1f\ .
        \end{split}
        \]
    \end{enumerate}
    \tb
  \item $x \in \Ob(A'), \quad y, z \in \Ob(B\sauf A), \quad f =
    \overline{(f_2, f_1)}$\ .
  \ts

    On a alors 
    \[ H'(0 \to 1, g\overline{(f_2, f_1})) = H'(0 \to 1, \overline{(gf_2,
    f_1)}) = vH(0 \to 1, gf_2)f_1\ . \]
    \tb
    \begin{enumerate}[label=(\alph*), wide=2\parindent]
      \item De plus, on a
        \[
        \begin{split}
          H'(0 \to 1, g)H'(0, \overline{(f_2, f_1)}) 
          & = vH(0 \to 1, g)H'(0, v(f_2)f_1) \\
          & = vH(0 \to 1, g)H'(0, v(f_2))H'(0, f_1) \\
          & \phantom{=1} \text{(car le $2$-foncteur $H'\,|\,\{0\}\times B' =
            i'r'$ est strict)} \\
          & = vH(0 \to 1, g)vH(0, f_2)f_1 \\
          & \phantom{=1} \text{(en vertu de la compatibilité de $H'$ à $H$
          (\emph{cf.}~\ref{paragr:comp_H'_H_gr}))} \\
          & = v\big(H(0 \to 1, g)H(0, f_2)\big)f_1\ ,
        \end{split}
        \]
        et on pose
        \[
        H'((0 \to 1, g), (0, \overline{(f_2, f_1)})) = vH((0 \to 1, g), (0,
        f_2)) \comp f_1\ .
        \]
        Vérifions que cette formule est bien définie. Soit $(f'_2, f'_1)$ un
        second couple en relation avec $(f_2, f_1)$  \emph{via} une
        $1$-flèche $h$ de $A$. On a alors
        \[
        \begin{split}
        H((0 \to 1, g), (0, f_2)) 
        & = H((0 \to 1, g), (0, f'_2h)) \\
        & = H((0 \to 1, g), (0, f'_2)) \comp h\ , \\
        \end{split}
        \]
        où la dernière égalité résulte de l'identité~\eqref{eq:B1}.
        On en déduit immédiatement comme dans le cas précédent que
        \[
        vH((0 \to 1, g), (0, f_2)) \comp f_1 =
        vH((0 \to 1, g), (0, f'_2)) \comp f'_1\ .
        \]
        \tb
   \item On a
        \[
        \begin{split}
        H'(1, g)H'(0 \to 1, \overline{(f_2, f_1)})
        & = gvH(0 \to 1, f_2)f_1 \\
        & = vH(1, g)vH(0 \to 1, f_2)f_1 \\
        & = v\big(H(1, g)H(0 \to 1, f_2)\big)f_1
        \end{split}
        \]
        et on pose
        \[
        H'((1, g), (0 \to 1, \overline{(f_2, f_1)})) = vH((1, g), (0 \to 1,
        f_2)) \comp f_1\ .
        \]
        En vertu de \eqref{eq:B2}, si $(f'_2, f'_1)$ est un couple en relation avec
        $(f_2, f_1)$ \emph{via} une $1$-flèche $h$ de $A$, on a
        \[
        \begin{split}
        H((1, g), (0 \to 1, f_2)) = H((1, g), (0 \to 1, f'_2)) \comp h \\
        \end{split}
        \]
        et on en déduit que
        \[
        vH((1, g), (0 \to 1, f_2)) \comp f_1 =
        vH((1, g), (0 \to 1, f'_2)) \comp f'_1\ .
        \]
    \end{enumerate}
\end{enumerate}
\end{dparagr}

\begin{dparagr}\label{paragr:comp_H'_H}
Le paragraphe précédent achève la définition de $H'$. Il résulte immédiatement du fait que $H$
satisfait aux conditions de normalisation des $2$-foncteurs lax
(\emph{cf.}~définition~\ref{def:fonct_lax}) qu'il en est de même de $H'$. Il
est par ailleurs immédiat que $H'$~vérifie la condition~\eqref{item:RDF1}
(voir entre autres le cas~(1) du paragraphe précédent) et la
condition~\eqref{item:RDF2} (voir le cas (3.a) du paragraphe précédent). 
\tb
Vérifions la compatibilité de $H'$ à $H$, c'est-à-dire l'égalité
\[
H'(1_{\smp{1}}\times v)=vH\ .
\]
On a montré au paragraphe~\ref{paragr:comp_H'_H_gr} que
$H'(1_{\smp{1}}\times v)$ et $vH$ sont égaux en tant que morphismes de
$2$-graphes. Il s'agit donc de vérifier que leurs contraintes de
fonctorialité coïncident. C'est évident sauf éventuellement pour la
contrainte associée à
\[
  (0, x) \xrightarrow{(0 \rightarrow 1, f)} (1, y) \xrightarrow{\hpz(1,
  g)\hpz} (1, z), \qquad x, y \in \Ob(A), \quad z \in \Ob(B \sauf A)\ .
\]
Mais dans ce cas, on a
\[
\begin{split}
H'(1_{\smp{1}}\times v)((1, g), (0 \to 1, f))
& = H'((1, v(g)), (0 \to 1, v(f))) \\
& = H'((1, \overline{(g, \id{u(y)})}), (0 \to 1, u(f))) \\
& = vH((1, g), (0 \to 1, \id{u(y)})) \comp u(f) \\
& = vH((1, g), (0 \to 1, \id{u(y)})) \comp v(f) \\
& = v\big(H((1, g), (0 \to 1, \id{u(y)})) \comp f\big) \\
& = vH((1, g), (0 \to 1, f))\ ,
\end{split}
\]
où la dernière égalité résulte de l'identité \eqref{eq:B2}.
\end{dparagr}

\begin{dparagr}
Il nous suffit maintenant de montrer que $H'$ est bien un $2$-foncteur oplax.
Commençons par montrer que $H'$ vérifie la condition de cocycle.
Soient \[ x \xrightarrow{f} y \xrightarrow{g} z \xrightarrow{h} t \] trois
$1$-flèches composables de $B'$. De telles $1$-flèches déterminent cinq
suites de $1$\nbd-flèches composables de $\Delta_1 \times B'$ :
\begin{enumerate}[label=(\alph*)]
\item 
$(0, x) \xrightarrow{(0 \rightarrow 1, f)} (1, y) \xrightarrow{\hpz(1,
g)\hpz} (1, z) \xrightarrow{\hpz(1, h)\hpz} (1, t)$\ ;
\item 
$(0, x) \xrightarrow{\hpz(0, f)\hpz} (0, y) \xrightarrow{(0 \rightarrow 1,
g)} (1, z) \xrightarrow{\hpz(1, h)\hpz} (1, t)$\ ;
\item
$(0, x) \xrightarrow{\hpz(0, f)\hpz} (0, y) \xrightarrow{\hpz(0, g)\hpz} (0,
z) \xrightarrow{(0 \rightarrow 1, h)} (1, t)$\ ;
\item
$(0, x) \xrightarrow{\hpz(0, f)\hpz} (0, y) \xrightarrow{\hpz(0, g)\hpz} (0,
z) \xrightarrow{(0, h)} (0, t)$\ ;
\item
$(1, x) \xrightarrow{\hpz(1, f)\hpz} (1, y) \xrightarrow{\hpz(1, g)\hpz} (1,
z) \xrightarrow{(1, h)} (1, t)$\ ;
\end{enumerate}
et il s'agit de montrer que les conditions de cocycle associées à ces cinq
suites sont satisfaites par $H'$. Les deux derniers cas sont immédiats car
$H'\,|\,\{0\} \times B' = i'r' $ et $H'\,|\,\{1\} \times B' =
\id{B'}$ sont des $2$-foncteurs stricts. Traitons les trois premiers cas.
Les trois suites correspondantes sont caractérisées par la «~position de $0
\to 1$~» et les relations de cocycle associées s'écrivent respectivement
\begin{align*}
\MoveEqLeft
\tag*{\hskip 1em(\text{a})}
\big(H'(1,h) \comp H'((1,g),(0\to1,f))\big)H'((1,h),(0\to1,gf))\\
&=\big(H'((1,h),(1,g))\comp H'(0\to1,f)\big)H'((1,hg),(0\to1,f))\ , 
\\
\tag*{\hskip 1em(\text{b})}
\MoveEqLeft
\big(H'(1,h) \comp H'((0\to1,g),(0,f))\big)H'((1,h),(0\to1,gf))\\
&=\big(H'((1,h),(0\to1,g))\comp H'(0,f)\big)H'((0\to1,hg),(0,f))\ ,
\\
\tag*{\hskip 1em(\text{c})}
\MoveEqLeft
\big(H'(0\to1,h) \comp H'((0,g),(0,f))\big)H'((0\to1,h),(0,gf))\\
&=\big(H'((0\to1,h),(0,g))\comp H'(0,f)\big)H'((0\to1,hg),(0,f))\ .
\end{align*}
On distingue les configurations suivantes~:
\tb
\begin{enumerate}[label=(\arabic*), wide]
 \item $x, y, z, t \in \Ob(A')\ .$
 \tb
   Dans ce cas, les contraintes de fonctorialité de $H'$ sont triviales
   (quelque soit la position de $0 \to 1$) et la condition de cocycle est
   donc satisfaite.
 \tb
 \item $x, y, z, t \in \Ob(B\sauf A)\ .$ 
 \tb
      Ce cas est immédiat car on a
      \[ H'\,|\,\Delta_1 \times (B\sauf A) = vH\,|\,\Delta_1 \times (B \sauf A)\ . \]
  \tb
  \item $x, y, z \in \Ob(A'), \quad t \in \Ob(B\sauf A), \quad h =
   \overline{(h_2, a, h_1)}\ .$
  \tb
    \begin{enumerate}[label=(\alph*), wide=2\parindent]
      \item Si $0 \to 1$ est placé en première position, il s'agit de montrer,
en utilisant le fait que les contraintes $H'((1, g), (0 \to 1, f))$ et
$H'((1, h), (1, g))$ sont triviales, que
        \[
        H'((1, \overline{(h_2, h_1)}), (0 \to 1, gf)) 
        =
        H'((1, \overline{(h_2, h_1g)}), (0 \to 1, f))\ .
        \]
        Or ces $2$-flèches sont toutes les deux égales à
        \[ vH((1, h_2), (0 \to 1, \id{a})) \comp h_1gf\ . \]
      \tb
      \item Si $0 \to 1$ est placé en deuxième position, les contraintes
        \[ H'((0 \to 1, g), (0, f))\quad\text{et}\quad H'((0 \to 1, hg), (0, f)) \]
        sont triviales ($H'$ vérifie la condition~\eqref{item:RDF2}) et il
        s'agit de montrer, après simplification, que
        \[
        H'((1, \overline{(h_2, h_1)}), (0 \to 1, gf)) 
        =
        H'((1, \overline{(h_2, h_1)}), (0 \to 1, g)) \comp f\ .
        \]
        Or ces $2$-flèches sont toutes les deux égales à
        \[ vH((1, h_2), (0 \to 1, \id{a})) \comp h_1gf\ . \]
      \tb
      \item Si $0 \to 1$ est placé en troisième position, toutes les
        contraintes sont triviales.
    \end{enumerate}
 \tb
 \item $x, y \in \Ob(A'), \quad z, t \in \Ob(B\sauf A), \quad g =
   \overline{(g_2, a, g_1)}\ .$
 \ts
    \begin{enumerate}[label=(\alph*), wide=2\parindent]
      \item Si $0 \to 1$ est placé en première position, la contrainte
        $H'((1, h), (1, g))$ est triviale et il s'agit
        de montrer que
        \[
        \begin{split}
        \MoveEqLeft
        \big(h \comp H'((1, \overline{(g_2, g_1)}), (0 \to 1, f))\big)
        H'((1, h), (0 \to 1, \overline{(g_2, g_1f)})) \\
        & =
        H'((1, \overline{(hg_2, g_1)}), (0 \to 1, f))\ .
        \end{split}
        \]
        Or
        \[
        \begin{split}
        \MoveEqLeft
        \big(h \comp H'((1, \overline{(g_2, g_1)}), (0 \to 1, f))\big)
        H'((1, h), (0 \to 1, \overline{(g_2, g_1f)})) \\
        & =
        \big(h \comp vH((1, g_2), (0 \to 1, \id{a})) \comp g_1f\big)
        \big(vH((1, h), (0 \to 1, g_2)) \comp g_1f\big) \\
        & =
        v\big[\big(h \comp H((1, g_2), (0 \to 1, \id{a}))\big)
        H((1, h), (0 \to 1, g_2))\big] \comp g_1f \\
        & =
        v(H((1, hg_2), (0 \to 1, \id{a}))) \comp g_1f \\
        & \phantom{=1} \text{(par la condition de cocycle appliquée à
          \raisebox{0pt}[0pt][0pt]{$\xrightarrow{(0 \rightarrow 1, \id{a})}
          \xrightarrow{\hpz(1, g_2)\hpz} \xrightarrow{\hpz(1, h)\hpz}$})}
        \\
        & = H'((1, \overline{(hg_2, g_1)}), (0 \to 1, f))\ .
        \end{split}
        \]
      \tb
      \item Si $0 \to 1$ est placé en deuxième position, les contraintes
        \[ 
           H'((0 \to 1, g), (0, f))
           \quad\text{et}\quad
           H'((0 \to 1, hg), (0, f))
        \]
        sont triviales ($H'$ vérifie la condition \eqref{item:RDF2})
        et il s'agit de montrer que
        \[
        H'((1, h), (0 \rightarrow 1, \overline{(g_2, g_1f)}))
        =
        H'((1, h), (0 \rightarrow 1, \overline{(g_2, g_1)})) \comp f\ .
        \]
        Or ces $2$-flèches sont toutes les deux égales à
        \[ vH((1, h), (0 \to 1, g_2)) \comp g_1f\ . \]
      \tb
      \item Si $0 \to 1$ est placé en troisième position, les contraintes
        $H'((0, g), (0, f))$ et $H'((0 \to 1, hg), (0, f))$ sont triviales et
il s'agit de montrer que
        \[
        H'((0 \to 1, h), (0, \overline{(g_2, g_1f)}))
        =
        H'((0 \to 1, h), (0, \overline{(g_2, g_1)})) \comp f\ .
        \]
        Or ces $2$-flèches sont toutes les deux égales à
        \[ vH((0 \to 1, h), (0, g_2)) \comp g_1f\ . \]
\end{enumerate}
 \tb
 \item $x \in \Ob(A'), \quad y, z, t \in \Ob(B\sauf A), \quad f =
   \overline{(f_2, f_1)}\ .$
 \ts
    \begin{enumerate}[label=(\alph*), wide=2\parindent]
      \item Si $0 \to 1$ est placé en première position, la contrainte
        $H'((1, h), (1, g))$ est triviale et il s'agit de montrer que
        \[
        \begin{split}
        \MoveEqLeft
        \big(h \comp H'((1, g), (0 \to 1, \overline{(f_2, f_1)}))\big)
        H'((1, h), (0 \to 1, \overline{(gf_2, f_1)})) \\
        & =
        H'((1, hg), (0 \to 1, \overline{(f_2, f_1)}))\ ,
        \end{split}
        \]
        ou encore, en développant, que
        \[
        \begin{split}
        \MoveEqLeft
        \big(h \comp vH((1, g), (0 \to 1, f_2)) \comp f_1\big)
        \big(vH((1, h), (0 \to 1, gf_2)) \comp f_1\big) \\
        & =
        vH((1, hg), (0 \to 1, f_2)) \comp f_1\ .
        \end{split}
        \]
        Cela résulte, comme dans le cas (4.a), de la condition de cocycle
        de $H$ appliquée à
        \[
        \xrightarrow{(0 \rightarrow 1, f_2)} \xrightarrow{\hpz(1,
        g)\hpz} \xrightarrow{\hpz(1, h)\hpz}\ .
        \]
      \tb
      \item Si $0 \to 1$ est en deuxième position, il s'agit de montrer que
        \[
        \begin{split}
        \MoveEqLeft
        \big(h \comp H'((0 \to 1, g), (0, \overline{(f_2, f_1)}))\big)
        H'((1, h), (0 \to 1, \overline{(gf_2, f_1)})) \\
        & =
        \big(H'((1, h), (0 \to 1, g)) \comp H'(0, \overline{(f_2, f_1)})\big)
        H'((0 \to 1, hg), (0, \overline{(f_2, f_1)}))\ ,
        \end{split}
        \]
        ou encore, en développant (en utilisant en particulier l'identité
        \eqref{eq:C1} du paragraphe~\ref{paragr:eq_C}), que
        \[
        \begin{split}
        \MoveEqLeft
        \big(h \comp vH((0 \to 1, g), (0, f_2)) \comp f_1\big)
        \big(vH((1, h), (0 \to 1, gf_2)) \comp f_1\big) \\
        & =
        \big(vH((1, h), (0 \to 1, g)) \comp vH(0, f_2)f_1\big)
        \big(vH((0 \to 1, hg), (0, f_2)) \comp f_1\big)\ .
        \end{split}
        \]
        Cela résulte, comme dans le cas (4.a), de la condition de cocycle
        de $H$ appliquée à
        \[
        \xrightarrow{\hpz(0, f_2)\hpz} \xrightarrow{(0 \rightarrow 1, g)}
        \xrightarrow{\hpz(1, h)\hpz}\ .
        \]
      \tb
      \item Si $0 \to 1$ est en troisième position, la contrainte
        $H'((0, g), (0, f))$ est triviale et il s'agit de montrer que
        \[
        \begin{split}
        \MoveEqLeft
        H'((0 \to 1, h), (0, \overline{(gf_2, f_1)})) \\
        & =
        \big(H'((0 \to 1, h), (0, g)) \comp H'(0, \overline{(f_2, f_1)})\big)
        H'((0 \to 1, hg), (0, \overline{(f_2, f_1)}))\ ,
        \end{split}
        \]
        ou encore, en développant (en utilisant en particulier l'identité
        \eqref{eq:C1} du paragraphe~\ref{paragr:eq_C}), que
        \[
        \begin{split}
        \MoveEqLeft
        vH((0 \to 1, h), (0, gf_2)) \comp f_1 \\
        & =
        \big(vH((0 \to 1, h), (0, g)) \comp vH(0, f_2)f_1\big)
        \big(vH((0 \to 1, hg), (0, f_2)) \comp f_1\big)\ .
        \end{split}
        \]
        Cela résulte, comme dans le cas (4.a), de la condition de cocycle
        de $H$ appliquée à
        \[
        \xrightarrow{\hpz(0, f_2)\hpz} \xrightarrow{\hpz(0, g)\hpz}
        \xrightarrow{(0 \rightarrow 1, h)}\ .
        \]
    \end{enumerate}
\end{enumerate}
\end{dparagr}

\begin{dparagr}
Vérifions enfin que $H'$ est compatible à la composition horizontale. Soient
\[
\UseAllTwocells
\xymatrix@C=3pc{
x \rtwocell^f_g{\,\alpha}
&
y \rtwocell^k_l{\,\,\beta}
&
z
}
\]
deux $2$-flèches composables horizontalement de $B'$.  De telles $2$-flèches
déterminent quatre suites de $2$\nbd-flèches composables horizontalement de
$\Delta_1 \times B'$ :
\begin{enumerate}[label=(\alph*)]
\item 
$(0, x) \xrightarrowz{(0 \rightarrow 1, \alpha)} (1, y) \xrightarrowz{\hpz(1,
\beta)\hpz} (1, z)$\ ;
\item 
$(0, x) \xrightarrowz{\hpz(0, \alpha)\hpz} (0, y) \xrightarrowz{(0 \rightarrow 1,
\beta)} (1, z)$\ ;
\item
$(0, x) \xrightarrowz{\hpz(0, \alpha)\hpz} (0, y) \xrightarrowz{\hpz(0, \beta)\hpz} (0,
z)$\ ;
\item
$(1, x) \xrightarrowz{\hpz(1, \alpha)\hpz} (1, y) \xrightarrowz{\hpz(1, \beta)\hpz} (1,
z)$\ ;
\end{enumerate}
et il s'agit de vérifier la compatibilité de $H'$ aux compositions
horizontales associées. Les deux derniers cas sont immédiats car
$H'\,|\,\{0\} \times B' = i'r' $ et $H'\,|\,\{1\} \times B' = \id{B'}$
sont des $2$-foncteurs stricts. Traitons les deux premiers cas. On distingue
les configurations suivantes~:
\tb
\begin{enumerate}[label=(\arabic*), wide]
 \item $x, y, z \in \Ob(A')\ .$
 \tb
   Dans ce cas, la compatibilité à la composition verticale résulte du fait
   que $H'$ vérifie la condition \eqref{item:RDF1}.
 \tb
 \item $x, y, z \in \Ob(B\sauf A)\ .$
 \tb
   Ce cas est immédiat car on a
   \[ H'\,|\,\Delta_1 \times (B\sauf A) = vH\,|\,\Delta_1 \times (B \sauf A)\ . \]
 \tb
 \item $x, y \in \Ob(A'), \quad z \in \Ob(B\sauf A)\ .$
 \tb
        Il suffit de vérifier la compatibilité dans le cas où $\beta$ est un
        générateur de $\Homi_{B'}(y, z)$. On a alors
        \[
         k = \overline{(k_2, a, k_1)}, \quad l = \overline{(l_2, a, l_1)}
         \quad\text{et}\quad
         \beta = \overline{(\beta_2, a, \beta_1)}\ .
         \]
    \tb
    \begin{enumerate}[label=(\alph*), wide=2\parindent]
      \item Si $0 \to 1$ est placé en première position, il s'agit de
        montrer que
        \[
         \begin{split}
         \MoveEqLeft \big(H'(1, \overline{(\beta_2, \beta_1)}) \comp H'(0 \to 1,
         \alpha)\big)
         H'((1, \overline{(k_2, k_1)}), (0 \to 1, f)) \\
         & =
         H'((1, \overline{(l_2, l_1)}), (0 \to 1, g))
         H'(0 \to 1, \overline{(\beta_2, \beta_1 \comp \alpha)})\ .
        \end{split}
        \]
        Or on a
        {\allowdisplaybreaks
        % underfull vbox
        \begin{align*}
         \MoveEqLeft \big(H'(1, \overline{(\beta_2, \beta_1)}) \comp H'(0 \to 1,
         \alpha)\big)
         H'((1, \overline{(k_2, k_1)}), (0 \to 1, f)) \\
         & =
         \big(\overline{(\beta_2, \beta_1)} \comp \alpha\big)
         \big(vH((1, k_2), (0 \to 1, \id{a})) \comp k_1f\big) \\
         & =
         \big(v(\beta_2) \comp \beta_1 \comp \alpha\big)
         \big(vH((1, k_2), (0 \to 1, \id{a})) \comp k_1f\big) \\
         & =
         \big(v(\beta_2)vH((1, k_2), (0 \to 1, \id{a}))\big)
         \comp \beta_1 \comp \alpha \\
         & =
         v\big(H(1, \beta_2)H((1, k_2), (0 \to 1, \id{a}))\big)
         \comp \beta_1 \comp \alpha \\
         & =
         v\big[\big(H(1, \beta_2)\comp H(0 \to 1, \id{a})\big) H((1, k_2), (0
           \to 1, \id{a}))\big]
         \comp \beta_1 \comp \alpha \\*
         & \phantom{=1} \text{(car $H(0 \to 1, \id{a}) = \id{a}$ par
         \eqref{item:RDF1})} \\
         & =
         v\big[H((1, l_2), (0 \to 1, \id{a})) H((1, \beta_2) \comp (0 \to
           1, \id{a}))\big] \comp \beta_1 \comp \alpha \\*
         & \phantom{=1} \text{(par compatibilité de $H$ à la composition
         horizontale de $(0 \rightarrow 1, \id{a})$ et $(1, \beta_2)$)} \\
         & =
         v\big[H((1, l_2), (0 \to 1, \id{a})) H(0 \to 1, \beta_2)\big]
         \comp \beta_1 \comp \alpha \\
         & =
         \big(vH((1, l_2), (0 \to 1, \id{a})) \comp l_1g\big)
         \big(vH(0 \to 1, \beta_2) \comp \beta_1 \comp \alpha \big) \\
         & =
         H'((1, \overline{(l_2, l_1)}), (0 \to 1, g))
         H'(0 \to 1, \overline{(\beta_2, \beta_1 \comp \alpha)})\ .
        \end{align*}
        }
      \tb
      \item Si $0 \to 1$ est placé en seconde position, les contraintes de
        fonctorialité sont triviales ($H'$ vérifie la condition~\eqref{item:RDF2}) et
        il s'agit de montrer que
        \[
         H'(0 \to 1, \overline{(\beta_2, \beta_1)}) \comp H'(0, \alpha)
         =
         H'(0 \to 1, \overline{(\beta_2, \beta_1 \comp \alpha)})\ .
        \]
        Or ces $2$-flèches sont toutes les deux égales à
        \[
        vH(0 \to 1, \beta_2) \comp \beta_1 \comp \alpha\ .
        \]
    \end{enumerate}
  \tb
  \item $x \in \Ob(A'), \quad y, z \in \Ob(B\sauf A)\ .$
  \ts

        Il suffit de vérifier la compatibilité dans le cas où $\alpha$ est un
        générateur de $\Homi_{B'}(x, y)$. On a alors
        \[
        f = \overline{(f_2, f_1)}, \quad g = \overline{(g_2, g_1)}
        \quad\text{et}\quad
        \alpha = \overline{(\alpha_2, \alpha_1)}\ .
        \]
    \tb
    \begin{enumerate}[label=(\alph*), wide=2\parindent]
 \item Si $0 \to 1$ est placé en première position, il s'agit de
        montrer que
        \[
         \begin{split}
         \MoveEqLeft \big(H'(1, \beta) \comp H'(0 \to 1,
         \overline{(\alpha_2, \alpha_1)})\big)
         H'((1, k), (0 \to 1, \overline{(f_2, f_1)})) \\
         & =
         H'((1, l), (0 \to 1, \overline{(g_2, g_1)}))
         H'(0 \to 1, \overline{(\beta \comp \alpha_2, \alpha_1)})\ .
        \end{split}
        \]
        Or on a
        \[
         \begin{split}
         \MoveEqLeft \big(H'(1, \beta) \comp H'(0 \to 1,
         \overline{(\alpha_2, \alpha_1)})\big)
         H'((1, k), (0 \to 1, \overline{(f_2, f_1)})) \\
         & =
         \big(vH(1, \beta) \comp vH(0 \to 1, \alpha_2) \comp \alpha_1\big)
         \big(vH((1, k), (0 \to 1, f_2)) \comp f_1\big) \\
         & =
         v\big[
         \big(H(1, \beta) \comp H(0 \to 1, \alpha_2)\big)
         H((1, k), (0 \to 1, f_2))
         \big]
         \comp \alpha_1 \\
         & =
         v\big[
         H((1, l), (0 \to 1, g_2))
         H(0 \to 1, \beta \comp \alpha_2)
         \big]
         \comp \alpha_1 \\
         & \phantom{=1} \text{(par compatibilité de $H$ à la composition
         horizontale de $(0 \rightarrow 1, \alpha_2)$ et $(1, \beta)$)} \\
         & =
         \big(vH((1, l), (0 \to 1, g_2)) \comp g_1\big)
         \big(vH(0 \to 1, \beta \comp \alpha_2) \comp \alpha_1\big) \\
         & =
         H'((1, l), (0 \to 1, \overline{(g_2, g_1)}))
         H'(0 \to 1, \overline{(\beta \comp \alpha_2, \alpha_1)})\ .
         \end{split}
         \]
      \tb 
      \item Si $0 \to 1$ est placé en seconde position, il s'agit de montrer
        que
        \[
         \begin{split}
         \MoveEqLeft \big(H'(0 \to 1, \beta) \comp H'(0,
         \overline{(\alpha_2, \alpha_1)})\big)
         H'((0 \to 1, k), (0, \overline{(f_2, f_1)})) \\
         & =
         H'((0 \to 1, l), (0, \overline{(g_2, g_1)}))
         H'(0 \to 1, \overline{(\beta \comp \alpha_2, \alpha_1)})\ .
        \end{split}
        \]
        Or on a
        \[
         \begin{split}
         \MoveEqLeft \big(H'(0 \to 1, \beta) \comp H'(0,
         \overline{(\alpha_2, \alpha_1)})\big)
         H'((0 \to 1, k), (0, \overline{(f_2, f_1)})) \\
         & =
         \big(vH(0 \to 1, \beta) \comp vH(0, \alpha_2) \comp \alpha_1\big)
         \big(vH((0 \to 1, k), (0, f_2)) \comp f_1\big) \\
         & \phantom{=1} \text{(en vertu de l'identité \eqref{eq:C2} du
         paragraphe~\ref{paragr:eq_C})} \\
         & =
         v\big[
         \big(H(0 \to 1, \beta) \comp H(0, \alpha_2)\big)
         H((0 \to 1, k), (0, f_2))
         \big]
         \comp \alpha_1 \\
         & =
         v\big[
         H((0 \to 1, l), (0, g_2))
         H(0 \to 1, \beta \comp \alpha_2)
         \big]
         \comp \alpha_1 \\
         & \phantom{=1} \text{(par compatibilité de $H$ à la composition
         horizontale de $(0, \alpha_2)$ et $(0 \to 1, \beta)$)} \\
         & =
         \big(vH((0 \to 1, l), (0, g_2)) \comp g_1\big)
         \big(vH(0 \to 1, \beta \comp \alpha_2) \comp \alpha_1\big) \\
         & =
         H'((0 \to 1, l), (0, \overline{(g_2, g_1)}))
         H'(0 \to 1, \overline{(\beta \comp \alpha_2, \alpha_1)})\ .
         \end{split}
         \]
    \end{enumerate}
\end{enumerate}
\end{dparagr}

\begin{dparagr}
Le paragraphe précédent achève de montrer que $H'$ est un foncteur oplax et
donc que $r'$ et $H'$ définissent une structure de rétracte par déformation
oplax fort sur~$i'$ compatible à $r$ et $H$. Ceci termine la démonstration
de la proposition~\ref{prop:2rdf_img_dir}.
\end{dparagr}

\section{Objets cofibrants de la structure à la Thomason sur
\texorpdfstring{$\nCat{2}$}{2-Cat}}

Le but de cette section est de présenter une généralisation à $\nCat{2}$ du théorème de Thomason affirmant que les objets cofibrants de la structure de catégorie de modèles qu'il a définie sur $\cat$ sont des ensembles ordonnés (\cite[proposition~5.7]{Th}, dont la preuve légèrement incorrecte est corrigée dans~\cite{Ci0}).

\begin{paragr}\label{trbetint}
Dans ce qui suit, on note $\trb:\nCat{2}\to\cat$ le foncteur de troncation bête (\emph{cf}.~\ref{tronq}), adjoint à droite  de l'inclusion pleine $\cat\to\nCat{2}$. Le foncteur $\trb$ admet lui-même un adjoint à droite, associant à une petite catégorie $C$ la $2$\nbd-catégorie ayant les mêmes objets et les mêmes $1$\nbd-flèches que $C$, et exactement une $2$-flèche de source~$f_1$ et de but $f_2$, pour chaque couple $f_1,f_2$ de $1$\nbd-flèches parallèles. En particulier, le foncteur $\trb$ commute à la fois aux limites projectives et aux limites inductives. On note $\tri:\nCat{2}\to\cat$ le foncteur de troncation intelligente (\emph{cf}.~\ref{tronq}), adjoint à gauche de l'inclusion pleine $\cat\to\nCat{2}$.
\end{paragr}

\begin{thm}\label{thm:cof2Thom}
Soit $C$ une $2$\nbd-catégorie. Si $C$ est un objet cofibrant pour la structure de catégorie de modèles à la Thomason sur $\nCat{2}$ \emph{(\emph{cf}.~théorème~\ref{thm:cmf_2cat}),} alors:
\begin{enumerate}
\item $\trb(C)$ est une catégorie librement engendrée par un graphe;
\item $\tri(C)$ est un ensemble ordonné;
\item pour tous objets $x,y$ de $C$, la catégorie $\sHom_C(x,y)$ est un ensemble ordonné, autrement dit, $C$ est une catégorie enrichie en ensembles ordonnés.
\end{enumerate}
\end{thm}

Le reste de cette section est principalement consacré à la démonstration de ce théorème qui résultera des corollaires~\ref{coflibre}, \ref{coford} et~\ref{cofcatenrord}. 

\begin{lemme}\label{trblibre}
Pour tout ensemble ordonné $E$, la catégorie $\trb c_2N(E)$ est la catégorie libre engendrée par le graphe des flèches de $E$ qui ne sont pas des identités, autrement dit, le graphe dont les sommets sont les éléments de l'ensemble ordonné $E$ et les arêtes les parties $\{x,y\}\subset E$ telles que $x<y$, la source \emph{(resp.}~le but\emph{)} d'une telle arête étant le sommet $x$ \emph{(resp.}~$y$\emph{)}.
\end{lemme}

\begin{proof}
Le lemme est conséquence immédiate de la description de la $2$\nbd-catégorie $\Orntg{E}$ donnée au paragraphe~\ref{orntg} et de la proposition~\ref{orntgadj}.
\end{proof}

\begin{paragr}
On note $K$ l'ensemble des flèches de $\cat$ formé de l'inclusion $\iota:\varnothing\to\smp{0}$ de la catégorie vide dans la catégorie ponctuelle, et de l'inclusion \hbox{$\partial:\{0,1\}\to\smp{1}$} de la catégorie discrète d'ensemble d'objets $\{0,1\}$ dans $\smp{1}$.
\end{paragr}

\begin{prop}\label{trbcell}
Pour tout monomorphisme $E\to F$ de la catégorie $\ord$, le foncteur $\trb c_2N(E)\to\trb c_2N(F)$ appartient à $\cell(K)$, autrement dit, il est un composé transfini d'images directes de flèches appartenant à $K=\{\iota,\partial\}$.
\end{prop}

\begin{proof}
On remarque qu'un foncteur $C\to D$ appartient à $\cell(K)$ si et seulement si $D$ est obtenue à partir de $C$ en ajoutant formellement des objets et des flèches, et si de plus $C\to D$ est le foncteur canonique correspondant. Or, en vertu du lemme précédent, les catégories $\trb c_2N(E)$ et $\trb c_2N(F)$ sont librement engendrées par des graphes qu'on notera $G_E$ et $G_F$. Par définition, $G_F$ est obtenu à partir de $G_E$ en ajoutant des sommets et des arêtes. Puisque le foncteur catégorie libre commute aux limites inductives, la catégorie $\trb c_2N(F)$ est obtenue à partir de la catégorie $\trb c_2N(E)$ en ajoutant formellement des objets et des flèches. De plus, $\trb c_2N(E) \to \trb c_2N(F)$ est le foncteur canonique correspondant, ce qui démontre la proposition.
\end{proof}

\begin{cor}\label{coflrK}
Si $A\to B$ est une cofibration de la structure de catégorie de modèles à la Thomason sur $\nCat{2}$, alors le foncteur $\trb(A)\to\trb(B)$ appartient à~$lr(K)$.
\end{cor}

\begin{proof}
Vu que le foncteur $\trb$ commute aux images directes, aux composés transfinis et aux rétractes, et que $lr(K)$ est stable par ces opérations, il suffit de prouver que l'image par $\trb$ de tout $2$-foncteur appartenant à l'ensemble $c_2\Sd^2(I)$ des cofibrations génératrices (\emph{cf}.~théorème~\ref{thm:cmf_2cat}) est dans $lr(K)$, ce qui résulte de la proposition précédente et du lemme~\ref{Sd2can}. 
\end{proof}

\begin{paragr}\label{defbieq}
On dit qu'une $1$\nbd-flèche $f:x\to y$ d'une $2$\nbd-catégorie $C$ est \emph{faiblement inversible} s'il existe une $1$\nbd-flèche $g:y\to x$ de $C$ et des $2$\nbd-flèches inversibles pour la composition verticale $gf\Rightarrow1_x$ et $fg\Rightarrow1_y$. On dit qu'un $2$\nbd-foncteur $u:A\to B$ est une \emph{biéquivalence} si pour tout couple d'objets $x_0,x_1$ de $A$, le foncteur 
$$\sHom(x_0,x_1)\toto\sHom(u(x_0),u(x_1))\ ,$$
induit par $u$, est une équivalence de catégories, et si pour tout objet $y$ de $B$, il existe un objet $x$ de $A$ et une $1$\nbd-flèche faiblement inversible $u(x)\to y$. On rappelle le théorème suivant dû à S.~Lack~\cite{Lack1,Lack2}:
\end{paragr}

\begin{thm}\label{Lack}
Il existe une structure de catégorie de modèles combinatoire propre sur $\nCat{2}$ dont les équivalences faibles sont les biéquivalences et dont la classe des cofibrations est formée des $2$\nbd-foncteurs dont l'image par le foncteur de troncation bête appartient à $lr(K)$. De plus, pour cette structure, tout objet est fibrant et les objets cofibrants sont les $2$\nbd-catégories dont l'image par le foncteur de troncation bête est une catégorie librement engendrée par un graphe.
\end{thm}

\begin{proof}
Voir~\cite[théorèmes 3.3, 4.8 et 6.3, et proposition 4.14]{Lack1}, corrigé par~\cite[théorème 4]{Lack2}.
\end{proof}

\begin{cor}\label{cofThomcofLack}
Toute cofibration de la structure de catégorie de modèles à la Thomason sur $\nCat{2}$ est une cofibration pour la structure de Lack. En particulier, tout objet cofibrant pour la structure de catégorie de modèles à la Thomason sur $\nCat{2}$ est un objet cofibrant pour la structure de Lack.
\end{cor}

\begin{proof}
C'est une conséquence immédiate du théorème précédent et du corollaire~\ref{coflrK}.
\end{proof}

\begin{cor}\label{coflibre}
L'image par le foncteur de troncation bête d'un objet cofibrant pour la structure de catégorie de modèles à la Thomason sur $\nCat{2}$ est une catégorie librement engendrée par un graphe.
\end{cor}

\begin{proof}
C'est une conséquence immédiate du théorème~\ref{Lack} et du corollaire précédent.
\end{proof}

\begin{cor}\label{bieqeqThom}
Toute biéquivalence est une équivalence faible de Thomason.
\end{cor}

\begin{proof}
Le corollaire~\ref{cofThomcofLack} et un argument standard d'adjonction impliquent que les fibrations triviales de la structure de Lack sont des fibrations triviales pour la structure à la Thomason. Le lemme de Ken Brown~\cite[lemme~1.1.12]{Ho} implique alors l'assertion, vu que tous les objets sont fibrants pour la structure de Lack (\emph{cf}.~théorème~\ref{Lack}).
\end{proof}

\begin{rem}
La raison pour laquelle le corollaire précédent n'est pas trivial est qu'une
biéquivalence n'admet pas nécessairement un quasi-inverse strict.
Néanmoins, un $2$-foncteur $u : A \to B$ est une biéquivalence si et
seulement s'il admet pour quasi-inverse (en un sens adéquat) un
pseudo-foncteur $v : B \to A$. On montre facilement qu'on peut même choisir
$v$ normalisé. On retrouve alors le corollaire précédent en observant qu'une
transformation entre foncteurs oplax normalisés induit une homotopie oplax
normalisée et en appliquant le corollaire~\ref{coro:eq_hom_faible}.
\end{rem}

\begin{lemme}\label{adjblabla}
Soit $(F,U)$ un couple de foncteurs adjoints entre deux catégories de modèles:
$$F : C\toto C'\ ,\qquad   U : C' \toto C\ .$$
On suppose que C est à engendrement cofibrant, engendrée par $(I,J)$, et que les ensembles $F(I)$ et $F(J)$ sont formés respectivement de cofibrations et de cofibrations triviales de $C'$. Alors $(F,G)$ est une adjonction de Quillen.
\end{lemme}

\begin{proof}
Le lemme résulte aussitôt de la description des cofibrations (resp. des cofibrations triviales) de $C$ comme rétractes de composés transfinis d'images directes de flèches appartenant à $I$ (resp. à $J$), et des propriétés d'exactitude du foncteur adjoint à gauche $F$.
\end{proof}

\begin{prop}\label{eq-Quillen-incl}
Le couple de foncteurs adjoints formé du foncteur de troncation intelligente et du foncteur d'inclusion
$$\tri:\nCat{2}\toto\cat\ ,\qquad\cat\toto\nCat{2}$$
est une équivalence de Quillen \emph{(}entre $\nCat{2}$ munie de la structure de catégorie de modèles à la Thomason et $\cat$ munie de la structure de Thomason~\cite{Th}\emph{)}.
\end{prop}

\begin{proof}
En vertu du théorème~\ref{thm:cmf_2cat}, la structure de catégorie de modèles sur $\nCat{2}$ est engendrée par $(c_2\Sd^2(I), c_2\Sd^2(J))$, où $I$ et $J$ sont les ensembles définis dans le paragraphe~\ref{enssimpl}. De même, la structure de catégorie de modèles de Thomason sur $\cat$ est engendrée par $(c\Sd^2(I), c\Sd^2(J))$, où $c=c_1$ désigne l'adjoint à gauche du foncteur nerf $N:\cat\to\simpl$ (\cite{Th} ou~\cite[théorème~5.2.12]{Ci}). Or, comme le foncteur~$N$ est égal au composé de l'inclusion $\cat\to\nCat{2}$ et de $N_2$, on a par adjonction un isomorphisme canonique $c\simeq\tri c_2$. Le lemme précédent implique donc que le couple de foncteurs, formé du foncteur de troncation intelligente et du foncteur d'inclusion $\cat\to\nCat{2}$, est une adjonction de Quillen. La proposition résulte alors du théorème~\ref{thm:Chiche}.
\end{proof}

\begin{cor}\label{coford}
L'image par le foncteur de troncation intelligente d'un objet cofibrant pour la structure de catégorie de modèles à la Thomason sur $\nCat{2}$ est un ensemble ordonné.
\end{cor}

\begin{proof}
L'image par le foncteur de troncation intelligente d'un tel objet est, en vertu de la proposition précédente, un objet cofibrant de $\cat$ pour la structure de catégorie de modèles de Thomason, et est donc un ensemble ordonné (\cite[proposition~5.7]{Th}, dont la preuve légèrement incorrecte est corrigée dans~\cite{Ci0}).
\end{proof}

\begin{paragr}\label{varamalg}
Afin de montrer l'assertion (\emph{c}) du théorème~\ref{thm:cof2Thom}, on aura besoin  d'une description de l'image directe d'un crible de $\nCat{2}$, différente de celle de la section~\ref{sec:2rdf}, dans le cas particulier où ce crible est l'image par le foncteur $\Orntgfonct$ (\emph{cf}.~\ref{orntg}) d'un crible d'ensembles ordonnés.
\tb

Dans la suite de ce paragraphe, on utilisera librement, sans mention explicite, la description du foncteur $\Orntgfonct$ donnée dans le paragraphe~\ref{orntg}. On se fixe un crible d'ensembles ordonnés $E\to F$ et on note $i:A\to B$ son image $\Orntgfonct(E)\to\Orntgfonct(F)$ par le foncteur $\Orntgfonct$. Celle-ci est un crible de $\ncat2$ en vertu de la proposition~\ref{orntgadj} et du point~(\emph{c}) de la propositions~\ref{prop:cN_abc}. On se fixe également une $2$\nbd-catégorie $A'$ et un $2$\nbd-foncteur $u:A\to A'$. Le but de ce paragraphe est de décrire l'image directe de $i$ le long de $u$, de façon plus simple que la description générale du paragraphe~\ref{sec:cribles2Cat}. Il s'agit donc de définir une $2$-catégorie $B'$ munie de $2$\nbd-foncteurs $i' : A' \to B'$ et $v : B \to B'$ tels que le carré
\[
\xymatrix{
A \ar[r]^u \ar[d]_i & A' \ar[d]^{i'}\\
B \ar[r]_v & B' \\
}
\]
soit un carré cocartésien de $\ncat2$.
\tb

Dans la suite, pour simplifier, nous supposerons que $i$ est une inclusion (ce qui est licite en vertu de la proposition~\ref{carcriblesncat}).

\begin{sparagr}\label{varamalg1}
Commençons par définir le graphe enrichi en catégories sous-jacent à $B'$.
\tb

L'ensemble des objets de $B'$ est donné par
\[ \Ob(B') = \Ob(A')\amalg\bigl(\Ob(B) \sauf \Ob(A)\bigr)=\Ob(A')\amalg(F \sauf E)\ . \]
Pour $x$ et $y$ deux objets de $B'$, on pose
\[
\Homi_{B'}(x, y) = 
         \begin{cases}
           \Homi_{A'}(x, y), & \text{si $x, y \in \Ob(A')$ ;} \\
           \Homi_B(x, y), & \text{si $x, y \in F \sauf E$ ;}\\
           \varnothing, & \text{si $x \in F \sauf E$ et $y \in \Ob(A')$ ;}\\
           B'_{x, y}, & \text{si $x \in \Ob(A')$ et $y \in
           F \sauf E$,}
         \end{cases}
\]
où $B'_{x, y}$ est la catégorie définie comme suit. Les objets de $B'_{x, y}$ sont les triplets $(S,a,h)$,
$$
\xymatrix{
x\ar[r]^-{h}
&u(a)\ ,
&a=i(a)\ar[r]^-{S}
&y\ ,
}
$$
où $a\in E$ est un objet de $A$, $h$ une $1$\nbd-flèche de $A'$ et $S\in\xi(F)$ une $1$\nbd-flèche de $B$ telle que $S\cap E=\{a\}$. Si $(S_0,a_0,h_0)$ et $(S_1,a_1,h_1)$ sont deux objets de $B'_{x, y}$, un morphisme du premier vers le second est un diagramme
$$
\xymatrixrowsep{.4pc}
\xymatrixcolsep{2.5pc}
\xymatrix{
&u(a_0)\ar[dddd]^{u(\{a_0,a_1\})}
&\kern 10pt
&a_0\ar[dddd]_{\{a_0,a_1\}}\ar[ddr]^{S_0}
\ddrlowertwocell<\omit>{<3.3>}
\\
&
\\
x\ar[uur]^{h_0}\ar[ddr]_{h_1}
\ddruppertwocell<\omit>{<-3.3>\kern 1pt \alpha}
&&&&y
\\
&
\\
&u(a_1)
&&a_1\ar[uur]_{S_1}
&\kern 20pt,
}
$$
où le triangle de gauche est dans $A'$ et le triangle de droite dans $B$. Ce dernier indique simplement que $a_0\leq a_1$ et que $S_0\subset S_1\cup\{a_0\}$ (vu que $S_0\cap E=\{a_0\}$ et $a_1\in E$, cette dernière condition équivaut à l'inclusion $S_0\sauf\{a_0\}\subset S_1\sauf\{a_1\}$). Ainsi, pour que l'ensemble des flèches de $(S_0,a_0,h_0)$ vers $(S_1,a_1,h_1)$ soit non vide, il faut que ces conditions soient satisfaites. On remarque que $(S_0,a_0,h_0)$ et $(S_1,a_1,h_1)$ étant \emph{fixés}, la flèche de $B'_{x, y}$ définie par le diagramme ci-dessus est entièrement déterminée par la $2$-flèche~$\alpha$; par abus, on la notera simplement $\undl{\alpha}$ quand aucune ambiguïté n'en résulte. On se gardera de confondre cette flèche de $\Homi_{B'}(x, y) =B'_{x, y}$ avec la flèche $\alpha$ de $\Homi_{B'}(x, u(a_1))=\Homi_{A'}(x, u(a_1))$.
\tb
Le composé de deux flèches composables de $B'_{x, y}$ 
$$
\xymatrix{
(S_0,a_0,h_0)\ar[r]^{\undl{\alpha_1}}
&(S_1,a_1,h_1)\ar[r]^{\undl{\alpha_2}}
&(S_2,a_2,h_2)
}\ ,
$$
où
$$\alpha_1:u(\{a_0,a_1\})h_0\toto h_1\qquad\hbox{et}\qquad\alpha_2:u(\{a_1,a_2\})h_1\toto h_2$$
sont des $2$-flèches de $A'$, est défini en posant
$$\undl{\alpha_2}\,\undl{\alpha_1}=\undl{\alpha\kern 1pt}\ ,\quad\hbox{où}\quad\alpha=\alpha_2\bigl(u(\{a_1,a_2\})\comp\alpha_1\bigr)\bigl(u(\{a_0,a_2\}\subset\{a_0,a_1,a_2\})\comp h_0\bigr)\ ,$$
et en observant que les inégalités $a_0\leq a_1$ et $a_1\leq a_2$ impliquent que $a_0\leq a_2$ et que les inclusions
$S_0\sauf\{a_0\}\subset S_1\sauf\{a_1\}$ et $S_1\sauf\{a_1\}\subset S_2\sauf\{a_2\}$ impliquent que $S_0\sauf\{a_0\}\subset S_2\sauf\{a_2\}$.
L'unité d'un objet $(S,a,h)$ de $B'_{x, y}$ est définie en posant $1_{(S,a,h)}=\undl{1_h}$. Les propriétés d'associativité et d'unité sont immédiates.
\end{sparagr}

\begin{sparagr}\label{varamalg2}
Poursuivons la construction de la $2$\nbd-catégorie $B'$ en définissant, pour tous objets $x$, $y$ et $z$ de $B'$, un foncteur de
composition horizontale 
\[ \Homi_{B'}(y, z) \times \Homi_{B'}(x, y) \to \Homi_{B'}(x, z)\ . \]
Soient donc $x$, $y$ et $z$ trois objets de $B'$.
\tb
\begin{enumerate}[label=(\arabic*), wide]
  \item Si $x$, $y$ et $z$ sont dans $A'$, la composition horizontale de
    $B'$ est héritée de celle de~$A'$.
    \tb
  \item Si $x$, $y$ et $z$ sont dans $F\sauf E$, la composition horizontale
    de $B'$ est héritée de celle de~$B$.
    \tb
  \item Si $x$ et $y$ sont dans $A'$ et $z$ est dans $F\sauf E$, la composition horizontale dans $B'$ est définie par
$$
\xymatrixcolsep{5.9pc}
\xymatrix{
z\rtwocell<\omit>{<-0>\kern 4pt \undl{\beta}}
&y\ar@/_2.pc/[l]_{(S_0,a_0,h_0)}\ar@/^2.pc/[l]^{(S_1,a_1,h_1)}\rtwocell<\omit>{<-0>\kern 3pt \alpha}
&x\ar@/_2.pc/[l]_{f_0}\ar@/^2.pc/[l]^{f_1}
}
\qquad\longmapsto\qquad
\xymatrix{
z\rtwocell<\omit>{<-0>\kern 10pt \undl{\beta\comp\alpha}}
&x\ar@/_2.pc/[l]_{(S_0,a_0,h_0f_0)}\ar@/^2.pc/[l]^{(S_1,a_1,h_1f_1)}
}\ ,
$$
où $a_0,a_1$ sont des éléments de $E$ tels que $a_0\leq a_1$, $f_0,f_1:x\to y$, $h_0:y\to u(a_0)$ et $h_1:y\to u(a_1)$ sont des $1$\nbd-flèches de $A'$, \hbox{$\alpha:f_0\to f_1$} et $\beta:u(\{a_0,a_1\})h_0\to h_1$ des $2$\nbd-flèches de $A'$, $S_0,S_1$ des éléments de $\xi(F)$ tels que $S_0\cap E=\{a_0\}$, $S_1\cap E=\{a_1\}$, $\max(S_0)=\max(S_1)=z$ et $S_0\sauf\{a_0\}\subset S_1\sauf\{a_1\}$, et $\beta\comp\alpha:u(\{a_0,a_1\})h_0f_0\to h_1f_1$ est le composé horizontal de $\beta$ et $\alpha$ dans $A'$.
\tb
\item Si $x$ est dans $A'$ et $y$ et $z$ sont dans $F\sauf E$, la composition horizontale dans $B'$ est définie par
$$
\xymatrixcolsep{5.9pc}
\xymatrix{
z\rtwocell<\omit>{<-0>\kern 4pt }
&y\ar@/_2.pc/[l]_{T_0}\ar@/^2.pc/[l]^{T_1}\rtwocell<\omit>{<-0>\kern 4pt \undl{\alpha\kern 2pt}}
&x\ar@/_2.pc/[l]_{(S_0,a_0,h_0)}\ar@/^2.pc/[l]^{(S_1,a_1,h_1)}
}
\qquad\longmapsto\qquad
\xymatrix{
z\rtwocell<\omit>{<-0>\kern 4pt \undl{\alpha\kern 2pt}}
&x\ar@/_2.pc/[l]_{(T_0\cup S_0,a_0,h_0)}\ar@/^2.pc/[l]^{(T_1\cup S_1,a_1,h_1)}
}\ ,
$$
où $a_0,a_1$ sont des éléments de $E$ tels que $a_0\leq a_1$, $h_0:x\to u(a_0)$ et $h_1:x\to u(a_1)$ des $1$\nbd-flèches de $A'$, $\alpha:u(\{a_0,a_1\})h_0\to h_1$ est une $2$\nbd-flèche de $A'$, et $S_0,S_1,T_0,T_1$ sont des éléments de $\xi(F)$ tels que $S_0\cap E=\{a_0\}$, $S_1\cap E=\{a_1\}$, $\max(S_0)=\max(S_1)=y$, $S_0\sauf\{a_0\}\subset S_1\sauf\{a_1\}$, $\min(T_0)=\min(T_1)=y$, $\max(T_0)=\max(T_1)=z$ et $T_0\subset T_1$.
\end{enumerate}
\tb
Montrons que le morphisme de graphes
\[ \Homi_{B'}(y, z) \times \Homi_{B'}(x, y) \to \Homi_{B'}(x, z)\  \]
qu'on vient de définir est bien un foncteur. La compatibilité aux identités étant évidente, il reste à prouver la compatibilité à la composition, autrement dit, la règle d'échange dans $B'$. Dans les cas (1) et (2) ci-dessus, elle résulte aussitôt de la règle d'échange dans $A'$ et $B$ respectivement. Examinons le cas (3):
$$
\xymatrixcolsep{7pc}
\xymatrix{
z\rtwocell<\omit>{<-4>\kern 6pt \undl{\beta_1}}\rtwocell<\omit>{<3>\kern 6pt \undl{\beta_2}}
&y\ar@/_2.8pc/[l]_{(S_0,a_0,h_0)}\ar@/^2.8pc/[l]^{(S_2,a_2,h_2)}\ar[l]_{(S_1,a_1,h_1)}
\rtwocell<\omit>{<-4>\kern 5pt \alpha_1}\rtwocell<\omit>{<3>\kern 5pt \alpha_2}
&x\ar@/_2.8pc/[l]_{f_0}\ar@/^2.8pc/[l]^{f_2}\ar[l]_{f_1}
}\kern 10pt.
$$
Le composé vertical $\undl{\beta_2}\,\undl{\beta_1}$ dans $B'$ est égal à $\undl{\beta}$, où
$$\beta=\beta_2\bigl(u(\{a_1,a_2\})\comp\beta_1\bigr)\bigl(u(\{a_0,a_2\}\subset\{a_0,a_1,a_2\})\comp h_0\bigr)\ ,$$
tandis que le composé vertical de $\alpha_2$ et $\alpha_1$ dans $B'$ est le même que dans $A'$. On en déduit que dans $B'$ on a
$(\undl{\beta_2}\,\undl{\beta_1})\comp(\alpha_2\alpha_1)= \undl{\gamma}$, où
$$\gamma=\bigl(\beta_2\bigl(u(\{a_1,a_2\})\comp\beta_1\bigr)\bigl(u(\{a_0,a_2\}\subset\{a_0,a_1,a_2\})\comp h_0\bigr)\bigr)\comp\alpha_2\alpha_1\ .$$
De même, dans $B'$ on a $(\undl{\beta_2}\comp\alpha_2)(\undl{\beta_1}\comp\alpha_1)=(\undl{\beta_2\comp\alpha_2})(\undl{\beta_1\comp\alpha_1})=\undl{\gamma'}$, où
$$\gamma'=\bigl(\beta_2\comp\alpha_2\bigr)\bigl(u(\{a_1,a_2\})\comp\beta_1\comp\alpha_1\bigr)\bigl(u(\{a_0,a_2\}\subset\{a_0,a_1,a_2\})\comp h_0f_0\bigr)\ .$$
On conclut en remarquant qu'en vertu de la règle d'échange dans $A'$, on a $\gamma=\gamma'$.
On laisse au lecteur le soin de constater que dans le cas (4) la vérification est évidente.
\tb
Enfin, si $x$ est un objet de $B'$, l'identité $1_x\in\sHom_{B'}(x,x)$ de $x$ est héritée de celle de~$A'$ ou de celle de $B$, selon que $x\in\ob(A')$ ou $x\in F\sauf E$. Les propriétés d'associativité et d'unité pour la composition horizontale étant immédiates, cela termine la définition de la catégorie $B'$.
\end{sparagr}

\begin{sparagr}\label{varamalg3}
Par définition, $A'$ est une sous-$2$-catégorie de $B'$. On notera $i' : A'\to B'$ le $2$\nbd-foncteur d'inclusion. On définit un $2$-foncteur $v : B \to B'$ de la manière suivante. Si $x$ est un objet de $B$, on pose
\[
v(x) =
\begin{cases}
  u(x), & \text{si $x \in E$,}\\
  x, & \text{si $x \in F\sauf E$.}
\end{cases}
\]
Si $S:x\to y$ est une $1$\nbd-flèche de $B$, $S\in\xi(F)$, $\min(S)=x$, $\max(S)=y$, on pose
\[
v(S)=
\begin{cases}
u(S), & \text{si $x,y\in E$,}\\
S, & \text{si $x,y\in F\sauf E$,}\\
(S'',a,u(S')), & \text{si $x\in E$, $y\in F\sauf E$,}
\end{cases}
\]
où $S'=S\cap E$, $a=\max(S')$ et $S''=(S\cap(F\sauf E))\cup\{a\}$, de sorte qu'on a bien $S''\cap E=\{a\}$, et des $1$\nbd-flèches
$$\xymatrix{
u(x)\ar[r]^{u(S')}
&u(a)
}\ ,\qquad
\xymatrix{
a\ar[r]^{S''}
&y
}$$
de $A'$ et $B$ respectivement. Il reste à définir l'image par $v$ d'une $2$-flèche de $F$, autrement dit, d'une inclusion $S_0\subset S_1$ où $S_0$ et $S_1$ sont des éléments de $\xi(F)$ tels que $\min(S_0)=\min(S_1)=x$ et $\max(S_0)=\max(S_1)=y$. On pose
\[
v(S_0\subset S_1)=
\begin{cases}
u(S_0\subset S_1), & \text{si $x,y\in E$,}\\
S_0\subset S_1, & \text{si $x,y\in F\sauf E$,}\\
\undl{u(S'_0\cup\{a_1\}\subset S'_1)}, & \text{si $x\in E$, $y\in F\sauf E$,}
\end{cases}
\]
où suivant la convention du paragraphe~\ref{varamalg1}, $\undl{u(S'_0\cup\{a_1\}\subset S'_1)}$ désigne la $2$\nbd-flèche de $B'$ de source et but 
$$v(S_0)=(S''_0,a_0,u(S'_0))\qquad\hbox{et}\qquad v(S_1)=(S''_1,a_1,u(S'_1))$$
respectivement, où
$$
\begin{aligned}
S'_0=S_0\cap E\,,\quad a_0=\max(S'_0)\,,\quad S''_0=(S_0\cap(F\sauf E))\cup\{a_0\}\ ,\\
S'_1=S_1\cap E\,,\quad a_1=\max(S'_1)\,,\quad S''_1=(S_1\cap(F\sauf E))\cup\{a_1\}\ ,
\end{aligned}
$$
définie par le diagramme
$$
\xymatrixrowsep{.4pc}
\xymatrixcolsep{2.5pc}
\xymatrix{
&u(a_0)\ar[dddd]^{u(\{a_0,a_1\})}
&\kern 10pt
&a_0\ar[dddd]_{\{a_0,a_1\}}\ar[ddr]^{S''_0}
\ddrlowertwocell<\omit>{<3.3>}
\\
&
\\
u(x)\ar[uur]^{u(S'_0)}\ar[ddr]_{u(S'_1)}
\ddruppertwocell<\omit>{<-3.3>\kern 1pt \alpha}
&&&&y
\\
&
\\
&u(a_1)
&&a_1\ar[uur]_{S''_1}
&\kern 20pt,
}
$$
où $\alpha=u(S'_0\cup\{a_1\}\subset S'_1)$, la $2$\nbd-flèche du triangle de droite correspondant à l'inclusion $S''_0\subset S''_1\cup\{a_0\}$.
\tb

Pour montrer que $v$ est un $2$\nbd-foncteur, commençons par vérifier sa compatibilité à la composition des $1$\nbd-flèches. Soient donc $x,y,z\in F$ et $S:x\to y$, $T:y\to z$ deux $1$\nbd-flèches composables de $A$, où $S,T\in\xi(F)$, $x=\min(S)$, $y=\max(S)=\min(T)$ et $z=\max(T)$. Il y a deux cas non triviaux à examiner:
\tb
\begin{enumerate}[label=(\arabic*), wide]
\item Si $x,y\in E$ et $z\in F\sauf E$, on a $v(S)=u(S)$ et $v(T)=(T'',a,u(T'))$, où
$$T'=T\cap E\,,\quad a=\max(T')\,,\quad T''=(T\cap(F\sauf E))\cup\{a\}\,.$$
On remarque qu'on a
$$(T\cup S)\cap E=T'\cup S\,,\quad\max(T'\cup S)=a\,,\quad((T\cup S)\cap(F\sauf E))\cup\{a\}= T''\,,$$
ce qui implique que $v(TS)=v(T\cup S)=(T'',a,u(T'\cup S))$. Or,
$$v(T)\,v(S)=(T'',a,u(T'))\,u(S)=(T'',a,u(T')\,u(S))=(T'',a,u(T'\cup S))\,,$$
ce qui prouve que $v(TS)=v(T)\,v(S)$.
\tb
\item Si $x\in E$ et $y,z\in F\sauf E$, on a $v(S)=(S'',a,u(S'))$, où
$$S'=S\cap E\,,\quad a=\max(S')\,,\quad S''=(S\cap(F\sauf E))\cup\{a\}\,,$$
et $v(T)=T$. On remarque qu'on a
$$(T\cup S)\cap E=S'\,,\quad((T\cup S)\cap(F\sauf E))\cup\{a\}= T\cup S''\,,$$
ce qui implique que $v(TS)=v(T\cup S)=(T\cup S'',a,u(S'))$. Or,
$$v(T)\,v(S)=T\,(S'',a,u(S'))=(T\cup S'',a,u(S'))\,,$$
ce qui prouve que $v(TS)=v(T)\,v(S)$.
\end{enumerate}
\tb

Vérifions la compatibilité de $v$ à la composition horizontale des $2$\nbd-flèches. Soit donc un diagramme
$$
\xymatrixcolsep{5.9pc}
\xymatrix{
z\rtwocell<\omit>{<-0>\kern 4pt }
&y\ar@/_2.pc/[l]_{T_0}\ar@/^2.pc/[l]^{T_1}\rtwocell<\omit>{<-0>}
&x\ar@/_2.pc/[l]_{S_0}\ar@/^2.pc/[l]^{S_1}
}$$
dans $B$. Il s'agit de montrer que $v(T_0\subset T_1)\comp v(S_0\subset S_1)=v(T_0\cup S_0\subset S_1\cup T_1)$.
Il n'y a, à nouveau, que deux cas non triviaux à examiner:
\tb
\begin{enumerate}[label=(\arabic*), wide]
\item Si $x,y\in E$ et $z\in F\sauf E$, on a $v(S_0\subset S_1)=u(S_0\subset S_1)$ et $v(T_0\subset T_1)=\undl{\alpha}$, où 
$$\alpha=u(T'_0\cup\{a_1\}\subset T'_1)\,,\quad T'_0=T_0\cap E\,,\quad T'_1=T_1\cap E\,,\quad a_1=\max(T'_1)\,,$$
et les égalités
$$(T_0\cup S_0)\cap E=T'_0\cup S_0\,,\quad (T_1\cup S_1)\cap E=T'_1\cup S_1\quad\hbox{et}\quad \max(T'_1\cup S_1)=a_1\,$$
impliquent que $u(T_0\cup S_0\subset T_1\cup S_1)=\undl{\beta}$, où $\beta=u(T'_0\cup S_0\cup\{a_1\}\subset T'_1\cup S_1)$.
Or, $v(T_0\subset T_1)\comp v(S_0\subset S_1)=\undl{\beta'}$, où 
$$\beta'=u(T'_0\cup\{a_1\}\subset T'_1)\comp u(S_0\subset S_1)=u(T'_0\cup\{a_1\}\cup S_0\subset T'_1\cup S_1)\,,$$ 
ce qui prouve l'assertion.
\tb
\item Si $x\in E$ et $y,z\in F\sauf E$, et si l'on pose
$$S'_0=S_0\cap E\,,\quad S'_1=S_1\cap E\,,\quad a_1=\max(S'_1)\,,$$
on a 
$$(T_0\cup S_0)\cap E=S'_0\quad\hbox{et}\quad(T_1\cup S_1)\cap E=S'_1\,,$$
ce qui implique facilement qu'on a à la fois 
$$v(T_0\subset T_1)\comp v(S_0\subset S_1)=\undl{\alpha}\quad\hbox{et}\quad v(T_0\cup S_0\subset S_1\cup T_1)=\undl{\alpha}\,,$$ 
où $\alpha=u(S'_0\cup\{a_1\}\subset S'_1)$. 
\end{enumerate}
\tb 
Montrons la compatibilité de $v$ à la composition verticale des $2$\nbd-flèches. Soit donc un diagramme
$$
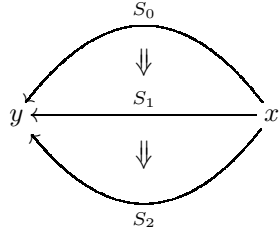
\xymatrixcolsep{7pc}
\xymatrix{
y\rtwocell<\omit>{<-4>}\rtwocell<\omit>{<3>}
&x\ar@/_2.8pc/[l]_{S_0}\ar@/^2.8pc/[l]^{S_2}\ar[l]_{S_1}
}
$$
dans $B$. Il s'agit de montrer que $v(S_1\subset S_2)\,v(S_0\subset S_1)=v(S_0\subset S_2)$. Il y a un seul cas non trivial, celui où $x\in E$ et $y\in F\sauf E$. Dans ce cas, si l'on pose
$$S'_p=S_p\cap E\,,\quad a_p=\max(S'_p)\,,\quad p=0,1,2\,,$$
on a
$$v(S_p\subset S_q)=\undl{\alpha_{pq}}\,,\quad\hbox{où}\quad\alpha_{pq}=u(S'_p\cup\{a_q\}\subset S'_q)\,,\quad0\leq p<q\leq2\,.$$
On en déduit que $v(S_1\subset S_2)\,v(S_0\subset S_1)=\undl{\alpha}$, où
$$\begin{aligned}
\alpha=u(S'_1\cup\{a_2\}\subset S'_2)\,&\bigl(u(\{a_1,a_2\})\comp u(S'_0\cup\{a_1\}\subset S'_1)\bigr)\\
&\bigl(u(\{a_0,a_2\}\subset\{a_0,a_1,a_2\})\comp u(S'_0)\bigr)=u(S'_0\cup\{a_2\}\subset S'_2)\,,
\end{aligned}$$
ce qui prouve l'assertion. La compatibilité aux identités étant évidente, ceci achève la vérification que $v$ est un $2$\nbd-foncteur.
\tb
On a donc un carré de $2$\nbd-foncteurs
\[
\xymatrix{
A \ar[r]^u \ar[d]_i & A' \ar[d]^{i'}\\
B \ar[r]_v & B' \pbox{.} \\
}
\]
On vérifie immédiatement que ce carré est commutatif.
\end{sparagr}

\begin{sprop}
Le carré
\[
\xymatrix{
A \ar[r]^u \ar[d]_i & A' \ar[d]^{i'}\\
B \ar[r]_v & B'
}
\]
est un carré cocartésien de $\ncat2$.
\end{sprop}

\begin{proof}
Soient $C$ une $2$-catégorie, et $j : A' \to C$
et $w : B \to C$ deux $2$\nbd-foncteurs tels que $ju = wi$. Il
s'agit de montrer qu'il existe un unique $2$\nbd-foncteur $l : B' \to C$ rendant commutatif le diagramme
\[
\raise 29pt
\vbox{
\xymatrix@R=.4pc@C=.6pc{
A \ar[rrr]^u \ar[ddd]_i & & & A' \ar[ddd]_{i'} \ar@/^1em/[dddddrr]^{j}\\
\\
\\
B \ar[rrr]^v \ar@/_1em/[ddrrrrr]_{w} & & & B' \ar@{-->}[ddrr]^{l} \\
\\
& & & & & C \pbox{.}
}
}
\leqno(\star)
\]
\textsc{Existence.} Définissons un tel $2$\nbd-foncteur $l$. Pour $x$ un objet de $B'$, on pose
\[
l(x) = 
\begin{cases}
j(x), & \text{si $x \in \Ob(A')$,}\\ 
w(x), & \text{si $x \in F\sauf E$.}\\ 
\end{cases}
\]
Pour $f : x \to y$ une $1$-flèche de $B'$, on pose
\[
l(f) =
\begin{cases}
j(f), & \text{si $x, y \in \Ob(A')$,}\\
w(f), & \text{si $x, y \in F\sauf E$,}\\
w(S)j(h), & \text{si $x \in \Ob(A')$, $y \in F\sauf E$} \\
\phantom{w(S)j(h),} & \text{et $f=(S,a,h)$.}
\end{cases}
\]
Pour $\beta$ une $2$\nbd-flèche de $B'$, 
$$
\xymatrixcolsep{3.9pc}
\xymatrix{
y\rtwocell<\omit>{<-0>\kern 3pt \beta}
&x\ar@/_1.pc/[l]_{f_0}\ar@/^1.pc/[l]^{f_1}
}\kern 10pt,
$$
on pose
\[
l(\beta)=
\begin{cases}
j(\beta), & \text{si $x, y \in \Ob(A')$,}\\
w(\beta), & \text{si $x, y \in F\sauf E$,}\\
\end{cases}
\]
et 
$$l(\beta)=\bigl(w(S_1)\comp j(\alpha)\bigr)\bigl(w(S_0\subset S_1\cup\{a_0\})\comp j(h_0)\big)\,,$$ 
$$
\xymatrixcolsep{-.2pc}
\xymatrixrowsep{1.8pc}
\xymatrix{
&{\phantom{aaaa\kern 40pt aaaaaa}}
&w(a_0)\dlltwocell<8>_{w(S_0)}^{\kern 30pt w(S_1\cup\{a_0\})}{\omit}
\rrtwocell<\omit>{<8>\kern 3pt=}
&=
&ju(a_0)
\\
w(y)\urtwocell<\omit>{<1.5>\kern 50ptw\vrule depth 4pt width0pt(S_0\subset S_1\cup\{a_0\})}
&&&&&{\phantom{aaaa\kern 40pt aaaaaa}}
&j(x)\ar@/^-1.5em/[llu]_{j(h_0)}\ar@/^2em/[lld]^{j(h_1)}\ar@/^-2em/[lld]_(.5){w(\{a_0,a_1\})j(h_0)}
&,
\\
&&w(a_1)\ar@/^1.5em/[llu]^{w(S_1)}
&=
&ju(a_1)\urrtwocell<\omit>{<-.0>j(\alpha)\vrule height 10pt width 0pt{\phantom{aa\kern 15pt aa}}}
}
$$
si $x \in \Ob(A')$, $y \in F\sauf E$, $f_0=(S_0,a_0,h_0)$, $f_1=(S_1,a_1,h_1)$ et
$$\beta=\undl{\alpha}=\kern 25pt
\raise 30pt
\vbox{
\xymatrixrowsep{.4pc}
\xymatrixcolsep{2.5pc}
\xymatrix{
&u(a_0)\ar[dddd]^{u(\{a_0,a_1\})}
&\kern 10pt
&a_0\ar[dddd]_{\{a_0,a_1\}}\ar[ddr]^{S_0}
\ddrlowertwocell<\omit>{<3.3>}
\\
&
\\
x\ar[uur]^{h_0}\ar[ddr]_{h_1}
\ddruppertwocell<\omit>{<-3.3>\kern 1pt \alpha}
&&&&y
\\
&
\\
&u(a_1)
&&a_1\ar[uur]_{S_1}
&\kern 20pt.
}
}
$$

Pour montrer que $l$ est un $2$\nbd-foncteur, commençons par vérifier sa compatibilité à la composition des $1$\nbd-flèches. Soient donc $x,y,z\in\ob{B'}$ et $f:x\to y$, $g:y\to z$ deux $1$\nbd-flèches composables de $B'$. Il y a deux cas non triviaux à examiner:
\tb
\begin{enumerate}[label=(\arabic*), wide]
\item Si $x,y\in E$, $z\in F\sauf E$, et $g=(S,a,h)$, on a
$$l(gf)=l(S,a,hf)=w(S)j(hf)=w(S)j(h)j(f)=l(g)l(f)\,.$$
\item Si $x\in E$, $y,z\in F\sauf E$, et $f=(S,a,h)$ et $g=T\in\xi(F)$, on a
$$l(gf)=l(T\cup S,a,h)=w(T\cup S)j(h)=w(T)w(S)j(h)=l(g)l(f)\,.$$
\end{enumerate}
\tb
Vérifions la compatibilité de $l$ à la composition horizontale des $2$\nbd-flèches. 
Il n'y a, à nouveau, que deux cas non triviaux à examiner:
\tb
\begin{enumerate}[label=(\arabic*), wide]
\item Si on a $x,y\in E$ et $z\in F\sauf E$, et un diagramme dans $B'$ de la forme
$$
\xymatrixcolsep{5.9pc}
\xymatrix{
z\rtwocell<\omit>{<-0>\kern 4pt \undl{\beta}}
&y\ar@/_2.pc/[l]_{(S_0,a_0,h_0)}\ar@/^2.pc/[l]^{(S_1,a_1,h_1)}\rtwocell<\omit>{<-0>\kern 3pt \alpha}
&x\ar@/_2.pc/[l]_{f_0}\ar@/^2.pc/[l]^{f_1}
}\kern 10pt,$$
alors
$$
\begin{aligned}
l(\undl{\beta}\comp\alpha)&=\bigl(w(S_1)\comp j(\beta\comp\alpha)\bigr)\bigl(w(S_0\subset S_1\cup\{a_0\})\comp j(h_0f_0)\bigr)\\
&=\bigl(w(S_1)\comp j(\beta)\comp j(\alpha)\bigr)\bigl(w(S_0\subset S_1\cup\{a_0\})\comp j(h_0)j(f_0)\bigr)\\
&=\bigl[\bigl(w(S_1)\comp j(\beta)\bigr)\bigl(w(S_0\subset S_1\cup\{a_0\})\comp j(h_0)\bigr)\bigr]\comp j(\alpha)\\
&=l(\undl{\beta})\comp l(\alpha)\ .
\end{aligned}
$$
\item Si on a $x\in E$ et $y,z\in F\sauf E$, et un diagramme dans $B'$ de la forme
$$
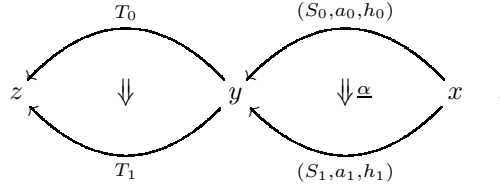
\xymatrixcolsep{5.9pc}
\xymatrix{
z\rtwocell<\omit>{<-0>\kern 4pt }
&y\ar@/_2.pc/[l]_{T_0}\ar@/^2.pc/[l]^{T_1}\rtwocell<\omit>{<-0>\kern 4pt \undl{\alpha\kern 2pt}}
&x\ar@/_2.pc/[l]_{(S_0,a_0,h_0)}\ar@/^2.pc/[l]^{(S_1,a_1,h_1)}
}\kern 10pt,
$$
alors
$$
\begin{aligned}
l\bigl((T_0\subset T_1)\comp\undl{\alpha}\bigr)&=\bigl(w(T_1\cup S_1)\comp j(\alpha)\bigr)\bigl(w(T_0\cup S_0\subset T_1\cup S_1\cup\{a_0\})\comp j(h_0)\bigr)\\
&=w(T_0\subset T_1)\,\comp\bigr[\bigl(w(S_1)\comp j(\alpha)\bigr)\bigl(w(S_0\subset S_1\cup\{a_0\})\comp j(h_0)\bigr)\bigr]\\
&=l(T_0\subset T_1)\comp l(\undl{\alpha})\ .
\end{aligned}
$$
\end{enumerate}
\tb
Pour montrer la compatibilité de $l$ à la composition verticale des $2$\nbd-flèches, le seul cas non trivial à considérer est celui d'un diagramme dans $B'$ de la forme
$$
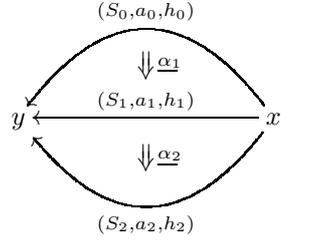
\xymatrixcolsep{7pc}
\xymatrix{
y\rtwocell<\omit>{<-4>\kern 6pt \undl{\alpha_1}}\rtwocell<\omit>{<3>\kern 6pt \undl{\alpha_2}}
&x\ar@/_2.8pc/[l]_{(S_0,a_0,h_0)}\ar@/^2.8pc/[l]^{(S_2,a_2,h_2)}\ar[l]_{(S_1,a_1,h_1)}
}\kern 10pt,
$$
avec $x\in E$ et $y\in F\sauf E$. Alors on a
{\allowdisplaybreaks
\begin{align*}
l(&\undl{\alpha_2}\,\undl{\alpha_1})=\bigl(w(S_2)\comp j\bigl[\alpha_2(u(\{a_1,a_2\})\comp\alpha_1)(u(\{a_0,a_2\}\subset\{a_0,a_1,a_2\})\comp h_0)\bigr]\bigr)\\
&\kern 50pt\bigl(w(S_0\subset S_2\cup\{a_0\})\comp j(h_0)\bigr)\\
\noalign{\vskip 5pt}
&=\bigl(w(S_2)\comp\bigl[j(\alpha_2)(w(\{a_1,a_2\})\comp j(\alpha_1))(w(\{a_0,a_2\}\subset\{a_0,a_1,a_2\})\comp j(h_0))\bigr]\bigr)\\
&\kern 50pt\bigl(w(S_0\subset S_2\cup\{a_0\})\comp j(h_0)\bigr)\\
\noalign{\vskip 5pt}
&=\bigl(w(S_2)\comp j(\alpha_2)\bigr)\bigl(w(S_2\cup\{a_1\})\comp j(\alpha_1)\bigr)\bigl(w(S_2\cup\{a_0\}\subset S_2\cup\{a_0,a_1\})\comp j(h_0)\bigr)\\
&\kern 50pt\bigl(w(S_0\subset S_2\cup\{a_0\})\comp j(h_0)\bigr)\\
\noalign{\vskip 5pt}
&=\bigl(w(S_2)\comp j(\alpha_2)\bigr)\bigl(w(S_2\cup\{a_1\})\comp j(\alpha_1)\bigr)\bigl(w(S_1\cup\{a_0\}\subset S_2\cup\{a_0,a_1\})\comp j(h_0)\bigr)\\
&\kern 50pt\bigl(w(S_0\subset S_1\cup\{a_0\})\comp j(h_0)\bigr)\\
\noalign{\vskip 5pt}
&=\bigl(w(S_2)\comp j(\alpha_2)\bigr)\bigl(w(S_2\cup\{a_1\})\comp j(\alpha_1)\bigr)\bigl(w(S_1\subset S_2\cup\{a_1\})\comp w(\{a_0,a_1\})\,j(h_0)\bigr)\\
&\kern 50pt\bigl(w(S_0\subset S_1\cup\{a_0\})\comp j(h_0)\bigr)\\
\noalign{\vskip 5pt}
&=\bigl(w(S_2)\comp j(\alpha_2)\bigr)\bigl(w(S_1\subset S_2\cup\{a_1\})\comp j(h_1)\bigr)\bigl(w(S_1)\comp j(\alpha_1)\bigr)\\
&\kern 50pt\bigl(w(S_0\subset S_1\cup\{a_0\})\comp j(h_0)\bigr)\\
\noalign{\vskip 5pt}
&=l(\undl{\alpha_2})\,l(\undl{\alpha_1})\ .
\end{align*}
}%
La compatibilité de $l$ aux identités étant évidente, cela achève de prouver l'existence d'un $2$\nbd-foncteur $l$ rendant le diagramme~($\star$) commutatif.
\tb

\noindent\textsc{Unicité.} Montrons l'unicité d'un $2$\nbd-foncteur $l$ rendant commutatif le diagramme~($\star$). L'unicité sur les objets est évidente. Soit $f:x\to y$ une $1$\nbd-flèche de $B'$. Si les objets~$x$ et~$y$ sont tous deux dans $\Ob(A')$ ou dans $F\sauf E$, l'unicité résulte aussitôt de la commutativité du diagramme~($\star$). Supposons donc que $x\in\ob(A')$, $y\in F\sauf E$ et que $f=(S,a,h)$. On remarque qu'on a les égalités
$$(S,a,h)=(S,a,1_{u(a)})\,h=(S,a,u(\{a\}))\,h=v(S)\,i'(h)\ ,$$
et la commutativité du diagramme~($\star$) implique que $l(f)=w(S)j(h)$, ce qui prouve l'assertion. De même, pour vérifier l'unicité sur les $2$\nbd-flèches, le seul cas non trivial à examiner est celui d'une $2$\nbd-flèche de la forme
$$
\xymatrixcolsep{3.9pc}
\xymatrix{
y\rtwocell<\omit>{<-0>\kern 4pt \undl{\alpha\kern 2pt}}
&x\ar@/_1.3pc/[l]_{(S_0,a_0,h_0)}\ar@/^1.3pc/[l]^{(S_1,a_1,h_1)}
}\ ,\kern 15pt
\undl{\alpha}=\kern 10pt
\raise 30pt
\vbox{
\xymatrixrowsep{.3pc}
\xymatrixcolsep{2.2pc}
\xymatrix{
&u(a_0)\ar[dddd]^{u(\{a_0,a_1\})}
&\kern 10pt
&a_0\ar[dddd]_{\{a_0,a_1\}}\ar[ddr]^{S_0}
\ddrlowertwocell<\omit>{<3.3>}
\\
&
\\
x\ar[uur]^{h_0}\ar[ddr]_{h_1}
\ddruppertwocell<\omit>{<-3.3>\kern 1pt \alpha}
&&&&y
\\
&
\\
&u(a_1)
&&a_1\ar[uur]_{S_1}
&\kern 20pt,
}
}
$$
avec $x\in E$ et $y\in F\sauf E$. Or, on a alors l'égalité
$$\undl{\alpha}=\bigl(v(S_1)\comp i'(\alpha)\bigr)\bigl(v(S_0\subset S_1\cup\{a_0\})\comp i'(h_0)\bigr)\ ,$$
qui se démontre en déroulant les définitions et en considérant le diagramme:
$$
\xymatrixrowsep{1.25pc}
\xymatrixcolsep{1.pc}
\xymatrix{
&&&&&&u(a_0)\ar[dd]^{u(\{a_0,a_1\})}
&&&&a_0\ar[dd]_{\{a_0,a_1\}}\ar[ddrrrr]^{S_0}
\\
&&&&\drruppertwocell<\omit>{<7.3>\kern 1pt \alpha}\drruppertwocell<\omit>{<-1.9>\kern 3pt =}
&&&&&&\drrtwocell<\omit>{<-1.5>\kern 3pt}
&\drrtwocell<\omit>{<7.5>\kern 3pt =}
\\
x\ar[rrr]^{h_0}\ar[rrrrrrdd]_{h_1}
&&&u(a_0)\ar[rrr]^{u(\{a_0,a_1\})}\ar[uurrr]^{1_{u(\kern -1pt a_0\kern -1pt)}}
&&&u(a_1)\ar[dd]^{1_{u(\kern -1pt a_1\kern -1pt)}}
&&&&a_1\ar[dd]_{1_{a_1}}\ar[rrrr]^{S_1}
&&&&y
&
\\
\\
&&&&&&u(a_1)
&&&&a_1\ar[uurrrr]_{S_1}
&&&&.
}
$$
La commutativité du diagramme ($\star$) implique alors que
$$l(\undl{\alpha})=\bigl(w(S_1)\comp j(\alpha)\bigr)\bigl(w(S_0\subset S_1\cup\{a_0\})\comp j(h_0)\big)\,,$$
ce qui achève la démonstration.
\end{proof}
\end{paragr}

\begin{lemme}\label{imdirOcrible}
Soient $E\to F$ un crible d'ensembles ordonnés, $i:A\to B$ le crible de $\nCat{2}$ image de $E\to F$ par le foncteur $\Orntgfonct$, et
\[
\xymatrix{
A \ar[r]^u \ar[d]_i & A' \ar[d]^{i'}\\
B \ar[r]_v & B' \\
}
\]
un carré cocartésien de $\ncat2$. Si $A'$ est une catégorie enrichie en ensembles ordonnés, il en est de même de $B'$. 
\end{lemme}

\begin{proof}
Le lemme résulte aussitôt de la description d'une telle somme amalgamée $B'$ présentée au paragraphe précédent.
\end{proof}

\begin{prop}
Soit $A\to B$ une cofibration de $\ncat2$ pour la structure de catégorie de modèles à la Thomason. Si $A$ est une catégorie enrichie en ensembles ordonnés, il en est de même de $B$.
\end{prop}

\begin{proof}
En vertu du théorème~\ref{thm:cmf_2cat}, le $2$\nbd-foncteur $A\to B$ est un rétracte d'un composé transfini d'images directes de flèches appartenant à $c_2\Sd^2(I)$, où $I$ est l'ensemble de flèches de $\simpl$ défini dans le paragraphe~\ref{enssimpl}. Or, il résulte des lemmes~\ref{Sd2can} et~\ref{Dwyer} et de la proposition~\ref{orntgadj} que toute flèche appartenant à $c_2\Sd^2(I)$ est isomorphe à l'image par le foncteur $\Orntgfonct$ d'un crible d'ensembles ordonnés. La proposition se déduit alors par récurrence transfinie du lemme précédent, et de la stabilité des catégories enrichies en ensembles ordonnés par rétractes et limites inductives filtrantes dans~$\nCat{2}$. 
\end{proof}

\begin{cor}\label{cofcatenrord}
Soit $A$ un objet cofibrant de $\ncat2$ pour la structure de catégorie de modèles à la Thomason. Alors $A$ est une catégorie enrichie en ensembles ordonnés.
\end{cor}

\begin{proof}
Le corollaire est un cas particulier de la proposition précédente, appliquée à la cofibration $\varnothing\to A$.
\end{proof}

\appendix

\section{Le théorème de Smith et les $\W$-cofibrations}

Dans cet appendice, on présente quelques applications d'un théorème de
Jeffrey Smith faisant intervenir la notion de $\W$\nbd-cofibration de
Grothendieck. En particulier, on démontre une variante de la
proposition~\ref{sansSmith}. Pour énoncer le théorème de Smith, on a besoin
de quelques compléments concernant les catégories accessibles (\emph{cf.}~\ref{acc}).

\begin{paragr}\label{complacc}
La catégorie des foncteurs d'une petite catégorie vers une catégorie
accessible est accessible~\cite[théorème 2.39]{AR}. En particulier, la
catégorie des flèches d'une catégorie accessible est accessible. On dit
qu'un foncteur $F:\C\to\D$ est \emph{accessible} s'il existe un cardinal
régulier~$\kappa$ tel que les catégories $\C$ et $\D$ soient
$\kappa$\nbd-accessibles et tel que $F$ commute aux petites limites
inductives $\kappa$\nbd-filtrantes. Un foncteur entre catégories accessibles
admettant un adjoint à gauche ou à droite est accessible
\cite[proposition~2.23]{AR}. On dit qu'une classe d'objets d'une catégorie
accessible $\C$ est \emph{accessible} si la sous-catégorie pleine
correspondante est accessible \emph{et} si le foncteur d'inclusion est
accessible. L'image inverse par un foncteur accessible d'une classe
accessible d'objets d'une catégorie accessible est une classe accessible
d'objets~\cite[remarque 2.50]{AR}. Si on admet un axiome de grands cardinaux
connu sous le nom de principe de Vop\u enka, toute classe d'objets telle que
la sous-catégorie pleine correspondante soit accessible est une classe
accessible d'objets~\cite[théorème~6.9]{AR}. On dit qu'une classe de flèches
d'une catégorie accessible $\C$ est \emph{accessible} si elle l'est comme
classe d'objets de la catégorie des flèches de $\C$.
\end{paragr}

\begin{thm}[Smith] 
Soient $\C$ une catégorie localement présentable, $\W$ une classe accessible
de flèches de $\C$, et $I$ un (petit) ensemble de
flèches de $\C$. On note $\Cof{}$ la classe $lr(I)$. On suppose que
\begin{enumerate}[label=\emph{\roman*)}]
\item la classe de flèches $\W$ satisfait à la condition du deux sur trois;
\item on a l'inclusion $r(I)\subset\W$;
\item la classe $\Cof{}\cap\W$ est stable par images directes et composés transfinis.
\end{enumerate}
Alors il existe une structure de catégorie de modèles combinatoire sur $\C$,
ayant comme classe d'équivalences faibles $\W$, et comme classe de
cofibrations $\Cof{}$.
\end{thm}

\begin{proof}
En vertu de~\cite[corollaire 2.45]{AR}, et en tenant compte du fait qu'une
classe de flèches accessible est stable par rétractes, ce théorème résulte
de~\cite[théorème~1.7]{Be}.
\end{proof}

\begin{rem}\label{remWacc}
Réciproquement, on peut montrer que la classe des équivalences faibles d'une
catégorie de modèles combinatoire est accessible~(\cite[corollaire
A.2.6.6]{Lurie}, ou~\cite[théorème 4.1]{Ros}, ou~\cite[corollaire
0.2]{Raptis}).
\end{rem}

\begin{cor}
Soient $\C$ une catégorie localement présentable, $\W$ une classe accessible
de flèches de $\C$ satisfaisant à la propriété du deux sur trois et stable
par petites limites inductives filtrantes, et $I$ un (petit) ensemble de
$\W$-cofibrations \emph{(\emph{cf.}~\ref{defcofGroth})}.
On note $\Cof{}$ la classe $lr(I)$. Si $r(I)\subset\W$,
alors il existe une structure de catégorie de modèles combinatoire sur $\C$,
propre à gauche, ayant comme classe d'équivalences faibles $\W$, et comme
classe de cofibrations $\Cof{}$.
\end{cor}

\begin{proof}
La classe~$\W$ étant stable par petites limites inductives filtrantes et par
composition, elle l'est aussi par composition transfinie. Comme $\Cof{}$ est
également stable par composition transfinie, il en est de même de
$\Cof{}\cap\W$. En vertu de l'argument du petit objet, toute flèche~$f$ de
$\C$ se décompose en $f=pi$, avec $p\in r(I)\subset\W$ et $i\in \cell(I)$.
Or, par hypothèse $I\subset\Cof{\W}$, et vu que $\Cof{\W}$ est stable par
images directes~(\emph{cf.}~\ref{defcofGroth}) et par composés transfinis
(\emph{cf.}~proposition~\ref{prprcofGroth}), on a $\cell(I)\subset\Cof{\W}$.
Il résulte donc de la proposition~\ref{lemmeGroth} que la classe
$\Cof{\W}\cap\W$ est stable par images directes. La classe $\Cof{}=lr(I)$
étant aussi stable par images directes, il en est de même de
$\Cof{}\cap\Cof{\W}\cap\W$. Comme $I\subset\Cof{\W}$ et $\Cof{\W}$ est
stable par images directes, composés transfinis et rétractes
(\emph{cf.}~\ref{defcofGroth} et proposition~\ref{prprcofGroth}), l'argument
du petit objet implique que $\Cof{}\subset\Cof{\W}$. On a donc
$\Cof{}\cap\Cof{\W}\cap\W=\Cof{}\cap\W$, et les assertions du corollaire
autres que la propreté à gauche résultent du théorème de Smith. La propreté
à gauche résulte de l'inclusion $\Cof{}\subset\Cof{\W}$
(\emph{cf.}~exemple~\ref{excofGroth}).
\end{proof}

\begin{prop}\label{avecSmith}
Soient $\C$ une catégorie localement présentable, $\M$ une catégorie de modèles combinatoire, engendrée par $(I,J)$, dont les équivalences faibles sont stables par petites limites inductives filtrantes,
$$F:\M\toto\C\ ,\quad U:\C\toto\M$$
un couple de foncteurs adjoints tel que $U$ commute aux petites limites
inductives filtrantes, $\W$ la classe des $U$-équivalences de $\C$
\emph{(\emph{cf.}~\ref{UeqUcof})}, et $\Cof{}$ la classe $lr(F(I))$. On suppose que
\begin{enumerate}
\item $F(I)\subset\Cof{\W}$.
%\item $r(F(I))\subset\W$.
\end{enumerate}
Alors il existe une structure de catégorie de modèles combinatoire sur $\C$, propre à gauche, ayant comme classe d'équivalences faibles $\W$, et comme classe de cofibrations~$\Cof{}$. Si de plus on a
\begin{enumerate}
\item[\emph{(b)}] $F(J)\subset\W$,
\end{enumerate}
alors la structure de catégorie de modèles ainsi définie sur $\C$ est
engendrée par $(F(I),F(J))$, et ses fibrations sont les flèches de $\C$ dont
l'image par $U$ est une fibration de $\M$. En outre, cette structure est
propre si $\M$ est propre à droite.
\end{prop}

\begin{proof}
On observe d'abord que si $f\in r(F(I))$, on a par adjonction $U(f)\in r(I)$, autrement dit, $U(f)$ est une fibration triviale de $\M$, et en particulier une équivalence faible. On en déduit que $f$ est une $U$\nbd-équivalence. On a donc $r(F(I))\subset\W$.
\tb

D'autre part, comme la catégorie de modèles $\M$ est combinatoire, sa classe d'équivalences faibles est une classe de flèches de $\M$ accessible (\emph{cf.}~remarque~\ref{remWacc}), et comme le foncteur $U$ est un adjoint à droite, il est accessible~\cite[proposition 2.23]{AR}. On en déduit que $\W$ est une classe de flèches de $\C$ accessible~\cite[remarque 2.50 et théorème 2.39]{AR}. De plus, comme la classe des équivalences faibles de $\M$ est stable par petites limites inductives filtrantes, il en est de même de la classe $\W$ des $U$\nbd-équivalences, puisque le foncteur $U$ commute auxdites limites. Grâce à l'hypothèse (\emph{a}), la première assertion de la proposition résulte donc du corollaire précédent.
\tb

Montrons que si de plus $F(J)\subset\W$, alors $\Cof{}\cap\W=lr(F(J))$. En vertu de l'argument du petit objet et des propriétés de stabilité de $\Cof{}\cap\W$, pour montrer l'inclusion $lr(F(J))\subset\Cof{}\cap\W$, il suffit de montrer que $F(J)\subset\Cof{}\cap\W$. Comme les éléments de $J$ sont des cofibrations de $\M$, autrement dit, des rétractes de composés transfinis d'images directes d'éléments de $I$, et comme le foncteur $F$ commute à ces opérations, cela résulte des propriétés de stabilité de $\Cof{}$ et de l'hypothèse (\emph{b}). À ce stade on peut conclure la preuve de la proposition par le \og lemme de transfert\fg~\ref{transfert}. Néanmoins, il n'est pas difficile de terminer la preuve sans utiliser ce lemme: 
\tb

Pour montrer l'inclusion $\Cof{}\cap\W\subset lr(F(J))$, soit $i\in\Cof{}\cap\W$. En vertu de l'argument du petit objet, il existe une décomposition $i=qj$, avec $q\in r(F(J))$ et $j$ composé transfini d'images directes d'éléments de $F(J)$, et en particulier dans $lr(F(J))$. Par adjonction, $U(q)$ est dans $r(J)$, autrement dit, $U(q)$ est une fibration. Or, on a vu que $lr(F(J))\subset\Cof{}\cap\W$, donc $j$ est dans $\W$, et par deux sur trois $q$~aussi. On en déduit que $U(q)$ est une équivalence faible, donc une fibration triviale, autrement dit, $U(q)\in r(I)$. À nouveau par adjonction, on en déduit que $q\in r(F(I))$.
%que $q$ a la propriété de relèvement à droite relativement aux flèches appartenant à $F(I)$. 
Comme $i\in\Cof{}=lr(F(I))$, le lemme du rétracte (\cite[proposition 7.2.2]{Hir} ou~\cite[lemme~1.1.9]{Ho}) implique alors que $i$ est un rétracte de $j$, ce qui prouve que $i\in lr(F(J))$.
\tb

En vertu de ce qui précède, les fibrations de $\C$ sont les flèches qui ont la propriété de relèvement à droite relativement aux flèches appartenant à $F(J)$, et encore une fois par un argument d'adjonction, elles sont celles dont l'image par $U$ a la propriété de relèvement à droite relativement aux flèches appartenant à $J$, autrement dit, dont l'image par $U$ est une fibration de $\M$. Enfin, l'affirmation concernant la propreté à droite résulte aussitôt du fait que le foncteur $U$ respecte les fibrations et les carrés cartésiens et reflète les équivalences faibles.
\end{proof}

\begin{rem}
C'est surtout la première assertion de la proposition ci-dessus qui est intéressante. La deuxième n'est qu'une version affaiblie de la proposition~\ref{sansSmith}.
\end{rem}

\backmatter

\bibliography{nthom}
\bibliographystyle{smfplain}

\end{document}